\documentclass[a4paper,12pt,reqno]{amsart}

\usepackage[margin=2.5cm]{geometry}
\usepackage{amsthm,amsmath,amssymb}
\usepackage{dsfont}
\usepackage[english]{babel}
\usepackage[utf8]{inputenc}
\usepackage[T1]{fontenc}
\usepackage{mathtools}
\usepackage{bm}
\usepackage{stmaryrd}
\usepackage{xcolor} \definecolor{bluegreen2}{RGB}{0, 85, 127}
\usepackage[linktocpage]{hyperref}
\hypersetup{colorlinks=true, linkcolor=bluegreen2, citecolor=bluegreen2, filecolor=magenta, urlcolor=bluegreen2}
\usepackage{tikz, tikz-cd, pgfplots} \usetikzlibrary{spy,arrows,trees,shapes,calc,decorations.markings,decorations.pathmorphing,patterns}
\usepackage{enumitem} \setlist{itemsep=1.5mm}
\usepackage{newunicodechar}
\usepackage{multicol}
\usepackage{array}
\usepackage{multirow}

\usepackage{comment}

\theoremstyle{plain}
\newtheorem{theorem}{Theorem}[section]
\newtheorem{theoremintro}{Theorem}
\newtheorem{prop}[theorem]{Proposition}
\newtheorem{cor}[theorem]{Corollary}
\newtheorem{lemma}[theorem]{Lemma}

\theoremstyle{definition}
\newtheorem{defi}[theorem]{Definition}

\newtheorem{ex}[theorem]{Example}
\newtheorem{setting}[theorem]{Setting}

\theoremstyle{remark}
\newtheorem{remark}[theorem]{Remark}

\newcommand{\wt}{\widetilde}
\newcommand{\wh}{\widehat}

\newcommand{\ol}{\overline}

\newcommand{\C}{\mathcal{C}}

\newcommand{\K}{\mathcal{K}}
\renewcommand{\AA}{\mathds{A}}

\renewcommand{\P}{\mathcal{P}}
\newcommand{\M}{\mathcal{M}}
\newcommand{\R}{\mathcal{R}}
\newcommand{\T}{\mathcal{T}}

\newcommand{\CC}{\mathds{C}}

\newcommand{\ZZ}{\mathds{Z}}
\newcommand{\QQ}{\mathds{Q}}

\newcommand{\PP}{\mathds{P}}
\newcommand{\LL}{\mathds{L}}
\newcommand{\set}[1]{\left\{ #1 \right\}}

\newcommand{\KVarC}{K_0(\Var_\CC)}

\newcommand{\Zmott}[1]{Z_{\mot,#1}}

\newcommand{\Ztopp}[1]{Z_{\topo,#1}}

\newcommand{\ZHod}{{Z_{\Hodge}}}

\DeclareMathOperator{\Card}{Card}

\DeclareMathOperator{\Sing}{Sing}
\DeclareMathOperator{\Var}{Var}
\DeclareMathOperator{\GL}{GL}
\DeclareMathOperator{\id}{Id}
\DeclareMathOperator{\red}{red}

\DeclareMathOperator{\mot}{mot}

\DeclareMathOperator{\topo}{top}
\DeclareMathOperator{\Hodge}{Hod}

\DeclareMathOperator{\Ram}{Ram}
\DeclareMathOperator{\supp}{supp}
\DeclareMathOperator{\mult}{mult}
\DeclareMathOperator{\Diff}{Diff}

\title[Poles of zeta functions for finite morphisms between normal surfaces]{On the poles of zeta functions for finite morphisms between normal surfaces}

\author[E.~León-Cardenal]{Edwin León-Cardenal}
\address[E.~León-Cardenal]{Department of Mathematics, IUMA\newline\indent
	University of Zaragoza\newline\indent
	Calle Pedro Cerbuna 12\newline\indent
	50019, Zaragoza, Spain}
\urladdr{\url{http://riemann.unizar.es/~eleon}}
\email{\href{mailto:eleon@unizar.es}{eleon@unizar.es}}

\author[J.~Mart\'{\i}n-Morales]{Jorge Mart\'{\i}n-Morales}
\address[J.~Mart\'in-Morales]{Department of Mathematics, IUMA\newline\indent
	University of Zaragoza\newline\indent
	Calle Pedro Cerbuna 12\newline\indent
	50019, Zaragoza, Spain}
\urladdr{\url{https://riemann.unizar.es/~jorge}}
\email{\href{mailto:jorge.martin@unizar.es}{jorge.martin@unizar.es}}

\author[W.~Veys]{Willem Veys}
\address[W.~Veys]{
	University of Leuven (KU Leuven),\newline\indent
	Department of Mathematics,\newline\indent
	Celestijnenlaan 200B, B-3001\newline\indent
	Leuven (Heverlee), Belgium}
\urladdr{\url{https://perswww.kuleuven.be/wim\_veys}}
\email{\href{mailto:wim.veys@kuleuven.be}{wim.veys@kuleuven.be}}

\author[J.~Viu-Sos]{Juan Viu-Sos}
\address[J.~Viu-Sos]{
	Dpto. de Matemáticas e Informática (DMIAICN)\newline\indent
	ETSI Caminos, Canales y Puertos\newline\indent
	Universidad Politécnica de Madrid\newline\indent
	C\textbackslash Prof. Aranguren 3, 28040 Madrid, Spain
}
\urladdr{\url{https://jviusos.github.io/}}
\email{\href{mailto:juan.viu.sos@upm.es}{juan.viu.sos@upm.es}}

 % Modifications; resolution of singularities
\thanks{
	EL is supported by MICIU/AEI/10.13039/501100011033 grant PID2024-156181NB-C33 and by the European Union NextGenerationEU/PRTR and UNIZAR via María Zambrano's Program and by CONAHCYT project CF-2023-G33.
	JM is supported by the European Union NextGenerationEU/PRTR grant RYC2021-034300-I and by MICIU/AEI/10.13039/501100011033 grants CNS2024-154271 and PID2024-156181NB-C33.
	WV is supported by KU Leuven grant GYN-E4282-C16/23/010.
	JV is supported by MICIU/AEI/10.13039/501100011033 and ERDF/EU grant PID2024-156181NB-C32. %2026 grant
	EL and JM are also partially supported by Departamento de Ciencia, Universidad y Sociedad del Conocimiento del Gobierno de Aragón grant E22 20R Álgebra y Geometría. JM is also supported by Junta de Andalucía grant FQM-333.
	}

\begin{document}
	
\begin{abstract}
For a divisor representing a function and another divisor representing a differential form on a normal surface singularity,
there is a notion of motivic and topological zeta function.
In this paper, given a finite morphism between two normal surfaces,
we prove that the set of poles of the motivic zeta function associated with the target is contained in the one
associated with the source. We illustrate by examples that this inclusion is strict in general, and that on the topological level
there are in general no inclusions between the sets of poles on source and target.
On the other hand, when the morphism is the quotient map induced by an action 
of a finite abelian group on $\CC^2$, and the divisor associated with the differential form on the source is trivial,
we do show  equality between the corresponding sets of poles, both on motivic and topological level.  
In addition, again for the quotient map induced by an action of a finite abelian group on $\CC^2$, but now with a general divisor associated with a differential form,  we  provide a criterion when the topological zeta function on the target is just a multiple of the one on the source.
Finally, we compare  log canonical models on source and target with a view on zeta functions.
\end{abstract}
	
\maketitle
	
%\tableofcontents

	% =========================================================
	% Introduction
	% =========================================================
	\section*{Introduction}
	Motivic and topological zeta functions are important invariants in singularity theory that have been widely investigated.
	Originally defined for functions on smooth varieties,  they now admit a more general definition, including a divisor $D$
	resembling a function and another divisor $W$ representing a differential form, on a possibly singular variety $S$,
	see e.g.~\cite{DL99, Veys99, CNS18}. More recently, we have developed a theory of zeta functions over $\QQ$-Gorenstein
	varieties for $\QQ$-divisors using the so-called $\QQ$-Gorenstein measure~\cite{LCMMVVS:formulas}.

	Topological zeta functions are defined in Section~\ref{SubSec:Ztop};  in particular one observes that they are rational functions over $\QQ$ in a variable $s$, with possible poles in $\QQ$. 
We will denote the topological zeta function associated to $D$ and $W$ on a variety (germ) $(V,o)$ by $\Ztopp{(V,o)} (D,W; s)$.

Assume now that $G$ is a small finite subgroup of $\GL(n,\CC)$,
	and let $\pi: \CC^n \to \CC^n/G$ be the quotient morphism. Consider two normal crossing $\QQ$-divisors $\ol{D}$ and $\ol{W}$ in $\CC^n/G$, and denote by $D$ and $W$ their corresponding pullbacks in $\CC^n$. Then
	\begin{equation}\label{eq:ztop-equality}
	\Ztopp{(\CC^n/G,[0])}(\ol{D},\ol{W}; s) = |G| \cdot \Ztopp{(\CC^n,0)} (D,W; s),
	\end{equation}
	see \cite[Corollary 3.4]{LCMMVVS:formulas}. This equality is a specialization of the motivic relation
	found in~\cite[Theorem 3]{LCMMVVS:formulas}. Motivated by this observation, we study in this work  relations between motivic and topological
	zeta functions under finite morphisms. In general we cannot expect an equality of type \eqref{eq:ztop-equality} to be true.
	For a precise description of the main results we present in this paper, some notation needs to be introduced.
	
	Let $(S,o)$ be a normal algebraic surface germ. Take a nonzero effective $\QQ$-Weil divisor $D$ and a $\QQ$-Weil divisor $W$ on $(S,o)$.
	We say that $W$ \emph{has no logarithmic poles outside} $D$ if any component of $\supp(W)\setminus \supp(D)$ has associated multiplicity different from~$-1$ in $W$, see Setting~\ref{setting1}. %Recall that $\Ztopp{(S,o)}(D, W; s)$ is a rational function on the variable~$s$ \cite{DL92}.
	Let us denote by $\P^{\topo}_{(S,o)}(D,W) \subset \QQ$ the set of poles of $\Ztopp{(S,o)}(D, W; s)$.
	In the motivic case the notion of pole is not straightforward, since the Grothendieck ring is not a domain~\cite{Po02,Bor18}.
	However, there exists well-defined notions of poles and their multiplicities, see Remark~\ref{defpole} for the details.
	Analogously, the set of poles of $\Zmott{(S,o)}(D, W; s)$ is denoted by $\P^{\mot}_{(S,o)}(D,W) \subset \QQ$.
We note that, taking any embedded resolution $\pi$ of $\supp(W)\cup \supp(D)$, each irreducible component of $\pi^{-1}(\supp(W)\cup \supp(D))$ induces a candidate pole of both the associated topological and motivic zeta function, and that each pole is of this form, see Section \ref{Subsec:TopMotZFunctions}.
Also, we have that
	\begin{equation}\label{eq:top-mot}
	\P^{\topo}_{(S,o)}(D,W)
	\subset \P^{\mot}_{(S,o)}(D,W),
	\end{equation}
	since the topological zeta function is a specialization of the motivic one via the Euler characteristic \cite{DL98}.
	The other inclusion is not true in general as Example~\ref{ex:noPoleUpstairs} shows, see also~\cite[Example 2.10]{ACLM05} for a higher dimensional example based on~\cite[Example 4.5]{IshiiKollar03}.

	Before discussing the technical content, we already mention two important techniques that have been employed throughout the paper. First, we amply use $\QQ$-resolutions, which are partial toric  resolutions arising naturally from weighted blow-ups, see Section~\ref{sec:embQres}. Second, the so-called $\alpha$-values appear naturally in residues of candidate poles; the corresponding $\alpha$-relations and $\alpha$-conditions  have been successfully adapted to the setting of $\QQ$-resolutions, see Sections~\ref{subsec:NormalSurf} and \ref{sec:minQresInQuotients}.
	 	
	The main aim of this manuscript is to elucidate the relations between the sets of poles of zeta functions for divisors on normal surfaces, that are related by a finite morphism. First, we can prove a very general inclusion result about poles on the level of motivic zeta functions.
	
	\begin{theoremintro}\label{Thm:Comparison_Zmot}
	Let $\rho\colon (S, o)\to (\ol{S}, \ol{o})$ be a finite 
	morphism between normal surfaces $S$ and $\ol{S}$, with associated ramification divisor $\Ram_\rho$ on $S$. Fix two $\QQ$-Weil divisors $\ol{D}$ and $\ol{W}$ on $(\ol{S},\ol{o})$,  with $\ol{D}$ effective and nonzero, and $\ol{W}$ without logarithmic poles outside $\ol{D}$. Put $D = \rho^* \ol{D} $ and $ W =\rho^* \ol{W} + \Ram_\rho$ on $(S, o)$.
	Then
	$$
	\P^{\mot}_{(\ol{S},\ol{o})}(\ol{D},\ol{W}) \subset \P^{\mot}_{(S, o)}(D,W).
	$$
	Furthermore, $s_0\in\P^{\mot}_{(\ol{S},\ol{o})}(\ol{D},\ol{W})$ is a pole of order two if and only if $s_0\in\P^{\mot}_{(S, o)}(D,W)$ is of order two
	as well.
	\end{theoremintro}

	When $W=0$, poles of the motivic zeta function are completely determined in terms of the minimal embedded resolution by strict transforms, cycles of rational exceptional components, and exceptional components which are either non rational or rupture, see~\cite[Theorem 5.5]{RodriguesVeys03}.
	The strategy of the proof of Theorem~\ref{Thm:Comparison_Zmot} begins with the extension of this result for arbitrary $W$,
	see	Theorem~\ref{Thm:RodVeysQres}. In particular, this allows us to generalize in Definition~\ref{def:rupture-divisor} the concept of rupture component for any normal surface and $W \neq 0$. Then we fix an embedded resolution $\ol{\pi} \colon \ol{X}\to (\ol{S},\ol{o})$
	of $\ol{M} = \supp(\ol{D})\cup\supp(\ol{W})\cup\supp(B_\rho)$, where $B_\rho$ is the branch divisor on $\ol{S}$ of $\rho$,
	and we construct an embedded $\QQ$-resolution of $\rho^{-1}(\ol{M})$ making the following diagram commute.
	\begin{equation*}
		\begin{tikzcd}
			S \arrow[d, twoheadrightarrow, "\rho"]
			& \arrow[l, dashed, "\pi" above] Z \arrow[d, dashed, twoheadrightarrow, "\wt{\rho}"]\\
			\overline{S} & \arrow[l, "\ol{\pi}" above] \ol{X}
		\end{tikzcd}
	\end{equation*}
	The way one obtains $Z$ is inspired by Jung's construction, the details are given in Section~\ref{subsec:Zmot_inducedRes}.
	To finish the proof of Theorem~\ref{Thm:Comparison_Zmot}, we use Theorem~\ref{thm:MultAlphas_proportional} relating the
	numerical data of $\pi$ and $\ol{\pi}$, see Section~\ref{sec:proof-main-thm1}.
	
	The reverse inclusion $\P^{\mot}_{(\ol{S},\ol{o})}(\ol{D},\ol{W}) \supset \P^{\mot}_{(S, o)}(D,W)$ does not hold in general,
	as Example~\ref{ex:noPoleDownstairs} shows. One of the main reasons is that the condition for an exceptional component
	to be a rupture component is not preserved when pushing forward through finite morphisms.
	
	It is worth noticing that Theorem~\ref{Thm:Comparison_Zmot} does not provide an automatic similar relation for the topological
	zeta function, since in general the inclusion~\eqref{eq:top-mot} is strict.
	In fact, on the level of the topological zeta function, neither the relation $\P^{\topo}_{(\ol{S},\ol{o})}(\ol{D},\ol{W}) \subset \P^{\topo}_{(S, o)}(D,W)$ nor  $\P^{\topo}_{(\ol{S},\ol{o})}(\ol{D},\ol{W}) \supset \P^{\topo}_{(S, o)}(D,W)$ holds,
	as evidenced by Example~\ref{ex:noPoleUpstairs} and~\eqref{ex:ztop-C2-C2G} in Example~\ref{ex:noPoleDownstairs}, respectively.
	
	In the second part of the paper, we investigate a more concrete situation.
	In particular, we establish in the next result a stronger relation when the finite morphism is the quotient morphism
	$\rho\colon \CC^2 \to \CC^2/G$, induced by the action of a finite abelian group $G\subset\GL(2,\CC)$, and the divisor
	associated with the differential form in $\CC^2$ is trivial, i.e.,~$W = 0$.
	
	\begin{theoremintro}\label{Thm:Comparison_Ztop}
	Let $G\subset\GL(2,\CC)$ be a finite abelian group and consider the natural covering $\rho: \CC^2\to\CC^2/G$, with associated branch divisor $B_\rho\subset \CC^2/G$. 
	Fix nonzero effective $\QQ$-divisors $D$ on $\CC^2$ and $\ol{D}$ on $\CC^2/G$ such that $D = \rho^* \ol{D}$.
	Then
	\begin{equation}\label{eq:equality-top}
	\P^{\mot}_{(\CC^2/G,[0])}(\ol{D},-B_\rho) =
	\P^{\topo}_{(\CC^2/G,[0])}(\ol{D},-B_\rho) \stackrel{(\theequation)}{=} \P^{\topo}_{(\CC^2, 0)}(D,0)
	= \P^{\mot}_{(\CC^2, 0)}(D,0). \nonumber \refstepcounter{equation}
	\end{equation}
	Moreover, the order of any pole $s_0$ is the same in any of the zeta functions.
	\end{theoremintro}
	
	In order to prove this result, we first start with the minimal embedded resolution $\pi$ of $(\CC^2,D)$, and then we construct an
	embedded $\QQ$-resolution $\ol{\pi}$ of $(\CC^2/G,\ol{D})$, so that the diagram
	\begin{equation*}
		\begin{tikzcd}
			\CC^2 \arrow[d, twoheadrightarrow, "\rho"] 	& \arrow[l, "\pi" above] X \arrow[d, dashed, twoheadrightarrow, "\wt{\rho}"]\\
			\CC^2/G & \arrow[l, dashed, "\ol{\pi}" above] X/G =:\ol{X}
		\end{tikzcd}
	\end{equation*}
	commutes, and this allows us to apply Theorem~\ref{thm:MultAlphas_proportional}.  It turns out that $\ol{\pi}$ verifies the so-called $\alpha$-condition, see Definition~\ref{def:alpha_condition}, which is based on the concept of the $\alpha$-values, given 
	in Definition~\ref{Def:Alpha_Qres}. These are numbers that appear naturally when writing down the residue of a candidate pole. With a careful study
	of these  residues, % corresponding to each divisor using the $\alpha$-values,
 we show Theorem~\ref{thm:VeysPolesZtop}, which 
	generalizes~\cite[Theorem~4.2~and~4.3]{Veys95}, leading to equality~\eqref{eq:equality-top}, see Section~\ref{subsec:ProofPolesZtop}. In the same section we also show that, for $S = \CC^2/G$ and $W = -B_\rho$, with $\rho: \CC^2\to\CC^2/G$ the natural covering,
	expression~\eqref{eq:top-mot} is in fact an equality. This
	way the chain of equalities in the statement follows. Notice that, when $W\neq 0$, the morphism $\ol{\pi}$ does not verify the $\alpha$-condition, as Example~\ref{ex:noPoleUpstairs} shows. This is one of the motivations to study Theorem~\ref{Thm:Comparison_Ztop} only for $W=0$.
	
	The proofs of these two main results presented in this work are essentially different, since
	Theorem~\ref{Thm:Comparison_Zmot} starts with the resolution $\ol{\pi}$ downstairs, while
	for Theorem~\ref{Thm:Comparison_Ztop} we first need the resolution $\pi$ upstairs.
	However, they both use the same strategy in spirit, namely one completes a diagram and
	applies Theorem~\ref{thm:MultAlphas_proportional} to relate the numerical data involved
	in the construction.
			
	Going back to our original motivation, we show that equality~\eqref{eq:ztop-equality}
	is true when the finite morphism is the quotient map $\rho: \CC^2 \to \CC^2/G$, and moreover every analytically
	irreducible component of both $D$ and $W$ is $G$-invariant.
	
	\begin{theoremintro}\label{Thm:dZtop}
	Let $G\subset\GL(2,\CC)$ be a finite abelian group of order $d$, and consider the natural covering map $\rho: \CC^2\to\CC^2/G$. 
	Fix two $\QQ$-Weil divisors $\ol{D}$ and $\ol{W}$ on $(\CC^2/G,[0])$,  with $\ol{D}$ effective and nonzero, and $\ol{W}$ without logarithmic poles outside $\ol{D}$. Put $D = \rho^* \ol{D} $ and $W =\rho^* \ol{W} + \Ram_\rho$ on $(\CC^2,0)$. %
	If every analytically irreducible component of both $D$ and $W$ is $G$-invariant, then
	$$\Ztopp{(\CC^2/G,[0])}(\ol{D},\ol{W}; s) = d\cdot \Ztopp{(\CC^2,0)}(D,W; s).$$
	\end{theoremintro}
	
	\smallskip

	 The example developed in Section~\ref{subsec:orbit-pathology} shows that the condition of each irreducible component being $G$-invariant is necessary, even in the case $W=0$.
	 
	 \medskip
		
	The structure of the paper is as follows. In Section~\ref{sec:settings} some preliminaries and settings about
	embedded $\QQ$-resolutions and motivic and topological zeta functions are presented. We also include the statement of
	Theorem~\ref{Thm:RodVeysQres}, whose proof is postponed until the end of the manuscript in order not to lose the main focus of the paper. To finish this section, we develop  relations between the $\alpha$-values and the genus of the exceptional components, in the context of  a $\QQ$-resolution.
	Section~\ref{sec:RamCovers} is devoted to finite morphisms, where we establish in Theorem~\ref{thm:MultAlphas_proportional}
	some relations between the numerical data of the corresponding resolutions. 
	The two main results of this work, namely Theorems \ref{Thm:Comparison_Zmot} and \ref{Thm:Comparison_Ztop},
	are proven in Sections~\ref{sec:Zmot} and~\ref{sec:Ztop}, respectively. A proof of Theorem~\ref{Thm:dZtop} is provided at the end of Section~\ref{Sec:Proportionality}. We take the opportunity to compare the log canonical models of divisors on normal surfaces that are connected by a finite morphism, and the impact on zeta functions, in Section~\ref{sec:log_canonical}.
	Finally, we prove in Section~\ref{sec:RodVeysThm} the mentioned generalization to the case $W\neq 0$ of a theorem about residues, namely Theorem~\ref{Thm:RodVeysExt}.

	% =========================================================
	% Preliminaries and settings
	% =========================================================
	\bigskip
	\section{Zeta functions of surface quotient singularities}\label{sec:settings}
	This section introduces the framework for the remainder of the paper. First, we collect some basic facts about quotient singularities and $\QQ$-resolutions. In Section~\ref{Subsec:TopMotZFunctions}, we
	give the definitions of the motivic and the topological zeta functions and present some formulas in terms of $\QQ$-resolutions. After Definition~\ref{Def:Alpha_Qres} we formulate Theorem~\ref{Thm:RodVeysQres}, characterizing the poles of the motivic zeta functions in terms of  
	such resolutions. Finally, in Section~\ref{subsec:NormalSurf} we use a generalized adjunction formula for Weil divisors on normal surfaces to present useful arithmetical relations between numerical data of exceptional divisors, adapted to the context of $\QQ$-resolutions. 
	
	\subsection{Preliminaries on quotient singularities and embedded \texorpdfstring{$\QQ$}{QQ}-resolutions}\label{sec:embQres}
	
	A $V$-\emph{manifold} (also called an \emph{orbifold}) of dimension $2$ is a complex analytic surface $X$, 
	admitting an open covering~$\{U_i\}_i$, where each $U_i$ is analytically isomorphic to a quotient singularity, 
	see e.g.~\cite{Satake}. In our case we will be dealing with quotients $\CC^2/\mu_m$ by a finite cyclic subgroup  
	$\mu_m$ of $\GL_2(\CC)$ with $m$ elements, where  $\zeta\in\mu_m$ acts as
	\[
	\zeta\cdot (x,y) = (\zeta^a x, \zeta^b y),
	\]
	for some $a,b\in\ZZ$. %
	We denote this type of singularity by $X(m;a,b)$. If $\gcd(m,a,b)=1$, then the action is faithful, it is also free if $\gcd(m,a)=\gcd(m,b)=1$. If these last conditions are verified, we say that the action of $\mu_{m}$ is \emph{small}. Note that being small is equivalent to the natural covering $\CC^2\to X(m;a,b)$ being unramified. In general, the ramification divisor of the covering is $(e_1-1)L_1 + (e_2-1)L_2$, where $e_1=\gcd(m,b)$ and $e_2=\gcd(m,a)$, and $L_1$, $L_2$ are the  divisors defined by $\set{x=0}$ and $\set{y=0}$, respectively.
	
	Some isomorphisms between quotient singularities are given below. Note, however, that the isomorphisms involved do not necessarily respect  the orbifold structures, see e.g.~\cite[Lemma 1.2]{AMO14A}.
	\begin{enumerate}[label=(\alph*)]
	  \item  $X(m;a,b)= X(km;ka,kb)$ for any positive integer $k$;
	  \item  $X(m;a,b)= X(m;ka,kb)$ for any positive integer $k$ satisfying $\gcd(m,k)=1$;
	  \item  $X(m;a,b) \cong X(m;b,a)$;
	  \item  $X(m;a,b) \cong X(m/e;a,b/e)$, where $e=\gcd(m,b)$. %via the map $[(x,y)]\mapsto [(x,y^e)]$.
	\end{enumerate}
	The last two isomorphisms are given by $[(x,y)]\mapsto [(y,x)]$ and $[(x^e,y)]\mapsto [(x,y)]$, respectively. By using these isomorphisms and group operations, it is possible to restrict $X(m;a,b)$ to a singularity of type $X(m;1,q)$ with $1\leq q\leq m-1$ and $\gcd(m,q)=1$, also known as an \emph{$A_{m,q}$-singularity} or \emph{Hirzebruch–Jung singularity}, see e.g. \cite{HNK71}.
	We say that a point $P$ in a $V$-manifold $X$ \emph{has order $m$} if $X$ is locally an $A_{m,q}$-singularity at $P$.
	
	\begin{remark}\label{rmk:Anq}
		Several characterizations of $A_{m,q}$-singularities can be found in e.g.~\cite[Section~2.3]{Nemethi22:book}. In particular, they can be obtained as the normalization of the singular surface $X_{m,q}:=\set{uv^{m-q}=w^m}\subset\CC^3$ via the map
	  \begin{equation*}
	    \begin{array}{ccc}
	      X(m;1,q) & \longrightarrow & X_{m,q}\\
	      ~[(x,y)] & \longmapsto & (x^m,y^m,xy^{m-q}).
	    \end{array}
	  \end{equation*}
	\end{remark}
	
	\begin{remark}\label{rmk:AllAbelianAreCyclic}
	  Any quotient space $\CC^2/G$ by a finite abelian group $G\subset\GL_2(\CC)$ is isomorphic to some $X(d;a,b)$, where $d\mid |G|$, see~\cite[Lemmas~1.1\&1.2]{AMO14A}. Moreover, $d$, $a$ and $b$ can be chosen such that $\gcd(d,a,b)=1$, and moreover also such that $\gcd(d,a)=\gcd(d,b)=1$ if we make use of isomorphisms of type $[(x^e,y)]\mapsto [(x,y)]$ as in (d) above. Note that, if $\gcd(d,a,b)=1$ and $e_1:=\gcd(d,b)$, $e_2:=\gcd(d,a)$, then
	  \[
	    \begin{array}{ccc}
	      X(d; a, b) & \longrightarrow &  X\left(\frac{d}{e_1 e_2}; \frac{a}{e_2}, \frac{b}{e_2}\right)\\[.5em] 
	      
	      [(x,y)]	& \longmapsto & \left[\left(x^{e_1},y^{e_2}\right)\right]
	    \end{array},
	  \]
	  is an isomorphism.
	\end{remark}
	
	Using quotient singularities, one can generalize the notion of embedded resolution of singularities given by a finite composition of blow-ups. Following~\cite{Steenbrink77}, a divisor $D$ on a surface $X$ has
	\emph{$\QQ$-normal crossing} if it is analytically isomorphic to a germ given by
	\begin{equation*}
	  x^{k_1}y^{k_2} \colon X(m;a,b) \to \CC.
	\end{equation*}
	It should be stressed that the quotient space $X(m;a,b)$ carries a particular choice of coordinates coming from a diagonal action on $\CC^2$. 

\begin{defi}\label{defi:Qres}
		An \emph{embedded $\QQ$-resolution} of a $\QQ$-Weil divisor $D$ on a normal
surface germ $(S,o)$ is a proper bimeromorphic map $\pi:X\to (S,o)$ such that
\begin{enumerate}
	\item $X$ is a $V$-manifold with abelian quotient singularities,
	\item the restriction of $\pi$ to $X\setminus \pi^{-1}(D)$ is an isomorphism,
	\item $\pi^{-1}(D)$ is a $\QQ$-normal crossing Weil divisor  on $X$.
\end{enumerate}
Moreover, $\pi$ is called \emph{minimal} if it does not factorize through another embedded $\QQ$-resolution of $(S,D)$. Since minimal usual embedded resolutions exist for surfaces, minimal embedded $\QQ$-resolutions also exist in this context.
\end{defi}

Recall that on a $V$-manifold with cyclic quotient singularities, $\QQ$-Weil divisors and $\QQ$-Cartier divisors coincide, see e.g.~\cite{Ful93,AMO14A}.

\medskip
	
A useful tool to find an embedded $\QQ$-resolution is a weighted blow-up. Set $(p,q)\in\ZZ_{\geq1}^2$ and consider the weighted projective line~$\PP^1_{(p,q)}$, i.e., the quotient of $\CC^2\setminus 0$ by the $\CC^*$-action $(x,y)\sim (t^p x,t^q y)$.  The $(p,q)$-\emph{blow up} of $X(m;a, b)$ at the origin is the proper birational map $\pi_{(p,q)}: X\to X(m;a,b)$, where
	\[
	  X = \left\{ \left((x,y),[u:v] \right) \in \CC^2 \times \PP^1_{(p,q)} \mid (x,y) = (t^{p}u, t^{q}v), \ t \in \CC \right\} \Big/ \mu_m,
	\]
	and $\pi_{(p,q)}$ maps $\left((x,y),[u:v] \right)$ to $[(x,y)]$. Here the action of $\mu_m$ is extended to $\PP^1_{(p,q)}$ as $\zeta \cdot [u:v] := [\zeta^a u : \zeta^b v]$.
	The surface $X$ is covered by two open sets $X = U \cup V$, where $U = \{ u\neq 0 \}$
	and $V = \{ v \neq 0 \}$. The orbifold charts are given by
	\begin{equation}
	\begin{aligned}\label{eq:chartC2G}
	& \varphi: X\left(\begin{array}{c|cc}
	p & -1 & q \\
	pm & a & pb-qa
	\end{array}\right) =: X_1
	\longrightarrow U,
	\quad [(x,y)] \mapsto ((x^p,x^qy),[1:y]), \\
	& \psi: X\left(\begin{array}{c|cc}
	q & p & -1 \\
	qm & qa-pb & b
	\end{array}\right) =: X_2
	\longrightarrow V,
	\quad [(x,y)] \mapsto ((xy^p,y^q),[x:1]).
	\end{aligned}
	\end{equation}
	Here $X\left(\begin{array}{c|cc}
	m_1 & a_1 & b_1 \\
	m_2 & a_2 & b_2
	\end{array}\right)$ denotes the linear action of $\mu_{m_1}\times\mu_{m_2}$ on $\CC^2$ given by
	\[
		(\zeta_1,\zeta_2)\cdot (x,y) := (\zeta_1^{a_1}\zeta_2^{b_1}x, \zeta_1^{a_2}\zeta_2^{b_2}y).
	\]
	Note that the exceptional divisor $E_{(p,q)} = \pi_{(p,q)}^{-1}(0)$ is identified with $\PP_{(p,q)}^1/\mu_m$. Also, $\PP_{(p,q)}^1\cong \PP^1\cong \PP^1/\mu_m$, so all of them are smooth varieties. For more details about this construction, see e.g.~\cite{AMO14B}.

	\subsection{Topological and motivic zeta functions} \label{Subsec:TopMotZFunctions}
	
	We present in this section two related types of zeta functions associated to a pair $(D,W)$ on $(S,o)$: the topological and motivic zeta functions. 
	Our formulation here is inspired by~\cite{Veys07} and~\cite{NemethiVeys12}.

\begin{setting}\label{setting1}
	Let $(S,o)$ be a normal algebraic surface germ. Take a nonzero effective $\QQ$-Weil divisor $D$ on $(S,o)$, and a $\QQ$-Weil divisor $W$ on $(S,o)$, such that any component  of $\supp(W)\setminus \supp(D)$ has associated multiplicity different from $-1$ in $W$. We say that {\em $W$ has no logarithmic poles outside $D$}.
\end{setting}
	We set $M:=\supp(D)\cup \supp(W)$ and fix a usual embedded resolution $\pi\colon X\to S$ of $M$ in $(S,o)$; note that this implies that both $\pi^*D$ and $\pi^*W$ are normal crossing $\QQ$-divisors on~$X$.
	Let $\{E_i\}_{i\in I}$ be the irreducible components (exceptional components and strict transforms) of~$\pi^{-1} (M)$.  
	It is convenient to use the partition $I=I_e\cup I_s$ to distinguish the components of the exceptional divisor and those of the strict transform of~$\pi^{-1}(M)$, respectively. It is also customary to denote by $E_i^\circ$, for $i\in I_e$, the set $E_i \setminus \bigcup_{j\in I, j\neq i} E_j$. 
	
In the sequel, we denote the canonical divisor (class) of a normal variety by $K_V$. The {\em relative canonical divisor} of $\pi$, denoted $K_\pi$, is the unique representative of $K_X-\pi^*K_S$ that is supported on its exceptional locus.  Writing 
\begin{equation}\label{eqn:canonical1}
K_{\pi} = \sum_{i \in I_e} (b_i-1)E_i,
\end{equation}
 the $\set{b_{i}}_{i\in I_e}$ are called {\em log discrepancies} of $\pi$.
Denote further 
	\[N_i=\mult_{E_i}\pi^\ast D \qquad\text{and}\qquad \nu_i=b_i+ \mult_{E_i}\pi^\ast W.
	\]
	Then we have  the following relations:
	\begin{equation}\label{eqn:canonical2}
		\begin{aligned}
			\pi^{*} W &= \widehat{W} + \sum_{i \in I_e} (\nu_i - b_i) E_i, \\
			\pi^{*} D &= \widehat{D} + \sum_{i \in I_e} N_i E_i,
		\end{aligned}
	\end{equation}
	where $\wh{D}$ and $\wh{W}$ are the strict transform of $D$ and $W$. We will call the pair $(N_i,\nu_i)$ the \emph{numerical data} of $E_i$,  $i\in I$, associated with $(D,W)$. For the sake of brevity, we will say that $\pi$ is a resolution of the pair $(D,W)$. Note that the numbers $\nu_{i}$ are in general rational, and can be positive, negative or zero. The numbers $N_i$ are non negative rational numbers, which are integers if $D$ is a usual effective divisor.
	
	\subsubsection{Topological zeta function}\label{SubSec:Ztop}
	
	We use the symbol $\chi(\cdot)$ to denote the topological Euler characteristic. For $(S,o)$ and $(D,W)$ as in Setting~\ref{setting1}, the \emph{local topological zeta function} of $(D,W)$ on $(S,o)$ is defined by
	\begin{equation}\label{Def:TopZeta}
		\Ztopp{(S,o)}(D,W; s):= \sum_{i\in I_e} \frac{\chi(E_i^\circ)}{\nu_{i}+N_{i}s}+\sum_{\{i,j\}\subseteq I} \frac{\chi(E_i\cap E_j)}{(\nu_{i}+N_{i}s)(\nu_{j}+N_{j}s)}.
	\end{equation}
	A standard calculation shows that this expression is invariant under a blow-up, and hence it does not depend on the chosen embedded resolution $\pi$. 
	Notice that the curves $E_i$, $i\in I_e$, appearing in this definition can have arbitrary genus $g_i$, and in particular one has the following identity:
	\begin{equation}\label{Eq:CharOpenStratum}
	\chi(E_i^\circ)=2-2g_i-\Card\left(E_i\cap \left(\cup_{j\neq i}E_j\right) \right).
	\end{equation}
	It is also clear that $\chi(E_i\cap E_j)=\Card(E_i\cap E_j).$

	\subsubsection{Motivic zeta function}\label{SubSec:Zmot}
	
	For the definition of the motivic zeta function, we first introduce briefly the Grothendieck ring of varieties over $\CC$. Denote by $\KVarC$ the Grothendieck ring of algebraic varieties over $\CC$, that is, the free abelian group generated by the symbols $[V]$, where $V$ is a variety, subject to the relations $[V]=[V']$ if $V \cong V'$, and $[V] = [V \setminus V'] + [V']$ if $V'$ is Zariski closed in $V$. Its ring structure is given by $[V] \cdot [V'] := [V \times V']$.  It is customary to denote $\LL:=[\AA^{1}]$.

Let  $(S,o)$ and $(D,W)$ be as in Setting~\ref{setting1}. The \emph{local motivic zeta function} of the pair $(D,W)$ on $(S,o)$ is given 
in terms of the resolution $\pi$ by the following formula~\cite{RodriguesVeys03,Veys07,NemethiVeys12}:
	\begin{equation}
		\label{Eq:ZmotFormula}
		\begin{aligned}
			\Zmott{(S,o)}(D,W;s) = \LL^{-2}&\sum_{i\in I_e} [E_i^\circ]\frac{\LL-1}{\LL^{\nu_i+N_i s}-1}\\ &+ \LL^{-2} \sum_{\set{i,j}\subset I} [E_i\cap E_j]\frac{(\LL-1)^2}{(\LL^{\nu_i+N_i s}-1)(\LL^{\nu_j+N_j s}-1)}.
		\end{aligned}
	\end{equation}
Here $\LL^{-s}$ should be considered as a formal variable $T$. This convention comes from the analogy with the $p$-adic Igusa zeta function, which is a rational function in $p^{-s}$~ \cite{Igu74}. Then the expression above is the traditional way to write
$$
\LL^{-2}\sum_{i\in I_e} [E_i^\circ]\frac{(\LL-1)T^{N_i}}{\LL^{\nu_i} - T^{N_i}}\\ + \LL^{-2}\sum_{\set{i,j}\subset I} [E_i\cap E_j]\frac{(\LL-1)^2 T^{N_i+N_j}}{(\LL^{\nu_i} - T^{N_i})(\LL^{\nu_j}-T^{N_j})}.
$$
%Denoting $r$ for the common denominator of all occurring $N_i$ and $\nu_i$, 
This last expression lives in the localization  $\mathcal{A}_T$ of a polynomial ring $\KVarC[\LL^{-1/r}][T^{1/r}]$ (for some appropriate divisible enough $r\in \ZZ_{>0}$),  with respect to the multiplicative system generated by the elements $\LL^{a/r}-T^{b/r}$, for $a\in \ZZ$ and $b\in \ZZ_{\geq 0}$ with $(a,b)\neq (0,0)$; for more precise details see~\cite[Section~4]{RodriguesVeys03}.

Again, classical arguments can be used to verify that the defining expression~\eqref{Eq:ZmotFormula} does not depend on the chosen embedded resolution.

\begin{remark}
	In some cases, these expressions come from an intrinsically defined converging \emph{motivic integral}. 
Namely, in the context of motivic integration, one can associate to $(D,W)$ an integral over the arc space of $S$ (and also more generally within an extension of our setting to  arbitrary dimension),  see e.g.~\cite{DL99,LCMMVVS:formulas,CNS18}.  %,Veys01}.
In case this integral converges, it is a formal power series in $T$ over a certain completion of $\KVarC[\LL^{-1/r}]$, and one proves an expression~\eqref{Eq:ZmotFormula} for it in terms of a resolution.  But for arbitrary $W$, that integral often does not converge.
\end{remark}

\begin{remark}\label{defpole}
Since $\KVarC$ has zero divisors~\cite{Po02,Bor18}, it is not clear a priori what would be the definition of a pole  of $\Zmott{(S,o)}(D,W; s)$. By using the localization $\mathcal{A}_T$ referred above, given $s_0 \in \QQ$, it is possible to define  when $\LL^{-s_0}$ is a pole of such a motivic zeta function, together with its pole order, mimicking as much as possible the obvious notions of pole and pole order for rational functions over a domain.
In particular, if $\LL^{-s_0}$ is a candidate pole of order $1$, there is a natural notion of residue, as we will see below, and $\LL^{-s_0}$ is a pole if and only if its residue is nonzero.
We will say in short that \emph{$s_0$ is a pole of $\Zmott{(S,o)}(D,W; s)$} if $\LL^{-s_0}$ is. A thorough treatment of those constructions is given in ~\cite[Section 4]{RodriguesVeys03}.
\end{remark}

In practice, many of the factors $\LL^{\nu_i+N_is}-1$ in the expression~\eqref{Eq:ZmotFormula} cancel in the final sum. 
In higher dimension it is a very difficult open problem to determine exactly the set of poles of   motivic and topological zeta functions. The general philosophy behind this quest is that the poles of these zeta functions are induced by certain \lq relevant\rq\ components $E_i$, appearing in every resolution.
When $W=0$, the poles of $\Zmott{(S,o)}(D,W; s)$ are completely determined by strict transforms, cycles of rational exceptional components, and exceptional components which are either non rational or of rupture type, see~\cite[Theorem 3.4]{RodriguesVeys03}. A generalization to arbitrary $W$ in the context of $\QQ$-resolutions will be presented in Theorem~\ref{Thm:RodVeysQres} below.

A prominent role in this type of results is played by the following numbers, which are central in the understanding of the poles for both motivic and topological zeta functions, see also~\cite{Loeser88,Veys95,Veys99,RodriguesVeys03}. 
Fixing an exceptional component $E_i$ with numerical data $(N_i,\nu_i)$ of the embedded resolution $\pi$, the value
$$
\alpha_j:=\nu_j-(\nu_i/N_i)N_j,
$$
associated to the intersection with another component $E_j$, is called the \emph{$\alpha$-value} of $E_j$ with respect to $E_i$. 
In fact, suppose that $E_i$ intersects exactly $k$ times other components, say $E_1,\ldots, E_k$, with numerical data $(N_j,\nu_j)$  for $j\in\{1,\ldots,k\}$. If 
$\nu_j/N_j\neq \nu_i/N_i$ for each $j$, then the contribution of ${E}_i$ to the residue of 
$s_0=-\nu_i/N_i$ for the topological zeta function~\eqref{Def:TopZeta} is given by
\begin{equation*}
	\R_i^e:=\frac{1}{N_i}\left(\chi({E}_i^\circ) + \sum_{j=1}^k\frac{1}{{\alpha}_{j}}\right).
\end{equation*}

Similarly, from~\cite[Section~5.4]{RodriguesVeys03}, the contribution of $E_i$ to the residue of  $s_0$ for the motivic zeta function~\eqref{Eq:ZmotFormula} is 
\begin{equation}\label{Eq:ResExc}
	\mathcal{R}_i^e:=\frac{1-\LL}{N_i\LL^{-\nu_i/N_i}}\left([E_i^\circ]+\sum_{j=1}^k \frac{\LL-1}{\LL^{\alpha_j}-1}\right),
\end{equation}
considered in a ring $\mathcal{A}$, that is the localization of $\KVarC[\LL^{-1/r}]$  with respect to the multiplicative system generated by the elements $1-\LL^{a/r}$ and $b$, for $a,b\in \ZZ\setminus \{0\}$.

Likewise, when $E_i$ is a component of the strict transform of $\pi$, intersecting the (unique) exceptional component $E_j$, its contribution to the residue of $s_0$ is
\begin{equation}\label{Eq:ResStri}
		\mathcal{R}_i^s:= \frac{1-\LL}{N_i\LL^{-\nu_i/N_i}}\cdot\frac{\LL-1}{\LL^{\alpha_j}-1} \in \mathcal{A}.
	\end{equation}

In any case, in order to determine which components of $\pi$ contribute to a pole, it is useful to study relations between the $\alpha$-values and also their range. 
A noteworthy property in this direction, for plane curves and when $W=0$, is the following.
\begin{prop}[{\cite[Proposition~II.3.1]{Loeser88}}]\label{lemma:alphas1}
	Assume that $\pi\colon X \to \CC^2$ is the minimal embedded resolution of an effective divisor $D$. 
	Let $E$ be an exceptional component of $\pi$, and take another component $E_j$, either of the exceptional divisor or of $\wh{D}$, intersecting $E$. Denote by $(N,\nu)$ and $(N_j,\nu_j)$ the numerical data of $E$ and $E_j$, respectively. Then $\alpha_j:=\nu_j-(\nu/N)N_j\in~[-1,1)$. 
\end{prop}

With this result at hand, one may characterize the contributing components to the poles of, for instance, the motivic zeta function.

\begin{prop}\label{RVmot}
	Let $s_0\in \QQ$.
	\begin{enumerate}
		\item Then $s_0$ is a pole of order $2$ of $\Zmott{(S,o)}(D,W; s)$ if and only if there exist two intersecting components  $E_i$ and $E_j$ with $s_0 = -\nu_i/N_i = - \nu_j/N_j$. 
		
		\item Let $E_1, \dots, E_m$ be precisely the components with $s_0 = -\nu_i/N_i$, and assume that $E_i\cap E_{i'}= \emptyset$ for $i\neq i'$. Let $\mathcal{R}_1,\dots,\mathcal{R}_m$ be their residue contributions to $s_0$, respectively, where the $\mathcal{R}_i$ are the expressions in (\ref{Eq:ResExc}) or (\ref{Eq:ResStri}), depending on $E_i$ being an exceptional or strict transform component.
		Then $s_0$ is a pole of order $1$ of $\Zmott{(S,o)}(D,W; s)$ if and only if $\mathcal{R}_1 +\dots + \mathcal{R}_m \neq 0$ (in the ring $\mathcal{A}$).
	\end{enumerate}
\end{prop}
The original formulation of Proposition~\ref{RVmot} is given in~\cite[Propositions 5.3 and 5.4.1]{RodriguesVeys03}, when $W=0$. A straightforward computation is enough to extend the original proof to the case $W\neq 0$. 

Towards the understanding of the poles by means of embedded $\QQ$-resolutions, we present in Definition~\ref{Def:Alpha_Qres} an analogue of the $\alpha$-values for singular points on surfaces. Later on, in Section~\ref{sec:minQresInQuotients}, we will see some conditions that allow us to extend Proposition~\ref{lemma:alphas1} to singular quotient spaces, see~Proposition~\ref{Prop:AlphasAbs1}. 
	
	\subsubsection{Zeta functions in terms of $\QQ$-resolutions}\label{SubSec:ZQres}
	
	We elaborated in~\cite{LCMMVVS:formulas} a theory of zeta functions for $\QQ$-divisors on $\QQ$-Gorenstein
	varieties, using embedded $\QQ$-resolutions. 
	We deduce now some consequences for surface quotient singularities. Suppose that $(S,o)\cong(X(m;a,b),[0])$, where the action is
	small, and assume that $M=\supp(D)\cup\supp(W)$ has $\QQ$-normal crossing. 
	Moreover, assume that locally at $(X(m;a,b),[0])$ one can express $D$ and $W$ as
	\[
	D\colon x^{N_1}y^{N_2} = 0\quad\text{ and }\quad W\colon x^{\nu_1-1}y^{\nu_2-1} = 0,
	\]
where $W$ has no logarithmic poles outside $D$.
	Using this notation,~\cite[Theorem~3]{LCMMVVS:formulas} gives the following explicit formula for the motivic zeta function:
	\[
	\Zmott{(X(m;a,b),[0])}(D,W; s) = \LL^{-2} \mathcal{S}_{m,a,b}(\bm{N},\bm{\nu})\frac{(\LL-1)^2}{(\LL^{\nu_1+N_1 s}-1)(\LL^{\nu_2+N_2 s}-1)},
	\]
	with factor $$\mathcal{S}_{m,a,b}(\bm{N},\bm{\nu}) := \sum_{i=0}^{m-1} 	
	\LL^{ \frac{1}{m} \big(\nu_1[{ia}]_m + \nu_2 [{ib}]_m   + \left( N_1[{ia}]_m  + N_2[{ib}]_m \right) s\big)}.$$ Here $[{k}]_m$ denotes the class of $k$ modulo $m$ satisfying $0 \leq [{k}]_m \leq m-1$. Note that this factor does not induce any candidate pole, and also that the specialization map satisfies  $\chi(\mathcal{S}_{m,a,b}(\bm{N},\bm{\nu}))=m$.

	For $(S,o)$ and $(D,W)$ as in Setting~\ref{setting1}, take an embedded $\QQ$-resolution
	$\pi \colon X\to (S,o)$ of $M:=\supp(D)\cup\supp(W)$ with associated irreducible components $\set{E_i}_{i\in I}$. Consider the finite set of points
	\begin{equation}\label{Eq:P_pi}
		\P_\pi := (\Sing X) %\cap \pi^{-1}(o))
	\cup \bigcup _{\set{i,j}\subset I} (E_i\cap E_j),
	\end{equation}
	and set $E_i^\circ = E_i\setminus\P_\pi$. For each $P_k\in\P_\pi$, we have that $(X,P_k)$  is analytically isomorphic to some $(X(m_k;a_k,b_k),[0])$, where the action is small. Denote the local data for $D$ and $W$ at $P_k$  by  $\bm{N}_k=(N_{1,k},N_{2,k})$ and  $\bm{\nu}_k=(\nu_{1,k},\nu_{2,k})$.  
	Then~\cite[Theorem~4]{LCMMVVS:formulas} says that $\Zmott{(S,o)}(D,W;s)$ equals
	\begin{multline*}\label{Eq:ZmotFormulaQres}
		\LL^{-2} \sum_{i\in I_e} [E_i^\circ]\frac{\LL-1}{\LL^{\nu_i+N_i s}-1} + \LL^{-2} \sum_{P_k\in \P_\pi} \mathcal{S}_{m_k,a_k,b_k}(\bm{N}_k,\bm{\nu}_k)\frac{(\LL-1)^2}{(\LL^{\nu_{1,k}+N_{1,k} s}-1)(\LL^{\nu_{2,k}+N_{2,k} s}-1)}.
	\end{multline*}
	For the topological zeta function, one gets the simpler formula
		\begin{equation}\label{Eqn:ZtopQres}
			\Ztopp{(S,o)}(D,W; s)= \sum_{i\in I_e} \frac{\chi(E_i^\circ)}{\nu_{i}+N_{i}s}+ \sum_{P_k\in\P_\pi} \frac{m_k}{(\nu_{1,k}+N_{1,k}s)(\nu_{2,k}+N_{2,k}s)}.
		\end{equation}
	
	Furthermore, the hypothesis that the action on $P_k$ is small is not necessary to apply the former formula for the topological zeta function.
	
	\begin{lemma}\label{lemma:Ztop_nonsmall}
	  Let $\ol{D}\colon x^{N_1}y^{N_2} = 0$ and $\ol{W}\colon x^{\nu_1-1}y^{\nu_2-1} = 0$ be two $\QQ$-divisors on $\CC^2/G$, where %$N_i, \nu_i-1\in\QQ_{\geq0}$ 
$\ol{W}$ has no logarithmic poles outside $\ol{D}$, 
and $G\subset\GL_2(\CC)$ is finite abelian, acting diagonally. Then
	  \[
	    \Ztopp{(\CC^2/G,[0])}(\ol{D},\ol{W}; s) = \frac{|G|}{(\nu_1+N_1 s)(\nu_2 + N_2 s)}.
	  \]
	\end{lemma}
	
	\begin{proof}
	  It is known that $(\CC^2/G,[0])\cong(X(m;a,b),[0])$,  where the action by $\mu_m$ is small, via group operations and an isomorphism $[(x,y)]\mapsto \left[\left(x^{\ell_1},y^{\ell_2}\right)\right]$, see Remark~\ref{rmk:AllAbelianAreCyclic}. It is clear that $m=\frac{|G|}{\ell_1\cdot\ell_2}$. Also, $(\frac{N_1}{\ell_i},\frac{\nu_1}{\ell_i})$ and $(\frac{N_2}{\ell_j},\frac{\nu_2}{\ell_j})$ are local numerical data for $\ol{D}$ and $\ol{W}$ in $(X(m;a,b),[0])$, for some choice of $\set{i,j}=\set{1,2}$. Applying~\eqref{Eqn:ZtopQres}, we obtain
	  \begin{align*}
	    \Ztopp{(\CC^2/G,[0])}(\ol{D},\ol{W}; s)
			&= \frac{m}{\left(\frac{\nu_1}{\ell_i} + \frac{N_1}{\ell_i}s\right)\left(\frac{\nu_2}{\ell_j} + \frac{N_2}{\ell_j}s\right)}
			= \frac{m\ell_1\ell_2}{(\nu_1+N_1 s)(\nu_2 + N_2 s)}\\
			&= \frac{|G|}{(\nu_1+N_1 s)(\nu_2 + N_2 s)}.
	  \end{align*}
	\end{proof}
	
	As a consequence, formula~\eqref{Eqn:ZtopQres} still holds for non-small actions. 
	Looking at that formula, it is easy to compute the contribution of $E_i$, $i\in I_e$, to the residue of $-\nu_i/N_i$. In order to provide a compact formula we introduce the following generalization of the $\alpha$-values.
	\begin{defi}\label{Def:Alpha_Qres}
	Consider a fixed irreducible component $E_i$ with numerical data $(N_i,\nu_i)$, $N_i\neq 0$,
	of the $\QQ$-resolution $\pi$ of the pair $(D,W)$. 
		For each point $P\in \P_\pi\cap E_i$ of order $m$, its $\alpha$-value with respect to $E_i$ is given by
		\begin{equation*}
	    \alpha_P = \begin{cases}
	                 \frac{\nu_j-(\nu_i/N_i)N_j}{m} & \text{if $P\in E_i\cap E_j$ for some $j\in I$},\\
	                 \frac{1}{m} & \text{otherwise} .
	               \end{cases}
	  \end{equation*}
	Strictly speaking, this definition depends on $D,W,\pi,E_i$ and $P$. Since we usually fix an exceptional divisor $E_i$ of a $\QQ$-resolution $\pi$ of $(D,W)$, to simplify notation we only write~$\alpha_P$.
	\end{defi}

The motivation for choosing $\frac{1}{m}$ in the second part of the definition comes from the fact that in this case $P\in\Sing X$, but it does not belong to any other $E_j$ with $j\in I$. One associates to this point the numerical data $(0,1)$, picturing an imaginary curvette passing through $P$ transversely, see Figure~\ref{fig:PpicapEj}.   
In the picture 
\[
\P_\pi \cap E_i =  \bigcup _{j=1}^{k} (E_i\cap E_j) \cup \{ P_1,\ldots,P_s \},
\]
and we denote by $m_j$ (resp.~$m_{P_j}$) the order of $E_i \cap E_j$ (resp.~$P_j$) as a cyclic quotient singularity.
	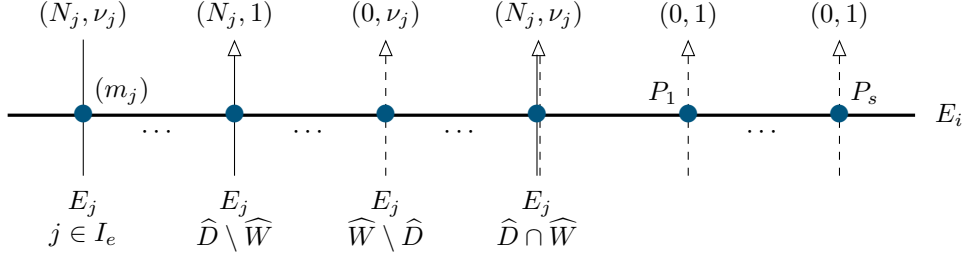
\begin{figure}[ht]
	\centering
	\definecolor{bluegreen2}{RGB}{0, 85, 127}
	\begin{tikzpicture}[scale=1, font=\footnotesize]
	\colorlet{colsing}{bluegreen2}
	\colorlet{colsingtext}{colsing}
	\draw[very thick] (-5,0) -- (7,0) node [right, xshift=3pt] {$E_i$};
	\draw (-4, -0.8) node[below] {$\begin{array}{c} E_j \\ j \in I_e \end{array}$} -- (-4,1)
	node[above]{$(N_j,\nu_j)$};
	\draw [-open triangle 45] (-2, -0.8) node[below]
	{$\begin{array}{c} E_j \\ \widehat{D} \setminus \widehat{W} \end{array}$} -- (-2,1)
	node[above]{$(N_j,1)$};
	\draw[-open triangle 45, dashed] (0, -0.8) node[below]
	{$\begin{array}{c} E_j \\ \widehat{W} \setminus \widehat{D}\end{array}$} -- (0,1)
	node[above]{$(0,\nu_j)$};
	\draw [-open triangle 45] (2.02, 1) node[above]{$(N_j,\nu_j)$};
	\draw (2, -0.8) node[below]
	{$\begin{array}{c} E_j \\ \widehat{D} \cap \widehat{W} \end{array}$} -- (2,0.8)
	;
	\draw[-open triangle 45, dashed] (4, -0.8) node[below]{}
	-- (4,1)node[above]{$(0,1)$};
	\draw[-open triangle 45, dashed] (6, -0.8) node[below]{}
	-- (6,1)node[above]{$(0,1)$};
	\draw[dashed] (2.04, -0.8) -- (2.04,0.89);
	\node[color=colsing] at (6,0) {\Large $\bullet$};
	\node[color=colsing] at (4,0) {\Large $\bullet$};
	\node[color=colsing] at (2,0) {\Large $\bullet$};
	\node[color=colsing] at (0,0) {\Large $\bullet$};
	\node[color=colsing] at (-2,0) {\Large $\bullet$};
	\node[color=colsing] at (-4,0) {\Large $\bullet$};
	\draw (6,0) node[above right] {$P_s$};
	\draw (4,0) node[above left] {$P_1$};
	\draw (5,0) node[below] {$\cdots$};
	\draw (-3,0) node[below] {$\cdots$};
	\draw (-1,0) node[below] {$\cdots$};
	\draw (1,0) node[below] {$\cdots$};
	\draw (-4,0) node[above right] {$(m_j)$};
	\end{tikzpicture}
	\caption{Description of the set $\P_\pi \cap E_i$.}
	\label{fig:PpicapEj}
	\end{figure}
	
	\noindent If no ambiguity arises, we denote for simplicity by $\alpha_j$ the $\alpha$-value with respect to $E_i$ at the point $E_i\cap E_j$.
	We have more concretely the following cases:
	\begin{equation}\label{eq:alpha-Ej}
		\alpha_j =
		\begin{cases}
			\frac{\nu_j - \frac{\nu_i}{N_i} N_j}{m_j} & \text{ for } i \in I_e \text{ or } E_j \text{ in } \widehat{D} \cap \widehat{W}, \\[7.5pt]
			\frac{1 - \frac{\nu_i}{N_i} N_j}{m_j} & \text{ for } E_j \text{ in } \widehat{D} \setminus \widehat{W}, \\[7.5pt]
			\frac{\nu_j}{m_j} & \text{ for } E_j \text{ in } \widehat{W} \setminus \widehat{D},
		\end{cases} \qquad\qquad
		\alpha_{P_j} = \frac{1}{m_{P_j}}.
	\end{equation}
	With this notation and following Definition~\ref{Def:Alpha_Qres}, the contribution of an exceptional component ${E}_i$ to the residue of $-\nu_i/N_i$ for the topological zeta function~\eqref{Eqn:ZtopQres}
	is given by
	\begin{equation}\label{Eq:Residue}
		\R_i^e:=\frac{1}{N_i}\left(\chi({E}_i^\circ) + \sum_{P\in \P_\pi\cap E_i}\frac{1}{{\alpha}_{P}}\right).
	\end{equation}
	
	Moreover, the notation recently discussed for the $\alpha$-values allows us to present the fo\-llo\-wing characterization of the poles of 
	$\Zmott{(S,o)}(D,W; s)$ in terms of $\QQ$-resolutions. 
	\begin{theorem}\label{Thm:RodVeysQres}
		Let $(S,o)$ and $(D,W)$ be as in Setting~\ref{setting1}. Fix an embedded $\QQ$-resolution $\pi \colon X\to (S,o)$ of $M=\supp(D)\cup \supp(W)$ and let $s_0\in \QQ$. We have that $s_0$ is a pole of $\Zmott{(S,o)}(D,W; s)$ if and only if
		\begin{enumerate}[label=(\roman*)]
			\item\label{ResPart1} $s_0=-\nu_i/N_i$ for some irreducible component $E_i$ in the set $\wh{D}\cap\wh{W}$ or in the set $\wh{D}\setminus\wh{W}$ (in which case $\nu_i=1$), or
			\item\label{ResPart2} $s_0=-\nu_i/N_i$ for some non-rational exceptional curve $E_i$, or
			\item\label{ResPart3} $s_0=-\nu_i/N_i$ for a cycle of rational exceptional curves $E_i$, or
			\item\label{ResPart4} $s_0=-\nu_i/N_i$ for some rational exceptional curve $E_i$ containing at least three points $P\in\P_\pi$ having $\alpha_{P}\neq 1$.
		\end{enumerate}
	\end{theorem}
	The main ideas of the proof of Theorem~\ref{Thm:RodVeysQres} are inspired by~\cite[Theorem 3.4]{RodriguesVeys03}, although the generalization is not  straightforward. To maintain the focus of the main goal of the paper, we postpone its proof  until Section~\ref{sec:RodVeysThm}. 
	
	In the plane curve case, it is customary to call an exceptional component, intersecting at least three times other components of the total transform, a \emph{rupture component}. These components are very important in singularity theory, in particular they characterize the exceptional components contributing to poles of local zeta functions when $W=0$~\cite{Veys95}. Condition \ref{RVPart4} in Theorem~\ref{Thm:RodVeysQres} serves to justify the extension of the name to a singular ambient surface and general $W$.
		\begin{defi}\label{def:rupture-divisor}
			Let $E$ be an irreducible component of an embedded $\QQ$-resolution $\pi \colon X\to (S,o)$ of $M=\supp(D)\cup \supp(W)$. If $E$ is a rational exceptional curve containing at least three points $P\in\P_\pi$ with $\alpha_{P}\neq 1$, then $E$ is called a \emph{rupture component}.
	\end{defi}
	
	\subsection{\texorpdfstring{$\QQ$}{QQ}-resolutions and relations between \texorpdfstring{$\alpha$}{α}-values}\label{subsec:NormalSurf}

	Let $(S,o)$ and $(D,W)$ be as in Setting~\ref{setting1}. 	
	Fix an embedded $\QQ$-resolution $\pi\colon X\to S$ of $M=\supp(D)\cup \supp(W)$, and let $\{E_i\}_{i\in I}$ be the irreducible components of $\pi^{-1} (M)$ with numerical data $\set{(N_i,\nu_i)}_{i\in I}$. Recall that the relative canonical divisor of $\pi$ is expressed, in terms of the log discrepancies $\set{b_{i}}_{i\in I_e}$, as $K_{\pi} = \sum_{i \in I_e} (b_i-1)E_i$. The next result provides arithmetic relations between the $\alpha$-values
	that become crucial when studying poles of local zeta functions; it generalizes the 
	classical case to $\QQ$-resolutions. We give the details of the proof for completeness; below we could treat all components $E_j$ simultaneously, resulting in a more compact proof, but we think it is useful to exhibit the differences between exceptional components and various sorts of strict transforms.

\begin{lemma}\label{Lemma:Arith-Rel}
		Denoting by $g_i$ the genus of the exceptional component $E_i$, we have that
		\[\sum_{P\in E_i} (\alpha_P-1) = 2g_i-2.
		\]
		Alternatively, by using the definition of $\P_{\pi}$
		in~\eqref{Eq:P_pi}, we may take  
		$P \in \P_{\pi} \cap E_i$ in the previous sum, since 
		$\alpha_P = 1$ outside 
		$\P_{\pi} \cap E_i$.
	\end{lemma}
\begin{proof}
We fix an index $i\in I_e$. The identities~\eqref{eqn:canonical1} and~\eqref{eqn:canonical2} lead to
\begin{equation}\label{Eq:DefAlpha}
\begin{aligned}
	K_X &=\pi^\ast K_S - \pi^\ast W +\wh{W} -\sum_{j\in I_e}(\nu_j-1)E_j + \frac{\nu_i}{N_i}\left(\pi^\ast D - \wh{D}+\sum_{j\in I_e}N_j E_j\right) \\
&=\pi^\ast\left(K_S -  W + \frac{\nu_i}{N_i}D\right) +\sum_{j\in I_e}\left(\nu_j - \frac{\nu_i}{N_i}N_j -1 \right)E_j +\wh{W}-\frac{\nu_i}{N_i} \wh{D} .
\end{aligned}
\end{equation}
and then the intersection of $K_X+E_i$ with $E_i$ is
\begin{equation}\label{Eq:Intersection_E_i}
	(K_X+E_i)\cdot E_i=
	\sum_{\substack{j\in I_e \\ j\neq i }}\left( \nu_j - \frac{\nu_i}{N_i}N_j -1 \right)  E_j\cdot E_i
	+ \wh{W}\cdot E_i
	- \frac{\nu_i}{N_i}\wh{D}\cdot E_i.	
\end{equation}
With the aim of developing~\eqref{Eq:Intersection_E_i} conveniently, we first introduce some notation.
Assume $E_i$ intersects $k$ times other components (not necessarily exceptional) of $\pi^{-1}(M)$, say $E_1,\ldots,E_{k}$.  In fact, we could have $E_j=E_{j'}$, for $j\neq j'$, when $E_j$ intersects several times $E_i$,
but we keep this convention for convenience. Also $E_i$ may contain singular points of $X$, say $P_1,\ldots,P_s$,
that do not belong to any other divisor $E_j$. This way the set $\P_\pi \cap E_i$ from~\eqref{Eq:P_pi} becomes
\[
\P_\pi \cap E_i =  \bigcup _{j=1}^{k} (E_j\cap E_i) \cup \{ P_1,\ldots,P_s \},
\]
see Figure~\ref{fig:PpicapEj}. The corresponding $\alpha$-values are the ones from~\eqref{eq:alpha-Ej}.
	 
	In order to compute the left-hand side of equation~\eqref{Eq:Intersection_E_i}, namely the intersection $(K_X+E_i)\cdot E_i$,
	we use a generalized adjunction formula for Weil divisors on normal surfaces
	that requires a correction term, called the \emph{different}, 
	\[
	\left.(K_X + E_i)\right|_{E_i} = K_{E_i} + \Diff(E_i),
	\]
	see e.g.~\cite[Section~2.2]{Kollar13} or~\cite[Chapter~16]{Kol92}. Recall that the degree of the canonical divisor of $E_i$ is $2g_i-2$. Since $X$ is a $V$-manifold, the pair $(X,E_i)$ is log canonical with only finitely many
	cyclic quotient singularities, and the different at $E_i$ is given by
	\[
	\Diff(E_i) = \sum_{P \in E_i} \left(1-\frac{1}{m_P}\right)[P],
	\]
	see e.g.~\cite[Proposition~16.6]{Kol92}.
	Note that, since $m_P = 1$ except for finitely many points in $E_i$, the previous sum is finite.
%	The term coming from $\Diff(E_i)$
This sum can be decomposed into
	\[
	\sum_{P \in E_i} \left(1-\frac{1}{m_P}\right)[P]
	= \sum_{j = 1}^{k} \left( 1 - \frac{1}{m_j} \right)[E_j\cap E_i]
	+ \sum_{j = 1}^s \left( 1 - \frac{1}{m_{P_j}} \right)[P_j],
	\]
	where $\frac{1}{m_{P_j}}$ is precisely the $\alpha$-value at $P_j$, see~\eqref{eq:alpha-Ej}.
The degree of the left-hand side of~\eqref{Eq:Intersection_E_i} is thus
	\begin{equation}\label{eq:left-hand-side}
%	(K_X+E_i) \cdot E_i = 
2g_i - 2
	+ \sum_{j = 1}^{k} \left( 1 - \frac{1}{m_j} \right)
	+ \sum_{j = 1}^s \left( 1 - \alpha_{P_j} \right).
	\end{equation}
	
Using the $\alpha$-values from~\eqref{eq:alpha-Ej}, and the fact that the degree of $E_i \cdot E_j$ is $\frac{1}{m_j}$ (\cite{Ful98,AMO14B}), we can compute the degree of the right-hand side of~\eqref{Eq:Intersection_E_i} as
	\begin{equation}\label{eq:right-hand-side}
	\begin{aligned}
	& 	 \sum_{\substack{j = 1 \\ j \in I_e}}^{k} \frac{\nu_j - \frac{\nu_i}{N_i} N_j-1}{m_j}
+ \sum_{\substack{j = 1 \\ E_j \text{ in } \wh{W}}}^{k} \frac{\nu_j-1}{m_j} 
	- \sum_{\substack{j = 1 \\ E_j \text{ in } \wh{D}}}^{k} \frac{\nu_i}{N_i} \frac{N_j}{m_j} \\
\stackrel{(*)}{=}	& \sum_{\substack{j = 1 \\ j \in I_e}}^{k} \left( \alpha_{j} - \frac{1}{m_j} \right)
+ \sum_{\substack{j = 1 \\ E_j \text{ in } \wh{D}\cup\wh{W}}}^{k}
	\left( \frac{\nu_j-\frac{\nu_i}{N_i} N_j}{m_j} - \frac{1}{m_j} \right) \\
=& \sum_{\substack{j = 1 \\ j \in I_e}}^{k} \left( \alpha_{j} - \frac{1}{m_j} \right)
	+ \sum_{\substack{j = 1 \\ E_j \text{ in } \wh{D}\cup\wh{W}}}^{k}
	\left( \alpha_{j} - \frac{1}{m_j} \right).
	\end{aligned}
	\end{equation}
Here equality~$(*)$ deserves some attention. Observe that
	$\widehat{D} \cup \widehat{W} = (\widehat{D} \setminus \widehat{W}) \sqcup (\widehat{W} \setminus \widehat{D})
	\sqcup (\widehat{D} \cap \widehat{W})$, while $W = (\widehat{W} \setminus \widehat{D}) \sqcup (\widehat{D} \cap \widehat{W})$
	and
	$D = (\widehat{D} \setminus \widehat{W}) \sqcup (\widehat{D} \cap \widehat{W})$.
	Then~$(*)$ follows by computing the contributions of the corresponding terms to each one of the previous sets $\widehat{D} \setminus \widehat{W}$,
	$\widehat{W} \setminus \widehat{D}$, and $\widehat{D} \cap \widehat{W}$.
	Joining~\eqref{eq:left-hand-side} and~\eqref{eq:right-hand-side}, we thus derived the following identity from for~\eqref{Eq:Intersection_E_i}:
	\[
	2g_i - 2
	= \sum_{j=1}^{k} \left( \alpha_{j} - 1 \right)
	+ \sum_{j=1}^{s} \left( \alpha_{P_j} - 1 \right),
	\]
	which can be more compactly written as
	\begin{equation}\label{Eq:AlphaArithm}
	\sum_{P \in \P_{\pi}\cap E_i} (\alpha_P-1) = 2g_i-2.
    \end{equation}
    Since $\alpha_P = 1$ outside $\mathcal{P}_\pi \cap E_i$, we can rewrite equation~\eqref{Eq:AlphaArithm} as
    the more conceptual relation
    \begin{equation}\label{Eq:AlphaArithmP}
    	\sum_{P\in E_i} (\alpha_P-1) = 2g_i-2.
    \end{equation}
\end{proof}

	% =========================================================
	% Finite morphisms of surfaces and $\QQ$-resolutions
	% =========================================================
	\section{Finite morphisms of surfaces and \texorpdfstring{$\QQ$}{QQ}-resolutions}\label{sec:RamCovers}
	
We briefly recall some notions related to ramification.
Let $\varphi \colon V \to \ol{V}$ be a finite morphism of degree $d$ between normal varieties; it induces a well-defined pullback $\varphi^*$ of Weil divisors (see e.g.~\cite[Section~4.2]{Har77},~\cite[Section~2.3]{Kollar13}).
So in particular, for any prime divisor $\ol{D}$ on $\ol{V}$, we can express $\varphi^*{\ol{D}}$ as $\sum_{i=1}^m  r_i D_i$, where $\varphi^{-1}{\ol{D}}=\cup_{i=1}^m  D_i$ and all $r_i \in \ZZ_{\geq 1}$. The number $r_i$ is the {\em ramification index of $D_i$ along $\varphi$}. If $r_i > 1$, we say that {\em $\varphi$ is ramified at $D_i$} and that $\varphi(D_i)$ belongs to the {\em branch locus} of $\varphi$. If $r_i = 1$, we say that {\em $\varphi$ is unramified at $D_i$}.
The {\em ramification divisor of $\varphi$} is
$$
\Ram_\varphi :=\sum_i (r_i-1) D_i,
$$
where $i$ runs over the (finitely many) divisors $D_i$ on $V$ along which $\varphi$ is ramified. The {\em branch divisor $B_\varphi$} of $\varphi$ is the $\QQ$-divisor on $\ol{V}$ whose pullback is
$\Ram_\varphi$, that is, $B_\varphi= \sum_i (1- \frac1{r_i}) \varphi(D_i)$.  We have Hurwitz's formula:
\begin{equation}\label{eqn:Ram}
	K_V \sim_{\QQ} \varphi^*K_{\ol{V}} + \Ram_\varphi.
\end{equation}
When both $V$ and $\ol{V}$ are smooth curves, $\Ram_\varphi$ is supported on a finite set of points $\set{P_i}_i$. Taking degrees, and denoting by $g(\cdot)$ the genus, one recovers the well-known relation between ramified coverings of curves:
\[
  2g(V)-2 = d\cdot (2g(\ol{V})-2) + \sum_{i} (r_i-1).
\]
	
	\subsection{Numerical relations between compatible \texorpdfstring{$\QQ$}{QQ}-resolutions}\label{ssec:Numerical_Relations}
	
	In order to prove relations between poles of zeta functions for finite coverings of normal surfaces, we study common $\QQ$-resolutions where the finite map restricts to \emph{nice} coverings between exceptional divisors. This leads to numerical relations between multiplicities, log discrepancies and eventually residues, see e.g.~\cite[Section~2.42]{Kollar13} for related results. In fact, the relations below can essentially be derived from loc.~cit., that has been also the source of the ideas in the proof. We think, though, that it is useful to give here a more self-contained  presentation, adding the context of $\QQ$-resolutions. 
	Our main objects of study are finite maps $\rho\colon (S, o)\to (\ol{S}, \ol{o})$ between (germs) of normal surfaces $S$ and $\ol{S}$. In particular, we assume that $\rho^{-1}(\ol{o}) = \set{o}$.

	\begin{theorem}\label{thm:MultAlphas_proportional}
	  Let $\rho\colon (S, o)\to (\ol{S}, \ol{o})$ be a finite morphism between normal surfaces $S$ and $\ol{S}$, with associated ramification divisor $\Ram_\rho$ on $S$ and branch divisor $B_\rho$ on $\ol{S}$. Let $D$ and $W$ (resp.~$\ol{D}$ and $\ol{W}$) be Weil divisors on $(S, o)$ (resp.~$(\ol{S},\ol{o})$) such that $D = \rho^* \ol{D} $ and $ W =\rho^* \ol{W} + \Ram_\rho$.
	  Fix an embedded $\QQ$-resolution $\pi\colon X\to S$ of $M=\supp (D)\cup\supp(W)\cup\supp(\Ram_\rho)$ (resp.~$\ol{\pi}\colon \ol{X}\to \ol{S}$ of $\ol{M}=\supp (\ol{D})\cup\supp(\ol{W})\cup\supp(B_\rho)$),
	 assuming that there exists a finite morphism $\wt{\rho}\colon X\to \ol{X}$ making the following diagram commutative.
	  \begin{center}
		\begin{tikzcd}
			S \arrow[d, twoheadrightarrow, "\rho"]
			& \arrow[l, "\pi" above] X \arrow[d, twoheadrightarrow, "\wt{\rho}"]\\
			\overline{S} & \arrow[l, "\ol{\pi}" above] \ol{X}
		\end{tikzcd}
	  \end{center}
	  Then $\deg(\rho)=\deg(\wt{\rho})$ and the following holds. 
	  \begin{enumerate}
	    \item Let $E$ (resp.~$\ol{E}$) be an irreducible component of ${\pi}^{-1}({M})$ (resp.~$\ol{\pi}^{-1}(\ol{M})$), such that $\wt{\rho}(E)=\ol{E}$. Let $(N,\nu)$ and $(\ol{N},\ol{\nu})$ be the respective numerical data. Then $(N,\nu)=e\cdot(\ol{N},\ol{\nu})$, where $e$ is the ramification index of $E$ along $\wt{\rho}$.
	
	    \item Moreover, when $E$ is an exceptional divisor,  let $P$ be a point of $E$ with $\wt{\rho}(P)=Q\in\ol{E}$. If  
	    $\alpha_P$  is the $\alpha$-value with respect to $E$, and $\ol{\alpha}_Q$ is the $\alpha$-value of $Q$ with respect to $\ol{E}$, then $\alpha_P=n\ol{\alpha}_Q$, where $n$ is the ramification index of $P$ along $\left.\wt{\rho}\,\right|_{E}$.
	  \end{enumerate}
	\end{theorem}
	
	\begin{proof}
		Recall that the degree of a finite morphism is the degree of the corresponding field extension. Hence it is clear that $\deg(\wt{\rho}) = [K(X):\K(\ol{X})]=[K(S):\K(\ol{S})]=\deg(\rho)$. 
		
		We denote by $\{\ol{E}_i\}_{i\in I}$ the set of irreducible components of $\ol{\pi}^{-1}(\ol{M})$, and by 
		$\{{E}_j\}_{j\in J}$ the set of irreducible components of ${\pi}^{-1}({M})$. 
		
(1)		Since $\wt{\rho}$ is finite, we have for any $\ol{E}_i$ that
		\begin{equation}\label{Eqn:Ei_upstairs}
			\wt{\rho}^* \ol{E}_i = \sum_{j=1}^{m_i} e_j^{(i)} E_j^{(i)},\quad \text{with $e_j^{(i)}\geq 1$ \ for  $j=1,\ldots, m_i$}.
		\end{equation}
		By definition of the numerical data, we have that  
			\begin{equation*}\label{Eqn:Ni_upstairs}
				\pi^* D = \sum_{j\in J}N_j E_j\quad \text{and} \quad \ol{\pi}^*\ol{D} = \sum_{i\in I} \ol{N}_i \ol{E}_i.
			\end{equation*}
		By commutativity of the diagram, and using~\eqref{Eqn:Ei_upstairs}, it follows that
		\[
		\sum_{j\in J}N_j E_j=\pi^*D = \pi^*\rho^*(\ol{D})=
		\wt{\rho}^*(\ol{\pi}^*\ol{D}) =  \sum_{i\in I} \sum_{j=1}^{m_i} \ol{N_i} e_j^{(i)} E_j^{(i)}.
		\]
		Since the $E_j$ are irreducible, we conclude that $N_j=\ol{N}_i e_j^{(i)}$. 
		
		A similar argument, using now Hurwitz's formula and the fact that $W = \rho^*\ol{W} + \Ram_\rho$, yields the analogous relation between the $\nu$-numbers.

\smallskip
(2)
Given the finite map $\varrho:=\wt{\rho}\,|_{E} :E\to \ol{E}$, 
	  and the inclusions $\gamma\colon E \hookrightarrow X$ and $\beta\colon \ol{E}\hookrightarrow\ol{X}$, the following is also a commutative diagram.
	  \begin{center}
		\begin{tikzcd}
			X \arrow[d, twoheadrightarrow, "\wt{\rho}"]
			\arrow[r, hookleftarrow, "\gamma" above, swap] & E \arrow[d, twoheadrightarrow, "\varrho"]\\
			\overline{X} \arrow[r, hookleftarrow, "\beta" above] & \ol{E}
		\end{tikzcd}
	  \end{center}
	  We now gather the finite amount of points $\varrho( \P_{\pi}\cap E) \cup (\P_{\ol{\pi}}\cap \ol{E})$ in the set $\P_o$,  using the index set $I_o$. We associate a divisor $\ol{E}_i$ to any $Q_i$, $i\in I_o$, as follows:
	  \begin{itemize}
	    \item $\ol{E}_i$ is the unique irreducible component (of the exceptional divisor or the strict transform) such that $Q_i:=\ol{E}\cap \ol{E}_i$, if it exists, or
	    \item $\ol{E}_i$ is a curvette intersecting $\ol{E}$ at $Q_i$, otherwise.
	  \end{itemize}
Say $Q_i$ is of order $\ol{m}_i$ in $\ol{X}$, then the $\alpha$-value of $Q_i$ with respect to $\ol{E}_i$ is 
$\ol{\alpha}_i = \frac{1}{\ol{m}_i}\left(\ol{\nu}_i-\frac{\ol{\nu}}{\ol{N}}\ol{N}_i\right)$. Now, consider the divisor $\Delta$ on $\ol{X}$ given by 
\[\Delta := \sum_{i\in I_o}\left(\ol{\nu}_i-\frac{\ol{\nu}}{\ol{N}}\ol{N}_i\right) \ol{E}_i.
\]
In order to describe the pullback of $\Delta$, we take  $\{P_1^{(i)},\dots,P_{s_i}^{(i)}\}$ as the set  $\varrho^{-1}Q_i$, so we can write 
 \begin{equation*}\label{Eqn:Qi_upstairs}
	    \varrho^*Q_i = \sum_{\ell=1}^{s_i} n_\ell^{(i)}P_\ell^{(i)},  
	  \end{equation*}
where $n_\ell^{(i)}\geq 1$ is the ramification index of $P_\ell^{(i)}$ along $\varrho$.
For any $\ell=1,\ldots, s_i$, let $E_\ell^{(i)}$ be the curve (exceptional or strict transform component, or curvette) in the preimage of $\ol{E}_i$ that intersects $E$ at the point $P_\ell^{(i)}$. Furthermore, let $P_\ell^{(i)}$ have order $m_\ell^{(i)}$ in $X$, and denote its $\alpha$-value with respect to $E_\ell^{(i)}$ by $\alpha_\ell^{(i)}$.

By~\eqref{Eqn:Ei_upstairs}, we can write $\wt{\rho}^* \ol{E}_i = \sum_{\ell=1}^{s_i} e_\ell^{(i)}E_\ell^{(i)}  + F_i$, where $\supp(F_i)\cap E = \emptyset$. Thus we have 
	\begin{equation*}
		     \varrho^*\beta^*\Delta = \varrho^*\left(\sum_{i\in I_o} \frac{\ol{\nu}_i - (\ol{\nu}/\ol{N}) \ol{N}_i}{\ol{m}_i} Q_i\right) = \sum_{i\in I_o}\sum_{\ell=1}^{s_i} \ol{\alpha}_i n_\ell^{(i)}P_\ell^{(i)}.
	\end{equation*}
The commutativity of the diagram and part (1) imply that the latter  is also equal to
	  \begin{align*}
	    \gamma^*\wt{\rho}^*\Delta &= \gamma^*\left( \sum_{i\in I_o} \left(\ol{\nu}_i - (\ol{\nu}/\ol{N}) \ol{N}_i \right)\left(\sum_{\ell=1}^{s_i}  e_\ell^{(i)} E_\ell^{(i)}+F_i\right) \right) \\
&= \gamma^*\left( \sum_{i\in I_o} \sum_{\ell=1}^{s_i}   \left(\nu_\ell^{(i)} - (\nu/N) N_\ell^{(i)}\right) E_\ell^{(i)} \right)\\
			&= \sum_{i\in I_o} \sum_{\ell=1}^{s_i} \frac{\nu_\ell^{(i)} - (\nu/N) N_\ell^{(i)}}{m_\ell^{(i)}} P_\ell^{(i)} = \sum_{i\in I_o} \sum_{\ell=1}^{s_i} \alpha_\ell^{(i)} P_\ell^{(i)}.
	  \end{align*}
	  We conclude that
	  \begin{equation*}\label{Eqn:Alphas_relations_updown}
	    \ol{\alpha}_i n_\ell^{(i)} = \alpha_\ell^{(i)},
	  \end{equation*}
	  for any $\ol{E}_i$ intersecting  $\ol{E}$ and for all $\ell=1,\ldots, s_i$. 
	  Finally, for any point $Q\in \ol{E}\setminus\P_o$ and $P\in E$ projecting on $Q$, it is clear that $\ol{\alpha}_Q=\alpha_P=1$. 
	\end{proof}
	
	\begin{cor}
		Let $E$ be as in Theorem \ref{thm:MultAlphas_proportional}(2). Consider another irreducible component $E_j$ of $\pi^{-1}(M)$ and take $P\in\P_{\pi}\cap E_j$, a point of order $m$ in $X$. If $Q=\wt{\rho}(P)\in\ol{E_i}=\wt{\rho}(E_j)$ has order $\ol{m}$ in $\ol{X}$, then $m\alpha_P=e\ol{m}\,\ol{\alpha}_Q$, where $e$ is the ramification index of $E_j$ along $\wt{\rho}$.
	\end{cor}

The next result shows a remarkable global property when $E$ is rational.
	
\begin{theorem}	
\label{thm:AtMostTwoAlphas}
Assume the hypothesis of Theorem~\ref{thm:MultAlphas_proportional}. Fix a rational exceptional component $E$ in $\pi^{-1}(M)$, and set $\ol{E}=\wt{\rho}(E)$ in $\ol{X}$. Assume that, for any component $E_j$ of $\pi^{-1}(M)$ that intersects $E$ in a point $P_j$, we have $\alpha_{P_j}:=\alpha_j\neq0$. If $\ol{E}$ is a rupture component, then $E$ is also a rupture component.
\end{theorem}

\begin{proof}
We may assume that the degree $d$ of the finite map $\varrho:=\wt{\rho}\,|_{E} \colon E\to \ol{E}$ is different from~$1$, otherwise the result is clear. We will show the contrapositive  statement, i.e., 
$$\#\set{P_j\in E \mid \alpha_j\neq1}\leq 2  \quad\Rightarrow\quad \#\set{Q_i\in \ol{E} \mid \ol{\alpha}_i\neq1}\leq2.$$
 Assume that  the $\alpha$-value with respect to $E$ is different from $1$ for at most two points. Let $P'$ and $P$ be two points in $E$ with associated values $\alpha'\neq 1$ and $\alpha$. These points project onto $Q'$ and $Q$ in $\ol{E}$ with values $\ol{\alpha}'$ and $\ol{\alpha}$, respectively. By hypothesis and~\eqref{Eq:AlphaArithm}, $\alpha+\alpha'=0$, and we can assume that $\alpha>0$ and $\alpha'<0$, implying that $\ol{\alpha}>0$ and $\ol{\alpha}'<0$ by Theorem~\ref{thm:MultAlphas_proportional}(1). In particular, we have that 
$Q\neq Q'$. Observe that $Q'$ must be a total ramification point of $\varrho$, i.e.,~$\varrho^*Q'=dP'$, since any other point $P_j$ in $E$ projecting to $Q'$ should have corresponding value $\alpha_j<0$, but there are no points in $E$ with $\alpha_j\neq 1$ other that $P'$ and possibly $P$. We will consider two cases, according to whether $Q$ is also a total ramification point  or not, and show that, in each case, there is at most one other point in $\ol{E}$ with associated value different from~1, which will imply the result.
	
	  First, if $Q$ is a total ramification point, 
 then for any other point $Q_i\neq Q, Q'$ in $\ol{E}$ one has that $\varrho^*Q_i = \sum_{\ell=1}^{s_i} n_\ell^{(i)}P_\ell^{(i)}$ with $n_\ell^{(i)}\geq 1$, for any $\ell=1,\ldots, s_i$. By Hurwitz's formula, we have that
	  \[
	    -2 = -2d + 2(d-1) + \sum_{i}\sum_{\ell=1}^{s_i} \left(n_\ell^{(i)} - 1\right),
	  \]
	  implying that $n_\ell^{(i)}=1$, for any $Q_i$ and any $\ell=1,\ldots, s_i$. By Theorem~\ref{thm:MultAlphas_proportional}(2),
	  $\ol{\alpha}_i=\alpha_\ell^{(i)}$, and by hypothesis $\alpha_\ell^{(i)} =1$, so $\ol{\alpha}_i=1$. This shows that no more points  in $\ol{E}$, other than $Q$ and $Q'$, have $\ol{\alpha}\neq 1$.
	
	  Second, suppose that we have exactly $2\leq s\leq d$ points $P=P_1,P_2,\ldots, P_s$ in $E$ with associated values 
	  $\alpha_j=n_j\ol{\alpha}$, with $n_j\geq1$ for $j=1,\ldots,s$, and such that 
	  \[\varrho^*Q = \sum_{j=1}^s n_jP_j.
	  \]
	  For $j\geq2$ one has $1=\alpha_j=n_j\ol{\alpha}$, implying that $n_j=n_2$. Then, by ramification relations,  
	  \begin{equation}\label{Eq:nis_relations}
	    d=n_1+(s-1)n_2.
	  \end{equation}
	  Assume that there exist other points $Q_i\neq Q,Q'$ in $\ol{E}$ such that $\ol{\alpha}_i\neq 1$. Let $\varrho^*Q_i = \sum_{\ell=1}^{s_i} n_\ell^{(i)}P_\ell^{(i)}$ with $n_\ell^{(i)}\geq 1$, for  $\ell=1,\ldots, s_i$. We choose indices in the $Q_i$, starting at $i= 3$. Since $\ol{\alpha}_i=\frac{\alpha_\ell^{(i)}}{n_\ell^{(i)}} = \frac{1}{n_\ell^{(i)}}$, we have that $n_\ell^{(i)}\geq2$ and also $n_\ell^{(i)}=n_1^{(i)}$, for all $\ell=1,\ldots, s_i$. Denoting $n_i=n_\ell^{(i)}$, it is clear that $Q_i$ has exactly $s_i=d/n_i$ preimages with constant ramification index $n_i$. Thus, the total ramification over $Q_i$ is $\frac{d}{n_i}(n_i-1)$.  Observe that
	  \begin{equation}\label{Eq:Ram_estimation}
	    \frac{d}{n_i}(n_i-1) \geq \frac{d}{2},
	  \end{equation}
	  since $n_i\geq 2$ for any $Q_i$. Now we have by Hurwitz's formula that
	  \[
	    -2 = -2d + (d-1) + (n_1-1) + (s-1)(n_2-1) + \sum_{i\geq3} \left(d-\frac{d}{n_i}\right).
	  \]
By~\eqref{Eq:nis_relations} and~\eqref{Eq:Ram_estimation}, the latter reduces to
	  \[
		s-1 = \sum_{i\geq3} \left(d-\frac{d}{n_i}\right)\geq \sum_{i\geq3} \frac{d}{2}.
	  \]
	  But we had $s$ preimages of $Q$ with $2\leq s\leq d$, so the previous relation implies that there exists a unique $Q_i\neq Q,Q'$ (say $Q_i=Q_3$) in $\ol{E}$, such that $\ol{\alpha}_3\neq 1$ and satisfying $s-1=d-\frac{d}{n_3}$. Substituting this in~\eqref{Eq:nis_relations} yields $d=n_1 + (d-d/n_3)n_2\geq n_1 + \frac{d}{2}n_2$. This forces $n_2=1$, and then $\ol{\alpha}=\alpha_2/n_2=1$. Hence the only points in $\ol{E}$ with $\alpha$-value different from $1$ are $Q'$ and $Q_3$, which implies the result.
	\end{proof}
	
	It is worth noticing that the proof of Theorem~\ref{thm:AtMostTwoAlphas} depends only on geometric properties of coverings between rational curves. 
	When these are related by the action of a cyclic group, we describe next, in simple terms, when the condition of being a rupture component is not preserved by the covering.
	\begin{prop}\label{prop:losingPoles}
	Under the hypothesis of Theorem~\ref{thm:AtMostTwoAlphas}, assume that $\varrho:=\wt{\rho}_{|E} \colon E\to \ol{E}$ is induced by the action of a finite cyclic group $\mu_d$ on $E$, with $d>1$. Then we have the following.
	  \begin{enumerate}
		\item There are exactly two different ramification points in $\ol{E}$, and both are total ramification points.
		
		\item If $P$ and $Q=\varrho(P)$ have associated values $\alpha$ and $\ol{\alpha}$, respectively, then either $\ol{\alpha}=\frac{\alpha}{d}$ or $\ol{\alpha}=\alpha$, according to whether $Q$ is a total ramification point or not.
		
		\item If $E$ is a rupture component,  
		then either
		\smallskip
		\begin{enumerate}
			\item $\ol{E}$ is also rupture,  			or
			\item $\ol{E}$ is not rupture, and 
			$\alpha_1=d$ or $\alpha_2=d$, where these values are the ones associated with the unique preimages of the two total ramification points of $\ol{E}$, respectively.
		\end{enumerate}
	\end{enumerate}
	\end{prop}
	
	\begin{proof}
	  The cyclic action forces that, for any $Q\in\ol{E}$, there exist some $r\geq1$ such that $\varrho^*Q=\sum_{i=1}^r nP_i$, with $n=d/r$. Following the proof of Theorem~\ref{thm:AtMostTwoAlphas}, together with Hurwitz's formula, the reader can verify (1), and then (2) follows by Theorem~\ref{thm:MultAlphas_proportional}(2). Assertion (3) is a direct consequence of (1) and (2). 
	\end{proof}
	
Theorem~\ref{thm:AtMostTwoAlphas} guarantees that the number of rupture components does not increase in the target whenever we  compare $\QQ$-resolutions. 
In the light of Theorem~\ref{Thm:RodVeysQres}, case~(3)--(b) above describes in geometrical terms the situation where a pole in the source `disappears' in the target, see Example~\ref{ex:noPoleDownstairs}.

\subsection{Normal crossing under finite cyclic group actions}\label{sec:BadCases}

Looking at the commutative diagram in Theorem~\ref{thm:MultAlphas_proportional}, one could naively conclude that any finite morphism
$\rho: S \to \ol{S}$ maps
$\QQ$-normal crossing $M = \supp(D) \cup \supp(W) \cup \supp (\Ram_\rho)$ in the source to $\QQ$-normal crossing
$\ol{M} = \supp(\ol{D}) \cup \supp(\ol{W}) \cup \supp(B_\rho)$ in the target, where $D = \rho^* \ol{D}$
and $W = \rho^* \ol{W} + \Ram_\rho$. To see that this is not always the case, it is enough to consider the simple case of quotient morphisms under finite cyclic group actions, see Example~\ref{ex:noQNC}. The purpose of this section is to discuss and classify pathological cases, where
$M$ is  normal crossing in $\CC^2$, while $\ol{M}$ is not in $\CC^2/\mu_d$.

\begin{setting}\label{setting2}
	Take $S=\CC^2$, $\ol{S}=\CC^2/\mu_d= X(d;a,b)$, with $\gcd(d,a,b)=1$,
	 and let $\rho\colon S \to \ol{S}$ be the natural covering map. The associated branch divisor $B_\rho$ on $\ol{S}$ is given by
	\[
	B_\rho = \left( 1-\frac{1}{e_1} \right) \ol{L}_1
	+ \left( 1-\frac{1}{e_2} \right) \ol{L}_2,
	\]
	where $\ol{L}_1$ and $\ol{L}_2$ are the divisors in $\ol{S}$ defined by $\{x=0\}$ and $\{y=0\}$, respectively, and $e_1 = \gcd(d,b)$ and $e_2 = \gcd(d,a)$. Note that the action of $\mu_d$ is small if and only if $B_\rho = 0$.
\end{setting}

Consider moreover two  $\QQ$-divisors $D$ and $W$  in $\CC^2$, as in Setting \ref{setting1}, that are (globally) invariant  under the action of $\mu_d$, and such that $M=\supp (D)\cup\supp(W)$ is normal crossing. For the purpose of our study, it suffices to analyze the behavior of the action on the irreducible components of $M$.

\subsubsection{Each irreducible component is $\mu_d$-invariant}\label{subsec:non-pathological}

First, assume that $C$ is a smooth $\mu_d$-invariant irreducible component of $M$. 
Let $\zeta$ be a $d$-th primitive root of unity. 
By irreducibility and invariance under the action, we classify the different possible situations.
\begin{enumerate}
  \item $C$ is one of the coordinate axes, and $a,b$ are arbitrary.
  
  \item $C$ is not an axis, but it is tangent to an axis, say to $L_1\colon x=0$. In this case, we can assume that the germ $C$ has associated holomorphic function
  \[
    f(x,y) := x + \lambda y^\ell + h(x,y),
  \]
  where $\ell\geq 2$,  $\lambda \in \CC\setminus \{0\}$, and $h(x,y)$ is a convergent power series of order at least $2$ containing other monomials (or identically zero). Since $C$ is $\mu_d$-invariant, $f$ should be equal to $(\zeta\cdot f)$ up to a constant, namely equal to
  \[
  x + \lambda\zeta^{b\ell-a}y^\ell + \zeta^{-a} (\zeta\cdot h)(x,y).
  \]
  In particular, $b\ell\equiv a \mod d$, and then $X(d;a,b)=X(d; b\ell, b) = X(d;\ell,1)$, since the relation $\gcd(d,a,b)=1$ imposes that $e_1=\gcd(d,b)=1$ in this situation. Finally, $B_\rho=(1-\frac{1}{e_2})\ol{L}_2$, where $e_2=\gcd(d,\ell)$.
(Note that $h(x,y)$ can contain other powers $y^{\ell'}$ of $y$, for which we conclude similarly that $b\ell'\equiv a \mod d$. Then $\ell'\equiv \ell \mod d$ and $\gcd(d,\ell')=\gcd(d,\ell)$.)
  
  \item $C$ is tangent to a line different from the axes. 
  Then $C$ has equation
  $$
	f(x,y) := \alpha x + \beta y + h(x,y), 
  $$
  where  $\alpha,\beta \in \CC\setminus \{0\}$, and $h(x,y)$ is a convergent power series of order at least $2$ (or identically zero). Applying the $\mu_d$-action, we obtain
  $$
    (\zeta\cdot f)(x,y) = \alpha\zeta^{a} x + \beta \zeta^{b} y + (\zeta\cdot h)( x, y).
  $$
  This forces $a\equiv b\mod d$, and then also $e_1=e_2=1$. Hence $X(d;a,b)=X(d;1,1)$ and $B_\rho=0$.
\end{enumerate}
It is worth noticing that any of the previous cases can be reduced to the axis case, since  $x' = f(x,y), y'=y$ provides a change of coordinates that is invariant under the action, and such that both $C$ and $\ol{C}$ are defined by $x'$, while $B_\rho$ remains invariant.

\begin{table}[ht]
	\begin{tabular}{|>{\centering\arraybackslash}m{1.9cm}|>{\centering\arraybackslash}m{5.2cm}|>{\centering\arraybackslash}m{1.9cm}|>{\centering\arraybackslash}m{5.2cm}|}
		\hline
		\multicolumn{2}{|c|}{$M=C$ is irreducible} & \multicolumn{2}{c|}{$M = C_1 \cup C_2$ and $C_i$ is $\mu_d$-invariant}
		\\ \hline
		\begin{tikzpicture}
			\colorlet{colsing}{bluegreen2}
			\colorlet{colsingtext}{colsing}
			\node[color=colsing] at (0,0) {\large $\bullet$};
			\draw[-open triangle 45, thick] (0,-1) -- (0,1);
			\draw[ thick, dashed] (.02,-1) -- (.02,1); 
			\draw[-open triangle 45, thick, dashed] (-1,0) -- (1,0);
		\end{tikzpicture} 
		&
		\shortstack{$X(d;a,b)$\\[.5em] $B_\rho = (1-\frac{1}{e_1})\overline{L}_1 + (1-\frac{1}{e_2})\overline{L}_2$} 
		&
		\begin{tikzpicture}
			\colorlet{colsing}{bluegreen2}
			\colorlet{colsingtext}{colsing}
			\node[color=colsing] at (0,0) {\large $\bullet$};
			\draw[-open triangle 45, thick] (0,-1) -- (0,1);
			\draw[thick, dashed] (.02,-1) -- (.02,1); 
			\draw[-open triangle 45, thick] (-1,0) -- (1,0);
			\draw[thick, dashed] (-1,.02) -- (1,.02);
		\end{tikzpicture}
		&
		\shortstack{$X(d;a,b)$\\[.5em] $B_\rho = (1-\frac{1}{e_1})\overline{L}_1 + (1-\frac{1}{e_2})\overline{L}_2$}
		\\ \hline
		\multirow{2}{*}{\shortstack{
				\begin{tikzpicture}
					\draw[->, gray] (0,-1) -- (0,1); 
					\colorlet{colsing}{bluegreen2}
					\node[color=colsing] at (0,0) {\large $\bullet$};
					\draw[-open triangle 45, thick] (-.8,-.8) to[out=0, in=270] (0,0) to[out=90, in=0] (-1,.8);
					\draw[-open triangle 45, thick, dashed] (-1,0) -- (1,0);
				\end{tikzpicture}
			}
		} & 
		\multirow{2}{*}{\shortstack{
				\\
				$X(d;\ell,1)$\\[.5em]
				$B_\rho = (1-\frac{1}{e_2})\overline{L}_2$%
			}
		}
		&
		\scalebox{0.5}{
			\begin{tikzpicture}
				\draw[->, gray] (0,-1) -- (0,1); 
				\colorlet{colsing}{bluegreen2}
				\node[color=colsing] at (0,0) {\large $\bullet$};
				\draw[-open triangle 45, thick] (-.8,-.8) to[out=0, in=270] (0,0) to[out=90, in=0] (-1,.8);
				\draw[-open triangle 45, thick] (-1,0) -- (1,0);
				\draw[ thick, dashed] (-1,0.02) -- (1,0.02); 
			\end{tikzpicture}
		}
		&  
		\shortstack{\small $X(d;\ell,1)$\\
			$B_\rho = (1-\frac{1}{e_2})\overline{L}_2$}
		\\ \cline{3-4} 
		& &  
		\scalebox{0.5}{
			\begin{tikzpicture}
				\draw[->, gray] (-1,0) -- (1,0); \draw[->, gray] (0,-1) -- (0,1);
				\colorlet{colsing}{bluegreen2}
				\node[color=colsing] at (0,0) {\large $\bullet$};
				\draw[-open triangle 45, thick] (-.6,-1) to[out=30, in=270] (0,0) to[out=90, in=-30] (-.6,.8);
				\draw[-open triangle 45, thick] (-1,-.4) to[out=30, in=180] (0,0) to[out=0, in=150] (1,-.4);
			\end{tikzpicture}
		}
		&  
		\shortstack{\small $X(d;\ell_1,1)=X(d;1,\ell_2)$\\
			$B_\rho = 0$} 
		\\ \hline 
		\begin{tikzpicture}
			\colorlet{colsing}{bluegreen2}
			\colorlet{colsingtext}{colsing}
			\draw[->, gray] (-1,0) -- (1,0); \draw[->, gray] (0,-1) -- (0,1);
			\node[color=colsing] at (0,0) {\large $\bullet$};
			\draw[-open triangle 45, thick] (-1,-.2) to[out=-15, in=225] (0,0) to[out=45, in=280] (0.2,1);
		\end{tikzpicture}                
		&
		\shortstack{$X(d;1,1)$\\[.5em] $B_\rho = 0$ \\[.5em] $C$ can be a line germ}
		&
		\begin{tikzpicture}
			\draw[->, gray] (-1,0) -- (1,0); \draw[->, gray] (0,-1) -- (0,1);
			\colorlet{colsing}{bluegreen2}
			\node[color=colsing] at (0,0) {\large $\bullet$};
			\draw[-open triangle 45, thick] (-1,-.2) to[out=-15, in=225] (0,0) to[out=45, in=280] (0.2,1);
			\draw[-open triangle 45, thick] (-.4,.8) to[out=270, in=135] (0,0) to[out=-45, in=180] (.8,-.2);
		\end{tikzpicture}
		&  
		\shortstack{$X(d;1,1)$\\[.5em] $B_\rho = 0$\\[.5em] Any $C_i$ can be a line germ}
		\\
		\hline
	\end{tabular}
	\medskip
	\caption{All possible situations where normal crossing in~$\CC^2$ induces $\QQ$-normal crossing in $\CC^2/\mu_d$, classified by their geometry with respect to the axes.}\label{table:NcQnc}
\end{table}

Secondly, consider $M=C_1\cup C_2$ normal crossing, where $C_1$ and $C_2$ are the two smooth irreducible $\mu_d$-invariant components of $M$. 
Since each $C_i$ ($i=1,2$) is $\mu_d$-invariant, and we are in a normal crossing situation, it follows from the previous discussion that we can have the following possible cases.
\begin{enumerate}
  \item $C_1$ and $C_2$ are the coordinate axes, %with multiplicities $N_1$ and $N_2$, 
and $a,b$ are arbitrary.
  
  \item $C_1$ is not an axis, but tangent to one, say to $L_1$. If $C_2$ is $L_2$, then $e_1=1$,  $X(d;a,b)=X(d;\ell,1)$ for some $\ell\geq 2$, and $B_\rho=(1-\frac{1}{e_2})\ol{L}_2$, where $e_2=\gcd(d,\ell)$. If $C_2$ is not $L_2$, but it is tangent to $L_2$, then $e_1=e_2=1$ and $B_\rho=0$.
 
 \item At least one $C_i$ is tangent to a line different from the axes. 
 This forces $X(d;a,b)=X(d;1,1)$ and $B_\rho=0$.
\end{enumerate}
Once more, $x' = f_1(x,y)$ and $y'=f_2(x,y)$, where $f_i$ is the function associated to $C_i$ ($i=1,2$), provides a $\mu_d$-equivariant change of coordinates, such that $C_1$ and $C_2$ (and also $\ol{C_1}$ and $\ol{C_2}$) become the coordinate axes.

Finally, take $D=N_1C_1+N_2C_2$ and $W = (\nu_1-1)C_1 + (\nu_2-1)C_2$ (where thus each $N_i \geq 0$, $(N_1,N_2)\neq (0,0)$, and, if $N_i=0$, then $\nu_i \neq 0$), and assume that $M=\supp(D)\cup\supp(W)$ is normal crossing in $\CC^2$. 
So $M$ has one or two components.
In any case, under the assumption that each irreducible component is $\mu_d$-invariant, it follows that $\ol{M}=\supp(\ol{D})\cup\supp(\ol{W})\cup B_\rho$ is $\QQ$-normal crossing, since in particular neither $\ol{D}$ nor $\ol{W}$ can be tangent to $B_\rho$ (see Table~\ref{table:NcQnc}). The topological zeta function in this setting is easily computable, since formula~\eqref{Eqn:ZtopQres} is still valid for non-small actions (see Lemma~\ref{lemma:Ztop_nonsmall}):
\[
  \Ztopp{(\ol{S},0)}(\ol{D},\ol{W};s) = \frac{d}{(N_1s+\nu_1)(N_2s+\nu_2)} = d \cdot \Ztopp{(S,0)}(D,W;s).
\]
In conclusion, both Theorem~\ref{Thm:Comparison_Ztop} and~\ref{Thm:dZtop} hold in this particular case.

\subsubsection{The two irreducible components are in the same $\mu_d$-orbit}\label{subsec:orbit-pathology}

Suppose that $M$ has irreducible (smooth) components $C_1$ and $C_2$, which are permuted by the $\mu_d$-action. Each $C_i$ is defined by
\[
f_i(x,y) := \alpha_i x + \beta_i y + h_i(x,y),
\]
where  $(\alpha_i,\beta_i) \neq (0,0)$ and $h_i(x,y)$ is a convergent power series of order at least $2$ (or identically zero), for $i=1,2$. Then $\ol{M}$ is not  $\QQ$-normal crossing, since $f_1 f_2$ has no linear term, independently of whether $B_\rho=0$ or not. 

Writing $D=N C_1+N C_2$ and $W = (\nu- 1)C_1 + (\nu- 1) C_2$, with $N>0$ and $\zeta\cdot C_1=C_2$ and vice versa, the fact that $\supp(D)\cup\supp(W)$ is normal crossing in $\CC^2$ forces either $\supp(D)=\supp(W)$, or $W=0$.

This pathological case, i.e.,~$M$ is normal crossing but $\ol{M}$ is not, gives rise to a very concrete situation where
both the quotient singularity $X(d;a,b)$ and the topological zeta functions can be explicitly computed. The rest of this section is devoted to this task.  

Note that the normal crossing condition on $M$ implies that $\alpha_1 \beta_2 - \alpha_2 \beta_1 \neq 0$. 
The equation of $C_1$ is mapped to the equation of $C_2$ under the action of $\zeta$, i.e., $f_2(x,y) = \alpha_1 x + \beta_1 \zeta^{b-a} y + \zeta^{-a} (\zeta\cdot h_1)(x,y))$. 
In particular, $\alpha_1 \beta_1 (\zeta^{b-a}-1) \neq 0$, since otherwise $C_1 \cup C_2$ would not define a normal crossing divisor. Similarly, the equation of $C_2$ is mapped to the equation of $C_1$ under the action of $\zeta$. We conclude that $\zeta^{2(b-a)}=1$, and thus that $\zeta^{b-a}=-1$. 

We may assume that $0 \leq a \leq b < d$. Then $0 \leq 2(b-a) < 2d$, and by the above we have also that $2(b-a) \equiv 0 \mod d$ and $b-a \not\equiv 0 \! \mod d$. Hence $2(b-a) = d$. 

Note that the action of $\mu_d$ is small if and only if $a$ and $b$ are both odd. In such a case $B_\rho = 0$ and $d$ is multiple of $4$. If the action is not small, then $a$ is odd and $b$ is even (or resp.~$a$ is even and $b$ is odd), $B_\rho = - \frac{1}{2} \ol{L}_1$ (resp.~$B_\rho = - \frac{1}{2} \ol{L}_2$), and $d \equiv 2 \!\mod 4$.

\smallskip

Case $(1)$: $d\equiv 0 \mod 4$. 
Let $\pi: X \to \CC^2$ be the blow-up of the origin and denote by $\ol{X}$ the quotient $X / \mu_d$. Let $\ol{\pi}: \ol{X} \to \CC^2/\mu_d$ be the induced map on the quotient spaces, which coincides with the blow-up of the origin of $X(d;a,b)$ with weights $(1,1)$.  The quotient morphism $\wt{\rho}: X \to \ol{X}=X / \mu_d$ makes the diagram 
of Theorem~\ref{thm:MultAlphas_proportional} commute.
\begin{figure}[ht]
	\centering
	\definecolor{bluegreen2}{RGB}{0, 85, 127}
	\begin{tikzpicture}[scale=1]
	\colorlet{col1}{blue!40}
	\colorlet{col2}{red!40}
	\colorlet{col3}{violet!40}
	\colorlet{colsing}{bluegreen2}
	\colorlet{colsingtext}{colsing}	
	\def\xmove{120}
	\begin{scope}[xshift=-\xmove]
		\draw[very thick] (-1.75,0)--(1.75,0) node [right, xshift=3pt] {$E \, (2N,2\nu)$};
		\node (X1) at (-.4,0) {$\times$};
		\node (X2) at (.4,0) {$\times$};
		\draw[-open triangle 45] (0.4, -0.8) -- (0.4,1);
		\draw[-open triangle 45] (-0.4, -0.8) -- (-0.4,1);
		\draw[-open triangle 45, dashed] (-1.2, -0.8) -- (-1.2,1);
		\draw[-open triangle 45, dashed] (1.2, -0.8) -- (1.2,1);
		\draw (-1.2,1.1) node[above] {$\widehat{L}_2$};
		\draw (1.2,1.1) node[above] {$\widehat{L}_1$};
		\draw[decoration={brace,raise=1pt},decorate] (-0.55,1) -- node[above=5pt] {$\widehat{D}$} (0.55,1);
	\end{scope}
	\draw[->] (.2,0.5) --  (1.2,0.5) node[midway,above] {$\widetilde{\rho}$};
	\draw (0.7,0.2) node {\small $d:1$};
	\begin{scope}[xshift=\xmove]		
		\draw[very thick] (-2.1,0)--(2.1,0) node [right, xshift=3pt] {$\overline{E} \, \left(\dfrac{4N}{d},\dfrac{4\nu}{d}\right)$};
		\node[below, yshift=-3pt, text=colsingtext] at (-1.9,0) {$(2;1,1)$};
		\draw (-1.2,1) node[above] {$\widehat{\overline{L}}_2$};
		\draw (1.2,1) node[above] {$\widehat{\overline{L}}_1$};
		\draw (0,1.1) node[above] {$\widehat{\overline{D}}$};
		\draw[-open triangle 45, dashed] (-1.2, -0.8) -- (-1.2,1);
		\draw[-open triangle 45, dashed] (1.2, -0.8) -- (1.2,1);
		\node (X1) at (0,0) {$\times$};
		\draw[-open triangle 45] (0, -0.8) -- (0,1);
		\node[color=colsing] (O1) at (-1.2,0) {\Large $\bullet$};
		\node[color=colsing] (O2) at (1.2,0) {\Large $\bullet$};
		\node[below, xshift=10pt, yshift=-3pt, text=colsingtext] at (1.6,0) {$(2;1,1)$};
	\end{scope}
	\end{tikzpicture}
	\caption{Embedded $\QQ$-resolution of $\ol{M}$ in $X(d;a,b)$, where $\widehat{\ol{D}}$ has numerical data $(N,\nu)$, and its covering map.}
	\label{fig:JorgeEx2}
\end{figure}
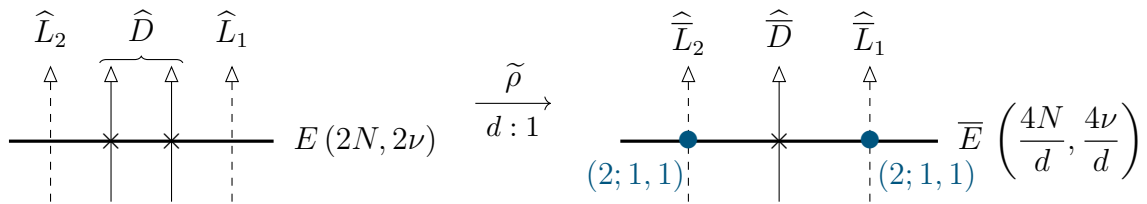

\noindent In this situation $\ol{\pi}: \ol{X} \to \CC^2/\mu_d$ is an embedded $\QQ$-resolution and we can use the numerical data on the right-hand side of Figure~\ref{fig:JorgeEx2}
to compute the topological zeta function. We divide the computation into four different strata (the open part of $\ol{E}$, the two origins of $\ol{X}$, and $\widehat{\ol{M}} \cap \ol{E}$). We obtain
\begin{equation}\label{Eq:ExZCase1}
	\begin{aligned}
		\Ztopp{{(X(d;a,b),[0])}}(\ol{D},\ol{W};s) &= 
		\frac{-1}{\frac{4N}{d} s + \frac{4\nu}{d}}
		+ 2 \cdot \frac{2}{\frac{4N}{d} s + \frac{4\nu}{d}}
		+ \frac{1}{(\frac{4N}{d} s + \frac{4\nu}{d})(Ns+\nu)}\\%= \frac{d}{4} \cdot \frac{3Ns + 4}{(Ns+1)^2}\\
		& = \frac{d}{4} \cdot \frac{3Ns + 3\nu+1}{(Ns+\nu)^2}.
	\end{aligned}  
\end{equation}

Following Figure~\ref{fig:JorgeEx2}, the two origins of $\ol{X}$ are singular, of type $(2;1,1)$. Blowing up these two points, one obtains a usual embedded resolution as shown in Figure~\ref{fig:JorgeEx3A}, where the numbers between brackets are self-intersection numbers. A concrete instance of this case is presented in Example~\ref{ex:noQNC}.

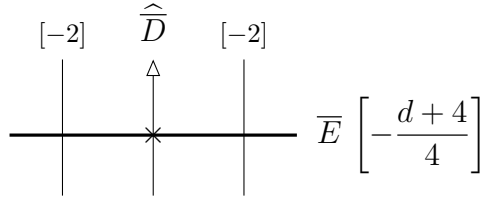
\begin{figure}[ht]
	\centering
	\definecolor{bluegreen2}{RGB}{0, 85, 127}
	\begin{tikzpicture}[scale=1]
	\colorlet{col1}{blue!40}
	\colorlet{col2}{red!40}
	\colorlet{col3}{violet!40}
	\colorlet{colsing}{bluegreen2}
	\colorlet{colsingtext}{colsing}
	\draw[very thick] (-1.9,0)--(1.9,0) node [right, xshift=3pt] {$\overline{E} \, \left[-\dfrac{d+4}{4}\right]$};
	\draw (-1.2,1) node[above, xshift=0pt] {\footnotesize $[-2]$};
	\draw (1.2,1) node[above, xshift=0pt] {\footnotesize $[-2]$};
	\draw (0,1.1) node[above] {$\widehat{\overline{D}}$};
	\draw (-1.2, -0.8) -- (-1.2,1);
	\draw (1.2, -0.8) -- (1.2,1);
	\node (X1) at (0,0) {$\times$};
	\draw[-open triangle 45] (0, -0.8) -- (0,1);
    \end{tikzpicture}
	\caption{Case $(d;1,\frac{d}{2}+1)$, $d \equiv 0 \ (4)$.}
	\label{fig:JorgeEx3A}
\end{figure}

\begin{ex}\label{ex:noQNC}
	Consider the divisor $D: (x^2+y^2)^N=0$, with two irreducible components in $\CC^2$ (and $W=0$). Take the natural quotient map $\CC^2\to X(4;1,3)$ with empty branch locus. Then $D$ is normal crossing up to a change of variables, but  $\ol{D}$ in $X(4;1,3)$ has a unique irreducible component and is not  $\QQ$-normal crossing. This follows from the fact that there is no analytic change of variables in $X(4;1,3)$ which transforms $\ol{D}$ into an axis and keeps the action diagonal. 
	An embedded $\QQ$-resolution can be achieved by a single blow-up, see~Figure~\ref{fig:JorgeEx2}. 
	Finally
	\[
	\Ztopp{(\CC^2,0)}(D,0;s)=\frac{1}{(Ns+1)^2}
	\ \text{ , but }\quad
	\Ztopp{(X(4;1,3),[0])}(\ol{D},-B_\rho;s)=\frac{3Ns+4}{(Ns+1)^2}.
	\]
\end{ex}

Case $(2)$: $d\equiv 2 \mod 4$, with $a$ odd and $b$ even. 
As in the previous case, $\pi: X \to \CC^2$ denotes the blow-up of the origin and $\ol{\pi}: \ol{X} \to \CC^2/\mu_d$ denotes the induced map on the quotient spaces. We show the configuration of  $\ol{X}$ in Figure~\ref{fig:JorgeExCase2}.

\begin{figure}[ht]
\centering
\definecolor{bluegreen2}{RGB}{0, 85, 127}
\begin{tikzpicture}[scale=1]
	\colorlet{col1}{blue!40}
	\colorlet{col2}{red!40}
	\colorlet{col3}{violet!40}
	\colorlet{colsing}{bluegreen2}
	\colorlet{colsingtext}{colsing}		
	\draw[very thick] (-2.1,0)--(2.1,0) node [right, xshift=3pt] {$\overline{E} \, \left(\dfrac{4N}{d},\dfrac{4\nu}{d}\right)$};
	\node[below, yshift=-3pt, text=colsingtext] at (-1.9,0) {$(2;1,1)$};
	\draw (-1.2,1) node[above] {$\widehat{\overline{L}}_2$};
	\draw (1.2,1) node[above] {$\widehat{\overline{L}}_1$};
	\draw (0,1.1) node[above] {$\widehat{\overline{D}}$};
	\draw[-open triangle 45, dashed] (-1.2, -0.8) -- (-1.2,1);
	\draw[-open triangle 45] (1.2, -0.8) -- (1.2,1);
	\node (X1) at (0,0) {$\times$};
	\draw[-open triangle 45] (0, -0.8) -- (0,1);
	\node[color=colsing] (O1) at (-1.2,0) {\Large $\bullet$};
\end{tikzpicture}
\caption{Embedded $\QQ$-resolution of $\ol{M}$ in $X(d;a,b)$, where $\wh{\ol{D}}$ and $\wh{\ol{L}}_1$ have numerical data $(N,\nu)$ and $(0,1/2)$, respectively.}
\label{fig:JorgeExCase2}
\end{figure}
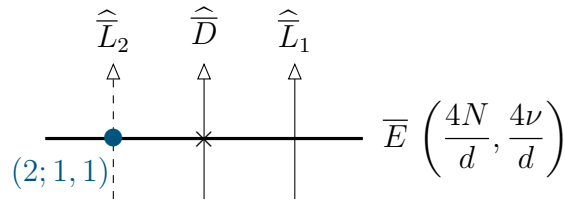

In this case the topological zeta function $\Ztopp{{(X(d;a,b),[0])}}(\ol{D},\ol{W};s)$  is given by 
$$
	\frac{-1}{\frac{4N}{d} s + \frac{4\nu}{d}}
	+ \frac{2}{\frac{4N}{d} s + \frac{4\nu}{d}}
	+ \frac{1}{(\frac{4N}{d} s + \frac{4\nu}{d})(0\cdot  s + \frac{1}{2})}
	+ \frac{1}{(\frac{4N}{d} s + \frac{4\nu}{d})(Ns+\nu)},
$$
which coincides with the expression~\eqref{Eq:ExZCase1}. There is only one quotient singular point of type $(2;1,1)$. Again, blowing-up this point one gets a usual embedded resolution as shown in Figure~\ref{fig:JorgeEx4}.

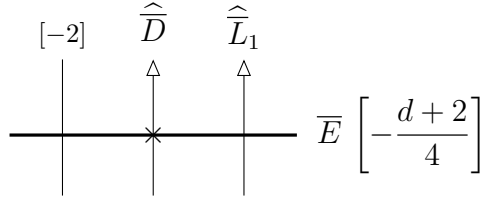
\begin{figure}[ht]
\centering
\definecolor{bluegreen2}{RGB}{0, 85, 127}
\begin{tikzpicture}[scale=1]
	\colorlet{col1}{blue!40}
	\colorlet{col2}{red!40}
	\colorlet{col3}{violet!40}
	\colorlet{colsing}{bluegreen2}
	\colorlet{colsingtext}{colsing}
	\draw[very thick] (-1.9,0)--(1.9,0) node [right, xshift=3pt] {$\overline{E} \, \left[-\dfrac{d+2}{4}\right]$};
	\draw (-1.2,1) node[above, xshift=0pt] {\footnotesize $[-2]$};
	\draw (1.2,1) node[above] {$\widehat{\overline{L}}_1$};
	\draw (0,1.1) node[above] {$\widehat{\overline{D}}$};
	\draw (-1.2, -0.8) -- (-1.2,1);
	\draw[-open triangle 45] (1.2, -0.8) -- (1.2,1);
	\node (X1) at (0,0) {$\times$};
	\draw[-open triangle 45] (0, -0.8) -- (0,1);
\end{tikzpicture}
\caption{Case $(d;1,\frac{d}{2}+1) \simeq (\frac{d}{2};1,\frac{d+2}{4})$, $d \equiv 2 \ (4)$.}
\label{fig:JorgeEx4}
\end{figure}

In both cases, note that Theorem~\ref{Thm:Comparison_Ztop} holds, but not the conclusion of Theorem~\ref{Thm:dZtop}. This phenomenon evidences that the hypotheses of Theorem~\ref{Thm:dZtop} are sharp.

\begin{remark}\label{rmk:BpNeverTangent}
  The analysis above describes any local situation that can occur for finite cyclic quotients of normal crossing divisors. In particular, it has the following consequences for constructing a compatible $\QQ$-resolution $\ol{\pi}\colon \ol{X} \to \CC^2/G$ from a usual resolution $\pi\colon X\to \CC^2$, produced as a finite sequence of blow-ups 
  (verifying the hypothesis of Theorem~\ref{thm:MultAlphas_proportional}).
  \begin{enumerate}[label=(\roman*)]
    \item If $\supp(D)\cup\supp(W)$ is already normal crossing in $\CC^2$, then $\pi=\id$ induces a $\QQ$-resolution $\ol{\pi}$ in the target, except in the pathological case described in Section~\ref{subsec:orbit-pathology}, where an extra blow-up  suffices to achieve the construction. In both situations, equality~\eqref{eq:equality-top} of Theorem~\ref{Thm:Comparison_Ztop} was verified by explicit computations of the corresponding topological zeta functions.
    
    \item Assume $\pi\neq\id$. Moreover, suppose that for any chart $U\subset X$ the restricted map $\wt{\rho}_{|U}\colon U\to \wt{\rho}(U)$ is equivalent to the natural covering $\CC^2\to X(d;a,b)$ by the action of some cyclic group $\mu_d$. Then $\ol{\pi}$ is a $\QQ$-resolution. This follows because any normal crossing situation in $X$ involves at least an exceptional component $E$, which is locally $\mu_d$-invariant by construction. Thus, only the non-pathological cases described in Section~\ref{subsec:non-pathological} may occur locally, and we have a $\QQ$-normal crossing in $\ol{X}$. In particular, the strict transform in $\ol{X}$ of the branch divisor $B_\rho$ becomes disjoint from the strict transform of $\ol{D}$ and $\ol{W}$.
  \end{enumerate}
\end{remark}

	% =========================================================
	% Poles of motivic zeta functions
	% =========================================================
	\bigskip
	\section{Poles of motivic zeta functions}\label{sec:Zmot}
	
	We study relations between poles at the level of motivic zeta functions, establishing the inclusion stated in~Theorem~\ref{Thm:Comparison_Zmot} by constructing embedded $\QQ$-resolutions in the source and the target, and then making use of the geometrical characterization of motivic poles. Finally, we discuss an example where the reverse inclusion of Theorem~\ref{Thm:Comparison_Zmot}  is not satisfied.
		
	\subsection{Induced resolutions and coverings of exceptional divisors}\label{subsec:Zmot_inducedRes}
	
	We go back to the original setting of Section~\ref{ssec:Numerical_Relations}, namely, $\rho\colon (S, o)\to (\ol{S}, \ol{o})$ is a finite morphism between normal surfaces $S$ and $\ol{S}$, with associated ramification divisor $\Ram_\rho$ on $S$ and branch divisor $B_\rho$ on $\ol{S}$.
	Fix two $\QQ$-Weil divisors  $\ol{D}$ and $\ol{W}$ on $(\ol{S},\ol{o})$, and put  $D = \rho^* \ol{D} $ and $ W =\rho^* \ol{W} + \Ram_\rho$. We  take $\ol{M}=\supp(\ol{D})\cup\supp(\ol{W})\cup\supp(B_\rho)$, and we fix an embedded resolution $\ol{\pi} \colon \ol{X}\to (\ol{S},\ol{o})$ of $\ol{M}$. %
	We are going to construct a $\QQ$-resolution of $M=\rho^{-1}(\ol{M})$ in the source, which commutes with the resolution in the target, generalizing the ideas of Jung's construction~\cite{Ju08}, see also~\cite{Pop11} for a modern presentation.
	
	\vspace{1cm}
	
	First, take the fiber product $Y=S\times_{\overline{S}}\ol{X}$. Note that $Y$ could be non reduced, so we consider its reduction $Y_\mathrm{red}$, and then its normalization $Z\rightarrow Y_\mathrm{red}$. We obtain the following commutative diagram.
	\begin{center}
		\begin{tikzcd}
			S \arrow[d, twoheadrightarrow, "\rho"]
			& \arrow[l, dashed] Y=S\times_{\overline{S}}\ol{X} \arrow[d, dashed, twoheadrightarrow] \arrow[r, hookleftarrow] &   Y_\mathrm{red}  & \arrow[l] Z \arrow[dll, dashed, twoheadrightarrow, "\wt{\rho}"] \arrow[lll, dashed, "\pi" above, bend right]\\
			\overline{S} & \arrow[l, "\ol{\pi}" above] \ol{X}
		\end{tikzcd}
	\end{center}
	
	\begin{prop}\label{prop:pi-Q-res}
	  As constructed above, we have a commutative diagram
	  \begin{equation}\label{eq:diagramQresS}
		\begin{tikzcd}
			S \arrow[d, twoheadrightarrow, "\rho"]
			& \arrow[l, dashed, "\pi" above] Z \arrow[d, dashed, twoheadrightarrow, "\wt{\rho}"]\\
			\overline{S} & \arrow[l, "\ol{\pi}" above] \ol{X}
		\end{tikzcd}
	  \end{equation}
	  verifying the following properties:
	  \begin{enumerate}
	
	    \item $\wt{\rho}$ is a finite morphism of the same degree as $\rho$,
	
	    \item $Z$ is a normal surface whose singularities are cyclic quotient singularities.
	  \end{enumerate}
	  Moreover, $\pi: Z\to (S,o)$ is an embedded $\QQ$-resolution of $M$. 
	\end{prop}
	
	\begin{proof}
	  By properties of fiber products and normalization, it follows that~\eqref{eq:diagramQresS} is a commutative diagram with $\pi$ proper and birational. In addition, since normalization preserves finiteness, $\wt{\rho}$ is finite of degree $\deg(\wt{\rho}) = [\K(Z):\K(\ol{X})]=[\K(S):\K(\ol{S})]=\deg(\rho)$.
	
	  Let $\ol{E}_i$ be an  exceptional component of $\ol{\pi}$. Since $\wt{\rho}$ is a finite morphism, $\wt{\rho}^{-1}(\ol{E}_i)$ is the union of certain irreducible components $E_1,\ldots, E_{r_i}$ in $Z$, and each $\wt{\rho}\,|_{E_j} : E_j\to \ol{E}_i$ is a finite morphism. Now, for any $j=1,\ldots, r_i$, one has that  $(\ol{\pi}\circ\wt{\rho})(E_j)=\set{\ol{o}}$, since $\wt{\rho}(E_j)=\ol{E}_i$ is exceptional. By commutativity and the fact that $\rho^{-1}(\ol{o})=\set{o}$, one has that $\pi(E_j)=\set{o}$. Thus, $E_j$ is an exceptional component of $\pi$ in $Z$.
	
	 Next, we study the branch locus $B_{\wt{\rho}}$. 
It is included in the union of the exceptional locus of $\ol{\pi}$ and the strict transform of $B_\rho$. Hence it has a normal crossing in $\ol{X}$, in particular,
any point on $B_{\wt{\rho}}$ is at worst a node. By~\cite[Theorem~5.2]{CompactComplexSurfaces:book}, see also~\cite[Theorem~2.3.1]{Nemethi22:book}, the singular points of $Z$ are cyclic quotient singularities. Moreover, analytically at any point $z\in Z$, the finite morphism $\wt{\rho}\colon (Z,z)\to (\CC^2,0)$ is isomorphic to the one given by the composition of maps
	 \begin{equation}\label{Eq:localZ}
	 	\begin{array}{ccccc}
	 		(\CC^2/\mu_m,[0]) & \longrightarrow & X_{m,q} & \longrightarrow & (\CC^2,0)\\[.1em]
	 		\left[(u^m,v^m)\right] & \longmapsto & (u,v,u^{m-q}v) & \longmapsto & (u^a,v^b)
	 	\end{array}
	 \end{equation}
	  for some $a,b>0$, see e.g.~\cite[p.~103]{CompactComplexSurfaces:book} and~\cite[p.~40]{Nemethi22:book}. Here $X_{m,q}$ is the surface defined in Remark~\ref{rmk:Anq}.
	  	
	 Finally, it is clear by construction that $Z$ is a $V$-manifold and that $\pi^{-1} (M)$ has $\QQ$-normal crossing. Since the restriction of $\ol{\pi}$ to $\ol{X}\setminus\ol{\pi}^{-1}(\ol{M})$ is an isomorphism, it turns out that $\pi|_{\pi^{-1}(M)}$ is also an isomorphism, and hence $\pi$ is an embedded $\QQ$-resolution of $M$.
	\end{proof}
	
	\subsection{Proof of Theorem~\ref{Thm:Comparison_Zmot}}\label{sec:proof-main-thm1}
	
	We must show that $\P^{\mot}_{(\ol{S},\ol{o})}(\ol{D},\ol{W}) \subset \P^{\mot}_{(S, o)}(D,W)$, and that the order of the poles is preserved. 
	Consider the commutative diagram constructed in~\eqref{eq:diagramQresS}. Recall that
	\begin{itemize}
		\item $\ol{\pi} \colon \ol{X}\to (\ol{S},\ol{o})$ is an embedded resolution of $\ol{M}=\supp(\ol{D})\cup\supp(\ol{W})\cup\supp(B_\rho)$,
		
		\item $\pi \colon Z\to (S,o)$ is an embedded $\QQ$-resolution of $M= \rho^{-1}(\ol{M})$,  %=\supp(D)\cup\supp(W)$,
		
		\item $\wt{\rho} \colon Z \to \ol{X}$ is a finite morphism, induced from  $\rho \colon S \to \ol{S}$.
	\end{itemize}
	Also, let $\{E_j\}_{j\in J}$ (resp.~$\{\ol{E}_i\}_{i\in I}$) be the irreducible components (of the exceptional divisor and the  strict transform) of $\pi$ (resp.~$\ol{\pi}$), and denote by $(N_j,\nu_j)$ (resp.~$(\ol{N}_i,\ol{\nu}_i)$) the {numerical data} of $E_j$ (resp.~$\ol{E}_i$) for $j\in J$ (resp.~$i\in I$).
	
	Assume that $s_0\in\QQ$ is a pole of $\Zmott{(\ol{S},\ol{o})}(\ol{D},\ol{W})$. First, if $s_0$ is a pole of order 2, then there exist two intersecting divisors $\ol{E}$ and $\ol{E}'$ with numerical data $(\ol{N},\ol{\nu})$ and $(\ol{N}',\ol{\nu}')$, such that $s_0 = -\ol{\nu}/\ol{N} = - \ol{\nu}'/\ol{N}'$. Take two intersecting components $E$ and $E'$ projecting on $\ol{E}$ and $\ol{E}'$, with numerical data $(N,\nu)$ and $(N',\nu')$, respectively. By Theorem~\ref{thm:MultAlphas_proportional}, one has  $\nu/N = \ol{\nu}/\ol{N}$ and $\nu'/N' = \ol{\nu}'/\ol{N}'$. Thus, $s_0$ is a candidate pole of order 2 of $\Zmott{(S,o)}(D,W)$, and by Proposition~\ref{RVmot}(1) 
	it is a true pole of order 2.
	
	Suppose now that $s_0\in\QQ$ is a pole of order 1 of $\Zmott{(\ol{S},\ol{o})}(\ol{D},\ol{W})$. By Theorem~\ref{Thm:RodVeysQres}, $s_0=-\ol{\nu}_i/\ol{N}_i$,  where $(\ol{N}_i,\ol{\nu}_i)$ comes from one of the following situations:
	\begin{enumerate}[label=(\roman*)]
		\item from some irreducible component $\ol{E}_i$ in the set $\wh{\ol{D}}\cap\wh{\ol{W}}$ or in the set $\wh{\ol{D}}\setminus\wh{\ol{W}}$ (in which case $\ol{\nu}_i=1$), or
		\item from some non-rational exceptional curve $\ol{E}_i$, or
		\item[(iv)] from some rational exceptional curve $\ol{E}_i$, intersecting other components $\ol{E}_\ell$ in points $Q_{\ell}$, with associated $\ol{\alpha}_{Q_\ell}:=\ol{\alpha}_\ell\neq 1$, at least three times.
	\end{enumerate}	
 
\noindent
{\em Case} (i). By commutativity, there exists a component $E_j$ of the strict transform of $D$ in $X$ projecting onto $\ol{E}_i$, with numerical data $(N_j,\nu_j)$, such that $s_0=-\nu_j/N_j$. Note that $E_i$ cannot be a component in $\wh{W}\setminus\wh{D}$, since in particular $N_i\neq 0$. Applying Theorem~\ref{Thm:RodVeysQres}(i) in this case, $s_0$ is a pole of order 1 of $\Zmott{(S,o)}(D,W)$. %
	
\noindent
{\em Case} (ii). 
By Hurwitz's formula,
	$g(\ol{E})\leq g(E)$, implying that  $s_0$ is a pole in the source by Theorem~\ref{Thm:RodVeysQres}(ii). 

	\noindent
{\em Case} (iv).
Assume thus that $\ol{E}_i$ is a rupture component. 
	Since $s_0$ is not a double pole of $\Zmott{(\ol{S},\ol{o})}(\ol{D},\ol{W})$, we have that $\ol{\alpha}_\ell\neq 0$ for any $Q_\ell\in\ol{E}_i$. Now fix  an exceptional component $E_j$ in $Z$ with numerical data $(N_j,\nu_j)$ mapping onto $\ol{E}_i$, with thus  $-\nu_j/N_j=s_0$. We can assume that $E_j$ is rational, otherwise we are in case (ii). From Theorem~\ref{thm:MultAlphas_proportional}(2), it follows that $\alpha_k\neq 0$ for any point $P_k\in\P_{\pi}$ such that $\wt{\rho}(P_k)=Q_\ell$. Therefore, we are within the assumptions of Theorem~\ref{thm:AtMostTwoAlphas}, so 	$E_j$ is a rupture component 	and then Theorem~\ref{Thm:RodVeysQres}\ref{ResPart4} implies that $s_0$ is a pole of order 1 of $\Zmott{(S,o)}(D,W)$. \hfill $\square$
	
 	\medskip
 	
The following example shows that the inclusion $\P^{\mot}_{(S, o)}(D,W)\subseteq \P^{\mot}_{(\ol{S},\ol{o})}(\ol{D},\ol{W})$ 
is not true in general.

\begin{ex}\label{ex:noPoleDownstairs}
	Let $S=\CC^2$, $\ol{S}=X(2;1,1)$, and let $\rho\colon \CC^2 \twoheadrightarrow X(2;1,1)$ be the natural covering, as in Setting~\ref{setting2}. Consider the divisors in $\CC^2$, given by \[D\colon x^4+y^{6}=0 \quad\text{and}\quad W\colon x^{4-1}=0.
	\]
	By using a $(3,2)$-blow-up as in Section~\ref{sec:embQres}, we obtain an embedded $\QQ$-resolution of $(D,W)$. There is just one exceptional
	component $E$, with numerical data $(N,\nu) = (12,14)$.  Moreover, $E$ contains two
	singular points $P_1$ and $P_2$ of orders $3$ and $2$, respectively. The strict transform of $D$ intersects $E$ in two smooth points $\set{P_3,P_4}$ of the ambient space, while the strict transform of $W$ intersects $E$ in $P_2$, see Figure~\ref{fig:noPoleDownstairs}.
	The associated $\alpha$-values at each point are
	\[
	\alpha_1 = \frac13,\quad \alpha_2 = \frac{4}{2}=2,\quad\text{and}\quad \alpha_3=\alpha_4=-\frac{1}{6},
	\]
	thus $s_0=-\frac76$ is a pole of $\Zmott{(\CC^2,0)}(D,W;s)$ of order 1.
	
	Since $D$ and $W$ induce well-defined divisors $\ol{D}$ and $\ol{W}$ in $X(2;1,1)$, one may follow the procedure described in~\eqref{eq:chartC2G} to construct an embedded $\QQ$-resolution of $(\ol{D}, \ol{W})$ in $X(2;1,1)$. In that case the exceptional component contains two cyclic quotient singularities of orders $6$ and $4$.
	
	\begin{figure}[ht]
	\centering
	\definecolor{bluegreen2}{RGB}{0, 85, 127}
	\begin{tikzpicture}[scale=1.1]
	\colorlet{col1}{blue!40}
	\colorlet{col2}{red!40}
	\colorlet{col3}{violet!40}
	\colorlet{colsing}{bluegreen2}
	\colorlet{colsingtext}{colsing}
	\def\xmove{100}
	\begin{scope}[xshift=-\xmove]
	\draw[very thick] (-1.75,0)--(1.75,0) node [right, xshift=3pt] {$E \, (12,14)$};
	\node[color=colsing] (O1) at (-1.2,0) {\Large $\bullet$};
	\node[below, yshift=-3pt, text=colsingtext] at (O1) {$(3)$};
	\node (X1) at (-.5,0) {$\times$};
	\node (X2) at (.5,0) {$\times$};
	\draw[-open triangle 45] (0.5, -0.75) -- (0.5,1);
	\draw[-open triangle 45] (-0.5, -0.75) -- (-0.5,1);
	\draw[decoration={brace,raise=1pt},decorate] (-0.6,1) -- node[above=5pt] {$\widehat{D}$} (0.6,1);
	\node[color=colsing] (O2) at (1.2,0) {\Large $\bullet$};
	\draw[-open triangle 45, dashed] (1.2, -0.75) -- (1.2,1);
	\node[below, xshift=10pt, yshift=-3pt, text=colsingtext] at (O2) {$(2)$};
	\node[above, xshift=10pt, yshift=20pt] at (O2) {$\widehat{W}$};
	\node[above] at (O1) {\footnotesize $P_1$}; \node[above right] at (O2) {\footnotesize $P_2$};
	\node[above right] at (X1) {\footnotesize $P_3$}; \node[above right] at (X2) {\footnotesize $P_4$};
	\end{scope}
	\draw[->] (.2,0.5) --  (1.2,0.5) node[midway,above] {\small $2:1$};	
	\begin{scope}[xshift=\xmove]
	\draw[very thick] (-1.75,0)--(1.75,0) node [right, xshift=3pt] {$\overline{E} \, (12,14)$};
	\node[color=colsing] (O1) at (-1.2,0) {\Large $\bullet$};
	\node[below, yshift=-3pt, text=colsingtext] at (O1) {$(6)$};
	\node (X1) at (0,0) {$\times$};
	\draw[-open triangle 45] (0, -0.75) -- (0,1);
	\node[above, xshift=10pt, yshift=20pt] at (X1) {$\widehat{\overline{D}}$};
	\node[color=colsing] (O2) at (1.2,0) {\Large $\bullet$};
	\draw[-open triangle 45, dashed] (1.2, -0.75) -- (1.2,1);
	\node[below, xshift=10pt, yshift=-3pt, text=colsingtext] at (O2) {$(4)$};
	\node[above, xshift=10pt, yshift=20pt] at (O2) {$\widehat{\overline{W}}$};
	\node[above] at (O1) {\footnotesize $Q_1$}; \node[above right] at (O2) {\footnotesize $Q_2$};
	\node[above] at (X1) {\footnotesize $Q_3=Q_4$};
	\end{scope}
	\end{tikzpicture}
	\caption{Two embedded $\QQ$-resolutions: of $D\colon x^4+y^{6}=0$ and $W\colon x^{4-1}=0$ in $\CC^2$, and of their images in the quotient $\CC^2\to X(2;1,1)$.}
	\label{fig:noPoleDownstairs}
	\end{figure}
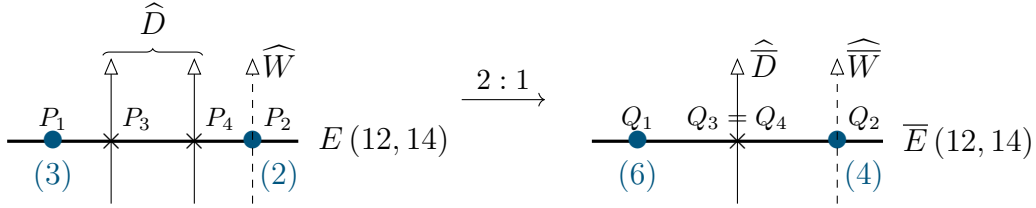
	
	In the target, we have that $\ol{\alpha}_1=\frac16$, $\ol{\alpha}_2 = 1$ and $\ol{\alpha}_3 = -\frac16$ are the associated $\alpha$-values of the points $Q_1$, $Q_2$ and $Q_3$, projections of $P_1$, $P_2$ and $\set{P_3,P_4}$, respectively. We conclude that $s_0$ is a pole for the source, but not a pole for the target $X(2;1,1)$.
	Note that this phenomenon remains true on the topological level. More precisely, the topological zeta functions are
	\begin{equation}\label{ex:ztop-C2-C2G}
		\Ztopp{(\CC^2,0)}(D,W;s)=\frac{3s+7}{4(s+1)(6s+7)} \quad \text{and} \quad
		\Ztopp{(X(2;1,1),[0])}(\ol{D},\ol{W};s)=\frac{1}{2(s+1)}.
	\end{equation}
	This example illustrates the geometric situation described in Proposition~\ref{prop:losingPoles}(3)--(b), when the condition of being a rupture component is not preserved.
\end{ex}

	% =========================================================
	% On the level of topological zeta functions
	% =========================================================
	\bigskip
	\section{On topological zeta functions}\label{sec:Ztop}

As recalled in~\eqref{eq:top-mot}, since the topological zeta function is a specialization of the motivic one via the Euler characteristic, one has 
\begin{equation*}
	\P^{\topo}_{(S,o)}(D,W)
	\subset \P^{\mot}_{(S,o)}(D,W).
\end{equation*}
Theorem~\ref{Thm:Comparison_Zmot} does not automatically provide the similar relation for the topological
zeta function in the context of finite morphisms, since in general the inclusion above is strict, see the first part of Example~\ref{ex:noPoleUpstairs}. 
In fact, on the level of the topological zeta function, neither the relation $\P^{\topo}_{(\ol{S},\ol{o})}(\ol{D},\ol{W}) \subset \P^{\topo}_{(S, o)}(D,W)$ nor  $\P^{\topo}_{(\ol{S},\ol{o})}(\ol{D},\ol{W}) \supset \P^{\topo}_{(S, o)}(D,W)$ holds,
as evidenced by Example~\ref{ex:noPoleUpstairs} and~\eqref{ex:ztop-C2-C2G} in Example~\ref{ex:noPoleDownstairs}, respectively.

	\begin{ex}\label{ex:noPoleUpstairs}
	Let $S=\CC^2$, $\ol{S}=X(2;1,1)$, and let $\rho\colon \CC^2 \twoheadrightarrow X(2;1,1)$ be the natural covering, as in Setting~\ref{setting2}. Consider the divisors in $\CC^2$  given by
	\[D\colon x^4+y^{10}=0 \quad\text{and}\quad W\colon x^{6-1}=0.
	\]
By using a $(5,2)$-blow-up as in Section~\ref{sec:embQres}, we obtain an embedded $\QQ$-resolution of $(D,W)$. 
	The associated $\alpha$-values at each point on the unique exceptional component, see the left-hand side of Figure~\ref{fig:noPoleUpstairs}, are
	\[
	\alpha_1 = \frac15,\quad \alpha_2 = \frac{6}{2}=3,\quad\text{and}\quad \alpha_3=\alpha_4=-\frac{3}{5}.
	\]
	Thus, $s_0=-\frac85$ is a pole of $\Zmott{(\CC^2,0)}(D,W;s)$ of order 1, by Theorem~\ref{Thm:RodVeysQres}\ref{ResPart4}.
	However, from~\eqref{Eqn:ZtopQres} and some computations, we obtain that
	\[
	\Ztopp{(\CC^2,0)}(D,W;s)=\frac{1}{6(s+1)}.
	\]		
Constructing  an embedded $\QQ$-resolution of $(\ol{D}, \ol{W})$ 	via the procedure described in~\eqref{eq:chartC2G},
again~\eqref{Eqn:ZtopQres} yields
	\[
	\Ztopp{(X(2;1,1),[0])}(\ol{D},\ol{W};s)=\frac{29s+32}{12(s+1)(5s+8)},
	\]
	showing that $-\frac{8}{5}$ is a pole of the latter, but not a pole of $\Ztopp{(\CC^2,0)}(D,W;s)$. An easy computation gives 
	$  \ol{\alpha}_1 = \frac{1}{10}$, $\ol{\alpha}_2 = \frac{3}{2}$,  and  $\ol{\alpha}_3=-\frac{3}{5}$
	with respect to $\ol{E}$, see the right-hand side of Figure~\ref{fig:noPoleUpstairs}.
	
	\begin{figure}[ht]
	\centering
	\definecolor{bluegreen2}{RGB}{0, 85, 127}
	\begin{tikzpicture}[scale=1.1]
	\colorlet{col1}{blue!40}
	\colorlet{col2}{red!40}
	\colorlet{col3}{violet!40}
	\colorlet{colsing}{bluegreen2}
	\colorlet{colsingtext}{colsing}
	\def\xmove{100}
	\begin{scope}[xshift=-\xmove]
	\draw[very thick] (-1.75,0)--(1.75,0) node [right, xshift=3pt] {$E \, (20,32)$};
	\node[color=colsing] (O1) at (-1.2,0) {\Large $\bullet$};
	\node[below, yshift=-3pt, text=colsingtext] at (O1) {$(5)$};
	\node (X1) at (-.5,0) {$\times$};
	\node (X2) at (.5,0) {$\times$};
	\draw[-open triangle 45] (0.5, -0.75) -- (0.5,1);
	\draw[-open triangle 45] (-0.5, -0.75) -- (-0.5,1);
	\draw[decoration={brace,raise=1pt},decorate] (-0.6,1) -- node[above=5pt] {$\widehat{D}$} (0.6,1);
	\node[color=colsing] (O2) at (1.2,0) {\Large $\bullet$};
	\draw[-open triangle 45, dashed] (1.2, -0.75) -- (1.2,1);
	\node[below, xshift=10pt, yshift=-3pt, text=colsingtext] at (O2) {$(2)$};
	\node[above, xshift=10pt, yshift=20pt] at (O2) {$\widehat{W}$};
	\node[above] at (O1) {\footnotesize $P_1$}; \node[above right] at (O2) {\footnotesize $P_2$};
	\node[above right] at (X1) {\footnotesize $P_3$}; \node[above right] at (X2) {\footnotesize $P_4$};
	\end{scope}
	\draw[->] (.2,0.5) --  (1.2,0.5) node[midway,above] {\small $2:1$};
	\begin{scope}[xshift=\xmove]
	\draw[very thick] (-1.75,0)--(1.75,0) node [right, xshift=3pt] {$\overline{E} \, (20,32)$};
	\node[color=colsing] (O1) at (-1.2,0) {\Large $\bullet$};
	\node[below, yshift=-3pt, text=colsingtext] at (O1) {$(10)$};
	\node (X1) at (0,0) {$\times$};
	\draw[-open triangle 45] (0, -0.75) -- (0,1);
	\node[above, xshift=10pt, yshift=20pt] at (X1) {$\widehat{\overline{D}}$};
	\node[color=colsing] (O2) at (1.2,0) {\Large $\bullet$};
	\draw[-open triangle 45, dashed] (1.2, -0.75) -- (1.2,1);
	\node[below, xshift=10pt, yshift=-3pt, text=colsingtext] at (O2) {$(4)$};
	\node[above, xshift=10pt, yshift=20pt] at (O2) {$\widehat{\overline{W}}$};
	\node[above] at (O1) {\footnotesize $Q_1$}; \node[above right] at (O2) {\footnotesize $Q_2$};
	\node[above ] at (X1) {\footnotesize $Q_3=Q_4$};
	\end{scope}
	\end{tikzpicture}
	\caption{Embedded $\QQ$-resolution of  $D\colon x^4+y^{10}=0$ and $W\colon x^{6-1}=0$ in~$\CC^2$, and of their images in the quotient  $\CC^2\to X(2;1,1)$.}
	\label{fig:noPoleUpstairs}
	\end{figure}
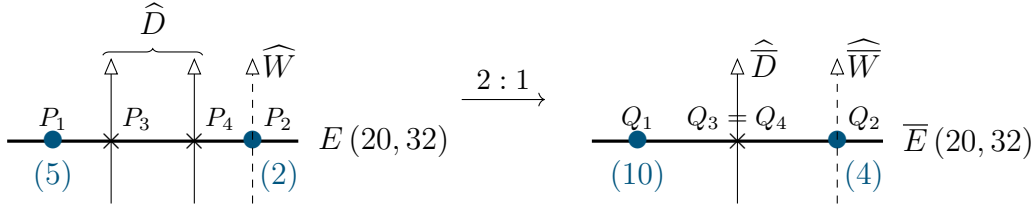
\end{ex}

The main purpose of this section is to study the relations among the poles of topological zeta functions when $W=0$. Moreover, they also produce a useful relation at the level of the motivic zeta function, see Theorem~\ref{Thm:Comparison_Ztop}. For a better presentation, we split the proof into three different sections as follows. Let $\pi\colon X\to\CC^2$ be the minimal embedded resolution of $(\CC^2,D)$. We first construct in Section~\ref{sec:minQresInQuotients} an embedded $\QQ$-resolution $\ol{\pi} \colon \ol{X}\to \CC^2/G$ to obtain a commutative diagram as in~\eqref{eq:diagramQresS}. It turns out that this $\ol{\pi}$ verifies some nice properties that are crucial for our purpose. In particular, it verifies Loeser's $\alpha$-condition, that is,  $-1\leq \ol{\alpha}_Q< 1$ holds for 
any point $Q$ in the exceptional part of $\ol{\pi}$, see Proposition~\ref{Prop:AlphasAbs1}. In Section~\ref{Sec:StructureThm}, we prove some properties of a decorated graph of $\ol{\pi}$, that allow one to control the $\alpha$-values of the components contributing to poles of the zeta function. This is very useful in Section~\ref{subsec:CharPolesZtop}, when studying the contribution of an exceptional $\ol{E}$ to the residue of $s_0=-\ol{\nu}/\ol{N}$, defined by 
\[
\R^e:=\frac{1}{\ol{N}}\left(\chi(\ol{E}^\circ) + \sum_{i=1}^k\frac{1}{\ol{\alpha}_i}\right),
\]
which is the key point to end the proof of Theorem~\ref{Thm:Comparison_Ztop} in Section~\ref{subsec:ProofPolesZtop}.

	\subsection{Construction of \texorpdfstring{$\ol{\pi}$}{pi}}\label{sec:minQresInQuotients}
	Let $S,\ol{S}$ and $\rho$ be as in Setting~\ref{setting2}. Fix  a divisor $D\subset \CC^2$, which is invariant under the action of $\mu_{d}$, and consider the associated Weil divisor $\ol{D}\subset \CC^2/\mu_d$,  verifying $D = \rho^* \ol{D}$. 
	Let $\pi\colon X\to\CC^2$ be the minimal embedded resolution of $(D,0)$, and assume that $\pi$ is not the identity. The case $\pi=\id$ was treated in Section~\ref{sec:BadCases}, see Remark~\ref{rmk:BpNeverTangent}(i). 

The action of $\mu_d$ on $\CC^2$ can be extended to every blow-up as in Section~\ref{sec:embQres}, so $\mu_d$ acts linearly also on $X$. Taking $\wt{\rho}\colon X \twoheadrightarrow \ol{X}:=X/\mu_d$, the induced quotient map associated to such an extended action, there exists a proper birational morphism $\ol{\pi} \colon \ol{X}\to \CC^2/\mu_d$ making the following diagram commute.
	\begin{equation}\label{Eq:diagram_CC2G}
	  \begin{tikzcd}
		\CC^2 \arrow[d, twoheadrightarrow, "\rho"] 	& \arrow[l, "\pi" above] X \arrow[d, dashed, twoheadrightarrow, "\wt{\rho}"]\\
		\CC^2/\mu_d & \arrow[l, dashed, "\ol{\pi}" above] X/\mu_d=:\ol{X}
	\end{tikzcd}
	\end{equation}
	Our first task is to show that this $\ol{\pi}$
verifies an analogue of the classical result stated in Proposition~\ref{lemma:alphas1}. 
	
\begin{prop}\label{Prop:AlphasAbs1}
	The morphism $\ol{\pi}$ is an embedded $\QQ$-resolution of $(\ol{D},-B_\rho)$. Moreover, for any exceptional component $\ol{E}$ with numerical data $(\ol{N},\ol{\nu})$, and any point $Q_i\in \ol{E}\cap\P_{\ol{\pi}}$,
one has that $$-1\leq {\alpha}_{Q_i}<1,$$ 
	where $\P_{\ol{\pi}}$  and ${\alpha}_{Q_i}$ were introduced in~\eqref{Eq:P_pi} and Definition~\ref{Def:Alpha_Qres}, respectively.
	The equality holds if and only if there exists a unique component $\ol{E}_i$ such that $\ol{E}\cap\ol{E}_i=\ol{E}\cap\P_{\ol{\pi}}$. 
\end{prop}

\begin{proof}
	According to Remark~\ref{rmk:BpNeverTangent}(ii), the first part is a consequence of the fact that $\supp (\pi^*D)\cup\supp (\pi^*\Ram_\rho)$ is a normal crossing divisor on $X$. Indeed, $\wt{\rho}$ is locally the quotient of a normal crossing divisor on $\CC^2$ by the action of $\mu_d$ and, in particular, irreducible components of the strict transforms of $D$ and $\Ram_\rho$ cannot meet on an exceptional divisor $E$, unless $D$ and $\Ram_\rho$ share common components. 
	Note that $\ol{E}=\PP^1/\mu_{d}\cong\PP^1$. 

 To show the second part of the statement, let $\ol{E}$ be an exceptional component  with numerical data $(\ol{N},\ol{\nu})$, and take $E$ in $X$ with numerical data $(N,\nu)$ projecting onto $\ol{E}$. Fix $Q_i\in\P_{\ol{\pi}}\cap\ol{E}$ of order $\ol{m}_i$. We distinguish three cases.
	\begin{enumerate}[label=(\roman*)]
		\item There exists another component $\ol{E}_i$, either exceptional or in $\wh{\ol{D}}$, intersecting  $\ol{E}$ at $Q_i$ with numerical data $(\ol{N}_i,\ol{\nu}_i)$. Take then $E_i$ on $X$ intersecting $E$, with numerical data $(N_i,\nu_i)$, which projects onto $\ol{E}_i$. Likewise, denote by $P_i$ the intersection point $E\cap E_i$. In order to simplify notation in the rest of the proof, we denote 
		$$\alpha_i:=\alpha_{P_i}\quad \text{and}\quad \ol{\alpha}_i:=\alpha_{Q_i}.$$ 
Here  $\alpha_{i}$ is the $\alpha$-value (associated to $D, W=0, \pi$) of $P_i$ with respect to $E_i$. Then, by Proposition~\ref{lemma:alphas1}, the value $\alpha_{_i}=\nu_i-(\nu/N)N_i$ verifies $-1\leq\alpha_{i}<1$, and, since $\ol{\alpha}_{i}=\alpha_{i}/n$ for some $n\geq 1$ by Theorem~\ref{thm:MultAlphas_proportional}(2), it follows that also $-1\leq \ol{\alpha}_{i}<1$.
		
		\item The point $Q_i$ belongs to a branch curve $\wh{L}_k$ in $\wh{B}_\rho$, that is not a component of $\wh{\ol{D}}$. Then $\ol{\alpha}_{i} = \frac{1}{e_k\cdot \ol{m}_i}$ is strictly smaller than $1$, since $\rho$ ramifies over the line $L_k$.
		
		\item The point $Q_i\in\Sing\ol{X}$ is of order $\ol{m}_i>1$, and we are not in the cases above. Take $P_i\in E$ projecting on $Q_i$; it is clear that $\alpha_i=1$, since $X$ is smooth, and then we have by Definition~\ref{Def:Alpha_Qres} that $\ol{\alpha}_{i}=1/\ol{m}_i<1$.
	\end{enumerate}
	Finally, suppose that $\ol{E}\cap\P_{\ol{\pi}} = \set{Q_1,\ldots, Q_k}$. By the equality $\sum_{i=1}^k \ol{\alpha}_i = k-2$ from~\eqref{Eq:AlphaArithm}, we have that $\ol{\alpha}_1=-1$ if and only if $k=1$.
\end{proof}

\begin{remark}\label{rmk:minres_also_works_forW}
  For $W\neq 0$, the same construction of diagram~\eqref{Eq:diagram_CC2G} works well when one fixes the minimal embedded resolution $\pi\colon X\to\CC^2$ of $(D,W)$. If $\pi\neq\id$, then $\ol{\pi}\colon\ol{X}\to\CC^2/\mu_d$ is an embedded $\QQ$-resolution  of $(\ol{D},\ol{W})$. 
  In particular, $\supp(\pi^*D)\cup\supp (\pi^*W)\cup\supp (\pi^*\Ram_\rho)$ is still a normal crossing divisor on $X$.
  However, in this more general situation, the related $\alpha$-values, in the source or the target, may not be bounded by 1, see Example~\ref{ex:noPoleDownstairs}.%
\end{remark}

In Section~\ref{Sec:StructureThm} we will study graphs of embedding $\QQ$-resolutions whose exceptional components satisfy the bound of Proposition~\ref{Prop:AlphasAbs1}. This motivates Definition~\ref{def:alpha_condition}, providing $\ol{\pi}$ as a useful example of this more general situation.
	
	\begin{defi}\label{def:alpha_condition}
		 Let $(S,o)$ and $(D,W)$ be as in Setting~\ref{setting1}.
	  We say that an embedded $\QQ$-resolution $h:X\to S$ of $\supp(D) \cup \supp(W)$, with irreducible components $E_i, i\in I,$ of $h^{-1}(\supp(D) \cup \supp(W))$, verifies the \emph{$\alpha$-condition} %
	  if the following holds:
	  \begin{enumerate}%
	    \item\label{Alphacond1} 
	    for any component $E$ with numerical data $(N,\nu)$, one has that $\nu>0$,
	
	    \item\label{Alphacond2} 
	    for any exceptional component $E$ with numerical data $(N,\nu)$, and any point $Q_i\in E\cap\P_{h}$ of order $m_i$ on $X$, one has that $\alpha_{Q_i}:=\frac{\nu_i-(\nu/N)N_i}{m_i}<1$.
	  \end{enumerate}
	\end{defi}
	
\subsection{\texorpdfstring{$\QQ$}{Q}-resolutions verifying the \texorpdfstring{$\alpha$}{a}-condition and Eisenbud-Neumann graphs} 
\label{Sec:StructureThm}
Let us consider $(S,o)$ and $(D,W)$ as in Setting~\ref{setting1}. Assume that $(D,W)$ admits an embedded $\QQ$-resolution  $h\colon X \to S$, 
verifying Definition~\ref{def:alpha_condition}.  Denote by  $\set{E_i}_{i\in I}$ the set of exceptional and strict transform components. In the following result, we collect some geometric properties of such embedded $\QQ$-resolutions verifying the $\alpha$-condition. 

\begin{prop}\label{prop:AlphasPropsEi}
	If $E$ is an exceptional component  of $h$ with numerical data $(N,\nu)$, and $E\cap\P_{h}=\set{Q_1,\ldots,Q_k}$, then
	\begin{enumerate}
		\item $E$ is rational;
		\item when $k\geq 2$ (resp.~$k=1$), we have that all $\alpha_i >-1$ (resp.~$\alpha_1=-1$);
		\item at most one $Q_i$ occurs with $\alpha_i<0$;
		\item if $k\geq 3$, then at most one $Q_i$ occurs with $\alpha_i\leq 0$.
		% 	    \item If $k=2$, then $\ol{\alpha}_1 = - \ol{\alpha}_2$.
	\end{enumerate}
\end{prop}
\begin{proof}
	These relations follow from part~\eqref{Alphacond2} of the $\alpha$-condition, together with the arithmetic relation~\eqref{Eq:AlphaArithm}. For this it is enough to use the same arguments as in~\cite[Corollary~2.7]{Veys95} where the classical case was treated.
\end{proof}
In order to prove Theorem~\ref{Thm:Comparison_Ztop}, we will interpret the previous result in terms of a resolution graph associated to $h$. More precisely, one constructs an \lq extended weighted Eisenbud-Neumann  decorated graph\rq\ $\Gamma_h$ as 
follows. 
	\begin{itemize}[label=$-$]
	  \item Associate to each exceptional component a vertex ($\bullet$), and to each irreducible component of the strict transform  of $D$ or $W$ an arrow ($\rightarrow$).
	
	  \item Decorate each vertex or arrow representing a component $E_i$, $i\in I$, with the ratio $\nu_i/N_i$, where we define this quotient to be $\infty$ if $N_i=0$.
	
	  \item Each point $P_k\in\P_{h}$ corresponds to an edge in the graph as follows.
	  \begin{enumerate}
	    \item[(i)] If $P_k$ is the intersection of two components $E_i$ and $E_j$, then the edge connects the respective vertices (i.e., \raisebox{-.3em}{\tikz{\node (P1) at (-.5,0) {$\bullet$}; \node (P2) at (.5,0) {$\bullet$}; \draw (P1.center) -- (P2.center);}} or  \raisebox{-.3em}{\tikz{\node (P1) at (-.5,0) {$\bullet$}; \node (P2) at (.5,0) {}; \draw[->] (P1.center) -- (P2.center);}}). The arrow will be dashed when one of the components belongs to $\wh{W}\setminus\wh{D}$ (i.e., \raisebox{-.3em}{\tikz{\node (P1) at (-.5,0) {$\bullet$}; \node (P2) at (.5,0) {}; \draw[dashed, ->] (P1.center) -- (P2.center);}}).
	
	    \item[(ii)] If $P_k$ lies only on one component $E_i$, that is, $P_k$ is singular, a dashed edge is drawn towards a new open vertex decorated with $\infty$ (i.e., \raisebox{-.3em}{\tikz{\node (P1) at (-.5,0) {$\bullet$}; \node (P2) at (.5,0) {$\circ$}; \draw[dashed] (P1.center) -- (P2.center); \node[right] at (P2) {\small $\infty\!$};}}).
	  \end{enumerate}
	
	  \item Decorate each edge with the order $m_k$ of the corresponding point $P_k$.
	\end{itemize}
Note that $\Gamma_h$ is connected.
	
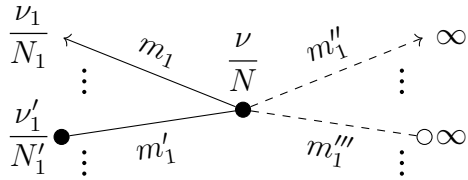
\begin{figure}[ht]
\begin{tikzpicture}[scale=1.2]
	\node[draw, circle, fill=black, inner sep=2pt] (C) at (0,0) {};
	\node[inner sep=0pt] (L1) at (-2,.8) {};
	\node[draw, circle, fill=black, inner sep=2pt] (L2) at (-2,-0.3) {};
	\draw[->] (C) -- (L1) node[sloped, midway, above] {\(m_1\)};
	\draw (C) -- (L2) node[sloped, midway, below] {\(m_1'\)};
	\node[inner sep=0pt] (R1) at (2,.8) {};
	\node[draw, circle, fill=white, inner sep=2pt] (R2) at (2,-.3) {};
	\draw[dashed,->] (C) -- (R1) node[sloped, midway, above] {\(m_1''\)};
	\draw[dashed] (C) -- (R2) node[sloped, midway, below] {\(m_1'''\)};
	\node at (-1.75,.4) {\LARGE $\vdots$};
	\node at (-1.75,-.5) {\LARGE $\vdots$};
	\node at (1.75,.4) {\LARGE $\vdots$};
	\node at (1.75,-.5) {\LARGE $\vdots$};
	\node[left] at (L1) {$\dfrac{\nu_1}{N_1}$}; \node[left] at (L2) {$\dfrac{\nu_1'}{N_1'}$}; 
	\node[right] at (R1) {$\infty$}; \node[right] at (R2) {$\infty$};
	\node[above,yshift=3] at (C) {$\dfrac{\nu}{N}$};
\end{tikzpicture}
\caption{Decorated subgraph centered at an exceptional vertex.}
\end{figure}
	
	\begin{remark}\mbox{}
	  \begin{enumerate}
		\item The convention in case (ii) above is motivated by the fact that we can picture a curvette passing through such a point $P_k$,  forming together with $E_i$ a $\QQ$-normal crossing divisor. 
		Recall also that a component in $\wh{W}\setminus\wh{D}$ has numerical data $(0, 1/e_k)$, while the curvette has data $(0,1)$. That is the reason to use $\infty$ as decoration.
	    \item  Assuming $h\neq\id$, each irreducible component of the strict transform of $D$ and $W$ intersects a unique exceptional component $E_i$, and then its  corresponding vertex is an  end vertex of the graph. 
	  \end{enumerate}
	\end{remark}
	
	Now fix an exceptional component $E$ with numerical data $(N,\nu)$, and consider another component $E_i$, exceptional or in $\wh{D}$, with numerical data $(N_i,\nu_i)$, that intersects $E$ in a point $P$. Then, according to Definition~\ref{Def:Alpha_Qres}, 
	 \[
	  \alpha_{P}>0\quad\iff\quad  \frac{\nu_i}{N_i}>\frac{\nu}{N}.
	  \]
	Otherwise, for an irreducible component $E_i$ in $\wh{W}\setminus \wh{D}$, one has that $\alpha_i = \frac{\nu_i}{m_i}\in(0,1)$, since $h$ verifies the $\alpha$-condition. This simple observation allows one to rephrase Proposition~\ref{prop:AlphasPropsEi} in terms of the graph $\Gamma_{h}$.
	
	\begin{prop}\label{prop:propertiesEdges}
	  Using the previous notation, fix an exceptional divisor $E$ with numerical data $(N,\nu)$, whose vertex ($\bullet$) has valency $k$ in $\Gamma_h$. Then we have the following.
		\begin{enumerate}
		\item There exists at most one vertex ($\bullet$) or arrow ($\to$) such that
		\[
			\frac{\nu}{N}\raisebox{-.3em}{\tikz{\node (P1) at (-.5,0) {$\bullet$}; \node (P2) at (.5,0) {$\bullet$}; \draw (P1.center) -- (P2.center);}}\frac{\nu_i}{N_i}
			\quad\text{or}\quad
			\frac{\nu}{N}\raisebox{-.3em}{\tikz{\node (P1) at (-.5,0) {$\bullet$}; \node (P2) at (.5,0) {}; \draw[->] (P1.center) -- (P2.center);}}\frac{\nu_i}{N_i}
			\qquad\text{with}\qquad \frac{\nu_i}{N_i}<\frac{\nu}{N}.
		\]
		\item If $k\geq 3$, there exists at most one vertex ($\bullet$) or ($\to$) as before with $\frac{\nu_i}{N_i}\leq \frac{\nu}{N}.$
		
		\item If $k=2$, then only one of the following can occur:
		\begin{enumerate}
			\item \raisebox{-.7em}{
			\begin{tikzpicture}
				\coordinate (P1) at (-1,0);
				\node (P2) at (0,0) {$\bullet$};
				\coordinate (P3) at (1,0);
				\draw[] (P1) -- (P2.center) node[above] {$\frac{\nu}{N}$} -- (P3);
				\node[draw, fill=white, minimum size=8pt, inner sep=0pt] at (P1) {\tiny 1};
				\node[draw, fill=white, minimum size=8pt, inner sep=0pt] at (P3) {\tiny 2};
				\node[left,xshift=-5] at (P1.west) {$\frac{\nu_1}{N_1}$};
				\node[right,xshift=5] at (P3.east) {$\frac{\nu_2}{N_2}$};
			\end{tikzpicture}
			},
			where each vertex $\scriptsize\boxed{i}$
			is either of type ($\bullet$) or $(\to)$, and
			\[
			\frac{\nu_1}{N_1}<\frac{\nu}{N} \iff \frac{\nu}{N} < \frac{\nu_2}{N_2},\quad\text{as well as}\quad\frac{\nu_1}{N_1}=\frac{\nu}{N} \iff \frac{\nu}{N} = \frac{\nu_2}{N_2}
			\]
			(i.e., ratios either increase or decrease through chains on $\Gamma_{h}$, otherwise they are all equal on the chain);
			
			\item \raisebox{-.7em}{
			\begin{tikzpicture}
				\coordinate (P1) at (-1,0);
				\node (P2) at (0,0) {$\bullet$};
				\coordinate (P3) at (1,0) {};
				\draw[] (P1) -- (P2.center) node[above] {$\frac{\nu}{N}$};
				\draw[dashed] (P2.center) -- (P3) node[right, xshift=2] {$\infty$};
				\node[draw, fill=white, minimum size=8pt, inner sep=0pt] at (P1) {\tiny 1};
				\node[draw, fill=white, minimum size=8pt, inner sep=0pt] at (P3) {\tiny 2};
				\node[left,xshift=-5] at (P1.west) {$\frac{\nu_1}{N_1}$};
			\end{tikzpicture}
			},
			where $\scriptsize\boxed{1}$ is either of type ($\bullet$) or $(\to)$ verifying $\frac{\nu_1}{N_1}<\frac{\nu}{N}$, and $\scriptsize\boxed{2}$ is either of type ($\circ$) or $(\to)$.
		\end{enumerate}
	  \end{enumerate}
	\end{prop}
The main result of this section is the following.
	\begin{prop}\label{prop:orderedTree}
	  Using the previous notation we have the following.
	  \begin{enumerate}[label=(\alph*)]
	    \item The graph $\Gamma_{h}$ is a tree.
	
	    \item Let $\M$ be the maximal subgraph of $\Gamma_{h}$ formed by the vertices with decoration the minimum value ${\nu}/{N}=\min_{i\in I}\set{{\nu_i}/{N_i}}$, and the edges connecting those vertices.
		\begin{enumerate}[label=(\roman*)]
			\item The subgraph $\M$ is connected and it has one of the following forms (with $r\geq0$).
			
			\begin{minipage}[t]{.3\textwidth}
			(1) \raisebox{-2.5em}{\scalebox{0.7}
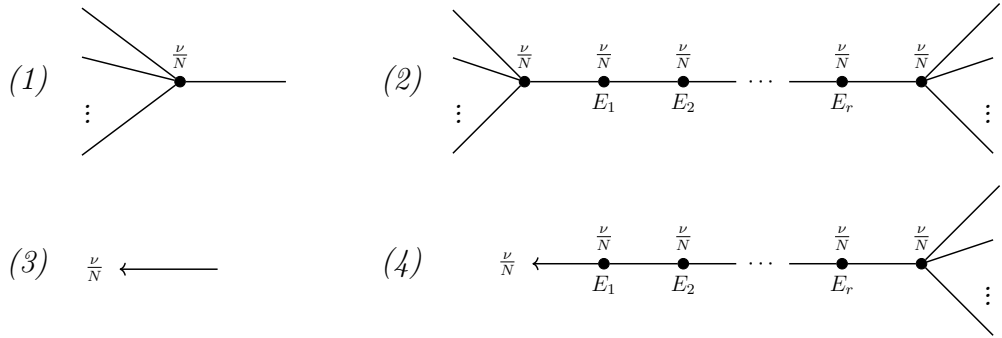
{
				\begin{tikzpicture}
					\node[draw, circle, fill=black, inner sep=2pt] (C) at (0,0) {};
					\node[] (L1) at (-2,1.5) {};
					\node[] (L2) at (-2,0.5) {};
					\node[] (L4) at (-2,-1.5) {};
					\draw[thick] (C) -- (L1) node[sloped, midway, above] {};
					\draw[thick] (C) -- (L2) node[sloped, midway, below] {};
					\draw[thick] (C) -- (L4) node[sloped, midway, below] {};
					\draw[thick] (C) -- (2,0);	
					\node at (-1.75,-.5) {\LARGE $\vdots$};
					\node[above,yshift=3] at (C) {$\frac{\nu}{N}$};
				\end{tikzpicture}
				}}\\[2em]
			
			(3)~~\raisebox{-.5em}{\scalebox{0.7}{
				\begin{tikzpicture}
					\node[] (C) at (0,0) {};
					\draw[thick,<-] (C) -- (2,0) node[sloped, midway, above] {};;
					\node[left] at (C) {$\frac{\nu}{N}$};
				\end{tikzpicture}
				}}
			\end{minipage}
			\begin{minipage}[t]{.5\textwidth}
			(2) \raisebox{-2.5em}{\scalebox{0.7}{
					\begin{tikzpicture}
						\node[draw, circle, fill=black, inner sep=2pt] (C) at (0,0) {};
						\node[draw, circle, fill=black, inner sep=2pt] (C1) at (1.5,0) {};
						\node[draw, circle, fill=black, inner sep=2pt] (C2) at (3,0) {};
						\node[draw, circle, fill=black, inner sep=2pt] (C3) at (6,0) {};
						\node[draw, circle, fill=black, inner sep=2pt] (D) at (7.5,0) {};
						\node[] (L1) at (-1.5,1.5) {};
						\node[] (L2) at (-1.5,0.5) {};
						\node[] (L4) at (-1.5,-1.5) {};
						\draw[thick] (C) -- (L1) node[sloped, midway, above] {};
						\draw[thick] (C) -- (L2) node[sloped, midway, below] {};
						\draw[thick] (C) -- (L4) node[sloped, midway, below] {};
						\draw[thick] (C) -- (C1) node[below,yshift=-2] {$E_1$} node[above,yshift=3] {$\frac{\nu}{N}$} -- (C2) node[below,yshift=-2] {$E_2$} node[above,yshift=3] {$\frac{\nu}{N}$} -- (4,0);
						\node[] at (4.5,0) {$\cdots$};
						\draw[thick] (5,0) -- (C3) node[below,yshift=-2] {$E_r$} node[above,yshift=3] {$\frac{\nu}{N}$} -- (D) node[above,yshift=3] {$\frac{\nu}{N}$};
						\node[inner sep=0pt] (R1) at (9,1.5) {};
						\node[] (R2) at (9,0.5) {};
						\node[] (R4) at (9,-1.5) {};				
						\draw[thick] (D) -- (R1) node[sloped, midway, above] {};
						\draw[thick] (D) -- (R2) node[sloped, midway, below] {};
						\draw[thick] (D) -- (R4) node[sloped, midway, below] {};
						\node at (-1.25,-.5) {\LARGE $\vdots$};
						\node at (8.75,-.5) {\LARGE $\vdots$};
						\node[above,yshift=3] at (C) {$\frac{\nu}{N}$};
					\end{tikzpicture}
				}}\\[.3em]
			
			(4) \raisebox{-2.5em}{\scalebox{0.7}{
					\begin{tikzpicture}
						\node at (-1.5,0) {};
						\node[] (C) at (0,0) {};
						\node[draw, circle, fill=black, inner sep=2pt] (C1) at (1.5,0) {};
						\node[draw, circle, fill=black, inner sep=2pt] (C2) at (3,0) {};
						\node[draw, circle, fill=black, inner sep=2pt] (C3) at (6,0) {};
						\node[draw, circle, fill=black, inner sep=2pt] (D) at (7.5,0) {};				
						\draw[thick,<-] (C) -- (C1) node[sloped, midway, above] {};;
						\draw[thick] (C1) node[below,yshift=-2] {$E_1$} node[above,yshift=3] {$\frac{\nu}{N}$} -- (C2) node[below,yshift=-2] {$E_2$} node[above,yshift=3] {$\frac{\nu}{N}$} -- (4,0);
						\node[] at (4.5,0) {$\cdots$};
						\draw[thick] (5,0) -- (C3) node[below,yshift=-2] {$E_r$} node[above,yshift=3] {$\frac{\nu}{N}$} -- (D) node[above,yshift=3] {$\frac{\nu}{N}$};
						\node[inner sep=0pt] (R1) at (9,1.5) {};
						\node[] (R2) at (9,0.5) {};
						\node[] (R4) at (9,-1.5) {};
						\draw[thick] (D) -- (R1) node[sloped, midway, above] {};
						\draw[thick] (D) -- (R2) node[sloped, midway, below] {};
						\draw[thick] (D) -- (R4) node[sloped, midway, below] {};
						\node at (8.75,-.5) {\LARGE $\vdots$};
						\node[left] at (C) {$\frac{\nu}{N}$};
					\end{tikzpicture}
				}}
			\end{minipage}
			
			\medskip
			
			The ending segments appearing in cases (1), (2) and (4) denote at least two more edges.
	
			\item Starting from any end vertex of $\M$, the ratios $\frac{\nu_i}{N_i}$ strictly increase along any path of the tree away from $\M$.
		\end{enumerate}
	  \end{enumerate}
	\end{prop}
	
	\begin{remark}
		In cases (2) and (4), the vertices $E_1,\ldots, E_r$ could have valency greater than or equal to two. However, 
		in the picture the valency is two for simplicity.
	\end{remark}
	
	\begin{proof}
		For the first part we proceed by contradiction. Assume that $\Gamma_{h}$ contains a cycle of necessarily exceptional vertices $\C\colon v_0\to v_1\to \cdots\to v_n\to v_{n+1}=v_0$.
		Note that not all the vertices in $\C$ could have valency $2$ in $\Gamma_{h}$,  since $\Gamma_{h}$ must contain at least one arrow vertex coming from the strict transform of $D$, and  this would imply that the graph is not connected. %
		Assume that $v_0$ has valency $k\geq 3$ and associated ratio $\frac{\nu_0}{N_0}$. By Proposition~\ref{prop:propertiesEdges}(2), at least $k-1$ of its neighbors $w_j$ have associated ratio $\frac{\nu_j}{N_j}$, verifying $\frac{\nu_0}{N_0}<\frac{\nu_j}{N_j}$. Then, at least one of the vertices $v_{n},v_1\in\C$ should verify this condition, say  $\frac{\nu_0}{N_0}<\frac{\nu_1}{N_1}$. For this reason,   Proposition~\ref{prop:propertiesEdges}(2)(3) forces $v_2$ to have an associated ratio $\frac{\nu_2}{N_2}$ satisfying $\frac{\nu_1}{N_1}<\frac{\nu_2}{N_2}$,  independently of the valency of $v_1$. Repeatedly applying this argument, we obtain that the ratios $\frac{\nu_i}{N_i}$ strictly increase along $\C$, which leads to a contradiction.
		
	  To prove part (b), note that the arguments in~\cite[Theorem~3.3]{Veys95} are formal and can be adapted to this more general context using~Proposition~\ref{prop:propertiesEdges}, once we know that $\Gamma_h$ is a tree.
	\end{proof}

	\subsection{A geometric characterization of poles in the quotient}
	\label{subsec:CharPolesZtop}
	
		We briefly recall the construction of $\ol{\pi}$ in Section~\ref{sec:minQresInQuotients}. Let $S=\CC^2$, $\ol{S}=\CC^2/\mu_d$, and let $\rho\colon \CC^2 \twoheadrightarrow \CC^2/\mu_d$ be the natural covering, as in Setting~\ref{setting2}. Given a divisor $D\subset \CC^2$ which is invariant under the action of $\mu_{d}$, we take $\pi\colon X\to\CC^2$ as the minimal embedded resolution of $(D,0)$. If $\ol{D}\subset \CC^2/\mu_d$ is given by $D = \rho^* \ol{D}$, and $\pi$ is not the identity map, then $\ol{\pi} \colon \ol{X}\to \CC^2/\mu_d$ is a proper birational morphism  that makes the diagram~\eqref{Eq:diagram_CC2G} commutative. 

Fix  an exceptional irreducible component $\ol{E}_i$ in $\ol{X}$ with numerical data $(\ol{N}_i,\ol{\nu}_i)$, and set $s_0=-\ol{\nu}_i/\ol{N}_i$. Suppose that $\ol{E}_i\cap\P_{\ol{\pi}}$ consists of exactly $k$ distinct points $P_j$, of order $\ol{m}_j$ and with associated numerical data $(\ol{N}_j,\ol{\nu}_j)$. Recall from Definition~\ref{Def:Alpha_Qres} that for an irreducible component $\ol{E}_j$, $j\in I$, in $\wh{B}_\rho\setminus \wh{\ol{D}}$, and  $P=\ol{E}_j\cap \ol{E}_i$, we have that  its numerical data are of the form $(0,\nu_\ell)$. Likewise, to $P\in E_i$ singular on $X$ but belonging only to $\ol{E}_i$, we associate formally the numerical data $(0,1)$. In this context  the corresponding $\alpha$-values with respect to $E_i$ are
	\[\ol{\alpha}_{j}:={\alpha}_{P_j}=\frac{\ol{\nu}_j-(\ol{\nu}_i/\ol{N}_i)\ol{N}_j}{\ol{m}_j},
	\]
	and we assume moreover  that $\ol{\alpha}_{j}\neq 0$ for all $j\in\{1,\ldots,k\}$. From~\eqref{Eq:Residue}, the contribution of $\ol{E}_i$ to the residue of $s_0$ in  $\Ztopp{(\CC^2/\mu_d,[0])}(\ol{D},-B_\rho; s)$ is
	\[
	\R_i^e:=\frac{1}{N_i}\left(\chi(\ol{E}_i^\circ) + \sum_{j=1}^k\frac{1}{\ol{\alpha}_j}\right) = \frac{1}{N_i}\left(2-k + \sum_{j=1}^k\frac{1}{\ol{\alpha}_j}\right).
	\]
	Since all exceptional components of $\ol{\pi}$ are rational, the arithmetical formula~\eqref{Eq:AlphaArithm} says that $\sum_{j=1}^k \ol{\alpha}_j=k-2$.  Thus $\R_i^e=0$ when $k=1$ or $2$.
	
	Analogously, if $\ol{E}_i$ is an analytically irreducible component of the strict transform of $\ol{D}$, necessarily intersecting a unique exceptional component $\ol{E}_j$,  its contribution to the residue is
	\[
	\R_i^s:= \frac{1}{N_i\ol{\alpha}_j} .
	\]
	Proposition~\ref{prop:propertiesEdges} already provides an important consequence: any ratio at an arrow is either a local maximum or a global minimum in the graph $\Gamma_{\ol{\pi}}$. Moreover, if $s_0$  %=-\ol{\nu}/\ol{N}$
	is a candidate pole associated to exactly $r$ components of the strict transform $\ol{E}_1,\dots, \ol{E}_r$, then either $r=1$ and $\R_1^s>0$, or $\R_i^s<0$ for any $i=1,\ldots,r$. In particular, any candidate pole induced by a strict transform of $\ol{D}$ is actually a pole of $\Ztopp{(\CC^2/\mu_d,[0])}(\ol{D},-B_\rho; s)$ if no other (exceptional) component is contributing.
	
	\medskip
	
	Furthermore, considering the ordered tree structure of $\Gamma_{\ol{\pi}}$ and `mimicking' the same proofs as in~\cite[Theorems~4.2~and~4.3]{Veys95}, we obtain the following characterization of the poles of $\Ztopp{(\CC^2/\mu_d,[0])}(\ol{D},-B_\rho; s)$ in terms of the embedded $\QQ$-resolution $\ol{\pi}$, constructed in~\eqref{Eq:diagram_CC2G}. It is worth noticing that, when the action is trivial, we obtain precisely~\cite[Theorems 4.2~and~4.3]{Veys95}.
	
	\begin{theorem}\label{thm:VeysPolesZtop}
		Consider the embedded $\QQ$-resolution $\ol{\pi}$ of $(\CC^2/\mu_d,\ol{M})$, constructed in~\eqref{Eq:diagram_CC2G}. Fix $s_0\in\QQ_{<0}$. Then
		\begin{enumerate}[label=(\roman*)]
			\item $s_0$ is a pole of $\Ztopp{(\CC^2/\mu_d,[0])}(\ol{D},-B_\rho)$ if and only if $s_0=-\ol{\nu_i}/\ol{N_i}$ for some irreducible component $\ol{E}_i$ of the strict transform of $\ol{D}$, or for some exceptional curve $\ol{E}_i$ containing at least three points $P_j\in\P_{\ol{\pi}}$;
			
			\item $s_0$ is a pole of order $2$ if and only if there exists two intersecting components $\ol{E}_i$ and $\ol{E}_j$ with $s_0=-\ol{\nu_i}/\ol{N_i}=-\ol{\nu_j}/\ol{N_j}$, and in that case $s_0$ is the pole closest to the origin.
		\end{enumerate}
	\end{theorem}
	
	\begin{remark}
		The assertions in this theorem still hold for any pair $(D,W)$ on a normal surface singularity $(S,o)$ as in Setting~\ref{setting1}, provided that it admits an embedded $\QQ$-resolution $h$ verifying the $\alpha$-condition, since this is a direct consequence of Proposition~\ref{prop:orderedTree} together with some formal arguments.
	\end{remark}

	\subsection{Proof of Theorem~\ref{Thm:Comparison_Ztop}}
	\label{subsec:ProofPolesZtop}
	
	Now we are ready to conclude with a proof of Theorem~\ref{Thm:Comparison_Ztop}. We first treat the general case, being $\pi\neq\id$. 
	Similarly as in the proof of Theorem~\ref{Thm:Comparison_Zmot}, we easily see that  $s_0\in\QQ_{<0}$ is a pole of order $2$ of $\Ztopp{(\CC^2,0)}(D,0;s)$ if and only if it is a pole of order $2$ of $\Ztopp{(\CC^2/\mu_d,[0])}(\ol{D},-B_\rho; s)$.
	
	Suppose now that $s_0$ is a pole of order 1 of $\Ztopp{(\CC^2/\mu_d,[0])}(\ol{D},-B_\rho; s)$. By Theorem~\ref{thm:VeysPolesZtop}, we have that $s_0=-\ol{\nu}/\ol{N}$ for some irreducible component $\ol{E}$ of the strict transform of $\ol{D}$, or for some exceptional curve $\ol{E}$ containing at least three points $Q_i\in\P_{\ol{\pi}}$ (and thus with $\ol{\alpha}_i<1$). In the first case, $s_0$ is also induced by some irreducible component of the strict transform of $D$, and hence it is a pole of $\Ztopp{(\CC^2,0)}(D,0;s)$ by~\cite[Theorem~4.3]{Veys95}.
	So assume that $\ol{E}$ is exceptional and contains at least three points $Q_i\in\P_{\ol{\pi}}$. Take $E$ exceptional in $\pi$, projecting on $\ol{E}$. By Theorem~\ref{thm:AtMostTwoAlphas}, we have that $\#\set{P_j\in E \mid \alpha_j\neq1}\geq 3$, and more concretely there exist at least three points $P_j\in E$ such that $\alpha_j<1$ (since $\pi$ is the minimal embedded resolution). Thus, $s_0$ is a pole of $\Ztopp{(\CC^2,0)}(D,0;s)$, again by~\cite[Theorem~4.3]{Veys95}.
	
	Assume on the other hand that $s_0=-\nu/N$ is a pole of order 1 of $\Ztopp{(\CC^2,0)}(D,0;s)$. The strict transform case is as above; so assume that there exists an exceptional  $E$ with numerical data $(N,\nu)$, containing  at least three points $P_j$ with $\alpha_i\neq1$. Once again, this implies that $\alpha_i<1$ by construction. In particular, we are never in the geometric case described in Proposition~\ref{prop:losingPoles}(3)(b). Thus the cardinality of $\set{Q_i\in E\mid \ol{\alpha}_i\neq1}$ is at least~$3$, and we conclude that $s_0$ is a pole of $\Ztopp{(\CC^2/\mu_d,[0])}(\ol{D},-B_\rho; s)$ by Theorem~\ref{thm:VeysPolesZtop}. 
	This proves equality~\eqref{eq:equality-top}. 
	
	Now, by the geometrical characterization of poles of motivic zeta functions for embedded $\QQ$-resolutions, see Theorem~\ref{Thm:RodVeysQres}, we have that any $s_0$ is a simple pole of the zeta function $\Zmott{(\CC^2/\mu_d,[0])}(\ol{D},-B_\rho; s)$ (resp.~$\Zmott{(\CC^2,0)}(D,0;s)$) if and only if it comes from either the strict transform of $D$ or a rupture component of $\ol{\pi}$ (resp.~$\pi$). By Theorem~\ref{thm:VeysPolesZtop}, this is equivalent to $s_0$ being a pole of $\Ztopp{(\CC^2/\mu_d,[0])}(\ol{D},-B_\rho; s)$ (resp.~$\Ztopp{(\CC^2,0)}(D,0;s)$). %	
	This shows that, for $S = \CC^2/\mu_d$, $W = -B_\rho$ and $\rho: \CC^2\to\CC^2/\mu_d$ the natural covering, expression~\eqref{eq:top-mot} is in fact an equality.
	
	Finally, one needs to deal with the situation when $\pi=\id$. There are finitely many such cases; they were already discussed in Section~\ref{sec:BadCases}, where the associated topological zeta functions are explicitly computed, see Remark~\ref{rmk:BpNeverTangent}(i). This completes the proof of Theorem~\ref{Thm:Comparison_Ztop}.
	
	\medskip
	
	Summarizing, we have constructed a `nice' particular embedded $\QQ$-resolution of $\CC^2/\mu_d$, sharing several good properties with minimal embedded resolutions. However, as in the usual case, all these properties are also verified for the \emph{minimal} embedded $\QQ$-resolution.
	
	\begin{cor}\label{cor:minEmbQres}
		Let $h\colon Z\to \CC^2/\mu_d$ be the minimal embedded $\QQ$-resolution of $\supp(\ol{D})\cup\supp(B_\rho)$. Then, both Propositions~\ref{Prop:AlphasAbs1} and~\ref{prop:orderedTree}, as well as Theorem~\ref{thm:VeysPolesZtop}, hold for $h$.
	\end{cor}
	
	\begin{proof}
		The key point is to verify Proposition~\ref{Prop:AlphasAbs1} in this context. The rest of the arguments are the same. By minimality, the embedded $\QQ$-resolution $\ol{\pi}$ constructed in~\eqref{Eq:diagram_CC2G} factorizes through $h$ via a sequence of contractions of chains between vertices, corresponding to rational exceptional components. In fact, $\Gamma_h$ and also its minimal subgraph $\M$ has only chains of length 1. As in the proof of Theorem~\ref{Thm:RodVeysQres}, see Section~\ref{Sec:ProofProp}, if $P\in Z$ is a point on an exceptional component $\ol{E}_j$ of $h$ and $P$ appears after the contraction of a chain of rational curves $\ol{F}_1,\ldots,\ol{F}_r$ from $X/G$, then $\ol{F}_r$ has associated value $\alpha_{P}=\alpha_r$ with respect to $\ol{E}_j$ in $\ol{\pi}$. Thus, Proposition~\ref{Prop:AlphasAbs1} applies since the set of $\alpha$-values associated to $h$ is a subset of those of $\ol{\pi}$. In particular, $h$ verifies the $\alpha$-condition.
	\end{proof}

\section{Proof of Theorem~\ref{Thm:dZtop}}\label{Sec:Proportionality}

As both Examples~\ref{ex:noPoleDownstairs} and~\ref{ex:noPoleUpstairs} have shown, we cannot expect to fully determine topological zeta functions in quotients $\CC^2/\mu_d$ from those of $\CC^2$ by a simple expression as in~\eqref{eq:ztop-equality}. 
Nevertheless, Theorem~\ref{Thm:dZtop} shows that the only cases where we `lose control of the situation' are exactly those where the action permutes different branches of the curve, see Example~\ref{ex:noQNC}.

In this context, we complete the study of arithmetic relations between compatible $\QQ$-resolutions given in Section~\ref{ssec:Numerical_Relations}, by numerically relating orders of singular points over exceptional components in terms of the degree of the action and the ramification indices.

\begin{prop}\label{prop:relation_dme}
	Let $S=\CC^2$, $\ol{S}=\CC^2/\mu_d$,  and let $\rho\colon \CC^2 \twoheadrightarrow \CC^2/\mu_d$ be the natural covering, as in Setting~\ref{setting2}. We assume further the hypothesis and use the notation of Theorem~\ref{thm:MultAlphas_proportional}. Let $\ol{E}_1$ and $\ol{E}_2$ be two irreducible components of $\ol{\pi}^{-1}(\ol{M})$ in $\ol{X}$, intersecting transversely in a unique point $Q$, of order $\ol{m}$. Let $\wt{\rho}^{-1}(Q)=\{P_1,\dots,P_r\}$; these points all have the same order, say $m$, since they all belong to the same orbit. Also, each $P_k$ is the intersection of a curve $E_i^{(1)}$ with a curve  $E_j^{(2)}$, where $\wt{\rho}(E_i^{(1)})=\ol{E}_1$ and $\wt{\rho}(E_j^{(2)})=\ol{E}_2$.  
Again for orbit reasons, all $E_i^{(1)}$ (resp.~all $E_j^{(2)}$) have the same ramification index, say $e_1$ (resp.~$e_2$). 
 Then
\begin{equation}\label{formula:relation_dme}
	d m = re_1e_2\ol{m}.
\end{equation}
\end{prop}

\begin{proof}
Using intersection theory (see e.g.~\cite{Ful98}), one has
	\begin{align*}
		\frac{1}{\ol{m}} &= \ol{E}_1\cdot\ol{E}_2
 = \frac{1}{d}\left(\wt{\rho}^*\ol{E}_1\right)\cdot \left(\wt{\rho}^*\ol{E}_2\right) 
= \frac{1}{d}\sum_{k=1}^r e_1e_2 \frac 1m
		= \frac{1}{d}\cdot\frac{re_1e_2}{m},
	\end{align*}
	and the result holds.
\end{proof}

\begin{remark}
The result of Proposition \ref{prop:relation_dme} is still valid when $\ol{E}_1$ is an exceptional component in $\ol{\pi}^{-1}(\ol{M})$ and $\ol{E}_2$ is a \lq curvette\rq\ through $Q$, meaning that the  $E_j^{(2)}$ are curvettes through the points $P_k$.
\end{remark}

We are ready to provide a proof of Theorem~\ref{Thm:dZtop}.

\medskip

  Let $\pi\colon X\to\CC^2$ be the {minimal} embedded resolution of $M=\supp(D)\cup\supp(W)\cup\supp(\Ram_\rho)$ on $(\CC^2,0)$. 
  First, if $M$ is normal crossing in $\CC^2$, then the formula holds by the discussion in Section~\ref{sec:BadCases}, and so we can assume that $\pi\neq\id$. 
  By Remark~\ref{rmk:minres_also_works_forW}, we obtain again the same commutative diagram as in~\eqref{Eq:diagram_CC2G}, where now $\ol{\pi}\colon\ol{X}\to\CC^2/\mu_d$ is an embedded $\QQ$-resolution of $\ol{M}=\supp(\ol{D})\cup\supp(\ol{W})\cup\supp(B_\rho)$. 

  The resolution $\pi$ is obtained by a finite composition $\pi=\pi_n\circ\cdots\circ\pi_1$ of blow-ups, where $\pi_i\colon X_i\to X_{i-1}$ is the blow-up of a point $P_{i-1}\in X_{i-1}$, starting with $X_{0}=\CC^2$ and $P_0=0$, and finishing with $X=X_n$. At each step $\pi_i$, we can assume by minimality that $P_{i-1}$ belongs to one of the (analytically irreducible) components of the strict transform $\wh{M}$ in $X_{i-1}$, producing an exceptional component $E_i=\pi^{-1}_{i}(P_{i-1})\subset X_i$. Note that $P_0=0$ is a $\mu_d$-fixed point of $\CC^2$, and thus $E_1=\pi^{-1}_{1}(P_{0})$ is $\mu_d$-invariant in $X_1$ when extending the action via $\pi_1$. 
  By hypothesis, each analytically irreducible component of $M$ is $\mu_d$-invariant, and then any point in $\wh{M}\cap E_1$ is a $\mu_d$-fixed point of $X_1$. A simple induction shows that any $E_i$ is $\mu_d$-invariant in each $X_j$,  $j\geq i$. Moreover, we will use the fact that $\wt{\rho}^*\ol{E}_i=e_iE_i$, for some $e_i\geq 1$ dividing $d$. 

  In this way, we have bijections between families of components in $X$ and in $\ol{X}$, specifically:
  \begin{itemize}
    \item between irreducible components of the strict transforms of $D$ and $\ol{D}$ (resp.~$W$ and $\ol{W}$) by hypothesis, and also between those of $\Ram_\rho$ and $B_\rho$,
    \item between exceptional components of $\pi$ and $\ol{\pi}$.
  \end{itemize}
Let $\{E_i\}_{i\in I}$ (resp.~$\{\ol{E}_i\}_{i\in I'}$) be the irreducible components
of $\pi^{-1}(M)$ (resp.~$\ol{\pi}^{-1}(\ol{M})$); we can assume that $I=I'$ and $I_e=I'_e$. Moreover, we will order the components in such a way that $\wt{\rho}(E_i)=\ol{E}_i$. Denote by $(N_i,\nu_i)$ (resp.~$(\ol{N}_i,\ol{\nu}_i)$) the {numerical data} of $E_i$ (resp.~$\ol{E}_i$), for $i\in I$. 

In fact, there are only two types of exceptional components $\ol{E}_i$:
  \begin{enumerate}[label=(T\arabic*)]
    \item\label{proofInvariance_case1} either $e_i=d$, and then $\wt{\rho}|_{E_i}:E_i\to \ol{E}_i$ is a $1:1$ map,
    
    \item\label{proofInvariance_case2} or $e_i\neq d$, and then, by Proposition~\ref{prop:losingPoles}, $\ol{E}_i$ contains exactly two total ramification points $Q_1,Q_2$ of the $\frac{d}{e_i} : 1$ map $\wt{\rho}|_{E_i}$. Also, any other intersecting component $\ol{E}_j$ will intersect either in $Q_1$ or $Q_2$ (otherwise $\ol{E}_j$ would have at least two preimages by $\wt{\rho}$, which leads to a contradiction).
  \end{enumerate}
The singular points of $\ol{X}$ can be described in terms of the components of the second type; we claim the following.
\begin{enumerate}
  \item {\em  Any singular point of $\ol{X}$ belongs to an exceptional $\ol{E}_i$ of type \ref{proofInvariance_case2} and it is one of the two ramification points.}
  
  \item {\em  Any ramification point of a curve of type \ref{proofInvariance_case2} not belonging to any other $\ol{E}_j,j\in I$,  is a singular point of $\ol{X}$.}
\end{enumerate}

Indeed, take $Q\in \Sing \ol{X}$ of order $\ol{m} >1$.  If $Q$ belongs to some $\ol{E}_i$ of type \ref{proofInvariance_case1}, then we would obtain in (\ref{formula:relation_dme}) (with $\ol{E}_i=\ol{E}_1$), that $d=1\cdot d\cdot e_2\cdot \ol{m}$, contradicting that $\ol{m} >1$.  So $Q$ should belong to a component $\ol{E}_i$ of type \ref{proofInvariance_case2}.  If $Q$ is not one of the two  ramification points, then $\wt{\rho}^{-1}(Q)$ consists of $\frac{d}{e_i}$ points, and (\ref{formula:relation_dme}) (with $\ol{E}_i=\ol{E}_1$) says that $d=\frac{d}{e_i}\cdot e_i\cdot e_2\cdot \ol{m}$, again contradicting that $\ol{m} >1$.

On the other hand, let $Q_\ell$ be a ramification point of a component $\ol{E}_i$ of type \ref{proofInvariance_case2}. If $Q_\ell$ does not belong to any other $\ol{E}_j,j\in I$, then (\ref{formula:relation_dme}) (with $\ol{E}_i=\ol{E}_1$) yields $d=1\cdot e_i\cdot 1\cdot \ol{m}$, implying that $\ol{m} >1$.  This finishes the proof of the claim.

\medskip

To achieve the formula in the statement of Theorem~\ref{Thm:dZtop} we have to enlarge the index set conveniently by introducing a set of curvettes $\{E_\ell\}$ defined as follows. Consider an exceptional component $\ol{E}_j$ of type \ref{proofInvariance_case2} and let $Q_\ell$ be one of its ramification points.
Assume that $Q_\ell$ does not belong to any other $\ol{E}_j,j\in I$. We consider a $\mu_d$-invariant curvette  $E_\ell$, intersecting $E_i$ transversely at $P=\wt{\rho}^{-1}(Q_\ell)$. 
Since $\wt{\rho}\colon X\to \ol{X}$ is a non-ramified $d$-covering outside $\bigcup_{i\in I} E_i$, each curvette $E_\ell$ has ramification order $e_\ell=1$. 
As usual, its numerical data are $(0,1)$. Thus we extend the index set from $I$ to $J$, to include our curvettes; in other words, all these curvettes belong to $\{E_\ell\}_{\ell \in J\setminus I}$. Their images in $\ol{X}$ are denoted by $\ol{E}_\ell, \ell \in J\setminus I$.

Now, recall from~\eqref{Eq:P_pi} in Section~\ref{SubSec:ZQres} the notation $\P_{\ol{\pi}}= (\Sing \ol{X})   \cup \bigcup_{\set{i,j}\subset I} (\ol{E}_i\cap \ol{E}_j)$. If we set 
$\P'_{\pi}:= \wt{\rho}^{-1}\P_{\ol{\pi}}$, % which is a subset of $\P_{\pi}$, 
the previous discussions imply that $ \wt{\rho}$ induces a bijection between $\P_{\pi}'$ and $P_{\ol{\pi}}$. The zeta function on the source can be rewritten accordingly as
     \begin{align*}
\Ztopp{(\CC^2,0)}&(D,W; s)=
\sum_{i\in I_e} \frac{\chi(E_i^\circ)}{\nu_i + N_i s} + \sum_{\{i,j\}\subset I} \frac{\chi(E_i\cap E_j)}{(\nu_i+N_i s)(\nu_j + N_j s)}  \\
&= \sum_{i\in I_e} \frac{\chi(E_i^\square) + \#\set{E_\ell\text{ curvette}\mid E_\ell\cap E_i\neq\emptyset}}{\nu_i + N_i s} + \sum_{\{i,j\}\subset I} \frac{\chi(E_i\cap E_j)}{(\nu_i+N_i s)(\nu_j + N_j s)}  \\
&= \sum_{i\in I_e} \frac{\chi(E_i^\square)}{\nu_i + N_i s} + \sum_{P\in\P_{\pi}'} \frac{1}{(\nu_i+N_i s)(\nu_j + N_j s)} ,
     \end{align*} 
  where $E_i^\square:= E_i\setminus \P_{\pi}'$, and $(N_i,\nu_i), (N_j,\nu_j)$ are the numerical data of the two components $E_i$ and $E_j$ ($i,j\in J$) 
intersecting at the point $P$.

For the zeta function in the target, we use the notation for numerical data as in Section~\ref{SubSec:ZQres} to obtain 
  \[
\Ztopp{(\CC^2/\mu_d,[0])}(\ol{D},\ol{W}; s)= \sum_{i\in I_e} \frac{\chi(\ol{E}_i^\circ)}{\ol{\nu}_i + \ol{N}_i s} + \sum_{Q\in {\P}_{\ol{\pi}}} \frac{\ol{m}_Q}{(\ol{\nu}_{i}+\ol{N}_{i}s)(\ol{\nu}_{j}+\ol{N}_{j}s)},
  \]
  where
 $(\ol{N}_i,\ol{\nu}_i), (\ol{N}_j,\ol{\nu}_j)$ are the numerical data of the two components $\ol{E}_i$ and $\ol{E}_j$ ($i,j\in J$) 
intersecting at $Q$, which is a point of order $\ol{m}_Q$ on $\ol{X}$.

 Summarizing, our zeta functions are expressed in terms of sums over bijective sets of indices, in particular we have a bijection term by term. So, our final step is to check the equality of Theorem~\ref{Thm:dZtop} term by term.  First, if $Q\in {\P}_{\ol{\pi}}$, one has
  \begin{align*}
\frac{\ol{m}_Q}{(\ol{\nu}_{i}+\ol{N}_{i}s)(\ol{\nu}_{j}+\ol{N}_{j}s)}
   &= \frac{\ol{m}_Q}{\left(\frac{\nu_i}{e_i} + \frac{N_i}{e_i} s \right)\left(\frac{\nu_j}{e_j} + \frac{N_j}{e_j} s \right)}
          = \frac{e_ie_j\ol{m}_Q}{\left(\nu_i + N_i s\right)\left(\nu_j + N_j s\right)}\\
          &=d\cdot \frac{1}{\left(\nu_i + N_i s\right)\left(\nu_j + N_j s\right)} ,
  \end{align*}
  by Proposition~\ref{prop:relation_dme} and Remark \ref{formula:relation_dme}. 
  
 Finally, we compare the terms involving the Euler characteristic of the open part, distinguishing by the type of $\ol{E}_i$.
	\begin{itemize}
	  \item[\ref{proofInvariance_case1}] If $e_i=d$, then, since $E_i\to \ol{E}_i$ is a $1:1$ map,  $\ol{E}_i^\circ$ is isomorphic to $E_i^\square$, and hence they have the same Euler characteristic. So
	  \[
        \frac{\chi(\ol{E}_i^\circ)}{\ol{\nu}_i + \ol{N}_i s}= \frac{\chi(E_i^\square)}{\frac{\nu_i}{e_i} + \frac{N_i}{e_i} s}= d\cdot\frac{\chi(E_i^\square)}{\nu_i + N_i s}.
      \]
	  
	  \item[\ref{proofInvariance_case2}] If $e_i\neq d$,  then  $\P_{\ol{\pi}}\cap \ol{E}_i=\set{Q_1,Q_2}$. Hence  $  \chi(\ol{E}_i^\circ) = \chi(E_i^\square) = 0$,
	 	  and it follows that 
      \[
        \frac{\chi(\ol{E}_i^\circ)}{\ol{\nu}_i + \ol{N}_i s} = d\cdot \frac{\chi(E_i^\square)}{\nu_i + N_i s} = 0.
      \]
    \end{itemize}

	% =========================================================
	% Comparison of log canonical models
	% =========================================================
	\bigskip
	\section{Comparison of log canonical models}\label{sec:log_canonical}

An alternative approach to proving our main theorems could be via 
comparing the log canonical models of both source and target. Here we state some results in this direction that are of independent interest, especially in the context of Setting~\ref{setting2}. 

A log canonical model is typically associated to a pair consisting of a normal variety and a reduced divisor on it, or, more generally, to a {\em boundary}, which is  a $\QQ$-divisor with coefficients belonging to the interval $[0,1]$. It can be considered as some \lq partial log resolution\rq\ of the pair.

\begin{defi} 
 Let $V$ be a normal variety and $B$ a boundary on $V$ such that the divisor $K_V+B$ is $\QQ$-Cartier.
A log canonical model of the pair $(V,B)$ is a normal variety $X$, equipped with a projective birational morphism $\pi:X\to V$, such that, denoting by $E$ the reduced exceptional divisor of $\pi$,
\begin{enumerate}
\item 
 the pair $(X, \wh{B}+E)$ is log canonical, and
\item  the divisor $K_X+ \wh{B}+E$ is $\pi$-ample.
\end{enumerate}
\end{defi}

Such a model exists and is unique, see e.g.~\cite[Theorem 1.32]{Kollar13}.  In dimension $2$ the condition that $K_V+B$ is $\QQ$-Cartier is not necessary.

\medskip	

As in Section~\ref{ssec:Numerical_Relations}, we now fix a finite morphism $\rho\colon (S, o)\to (\ol{S}, \ol{o})$  between normal surfaces $S$ and $\ol{S}$, with associated ramification divisor $\Ram_\rho$ on $S$ and branch divisor $B_\rho$ on $\ol{S}$. Moreover, let $C$ (resp.~$\ol{C}$) be a boundary on $(S, o)$ (resp.~$(\ol{S},\ol{o})$).
So
\begin{align*}
\ol{C}&=\sum_{i\in I} \ol{c}_i \ol{C}_i  \quad \text { with } 0\leq \ol{c}_i \leq 1 \text{ for all } i  ; \\
 C&=\sum_{i\in I} \sum_{j =1}^{r_i}c_j^{(i)} C_j^{(i)}   \text { with } 0\leq c_j^{(i)} \leq 1 \text{ for all } i,j ;
\end{align*}
where all $\ol{C}_i$ and $C_j^{(i)}$ are prime divisors, and where we assume moreover that $\rho^{-1}\ol{C}_i = \cup_{j =1}^{r_i} C_j^{(i)}$ for all $i$.

\smallskip

\begin{lemma}
Assume that $\rho^*(K_{\ol{S}}+ \ol{C}) = K_S+C$.
Let $\rho^*\ol{C}_i= \sum_{j=1}^{r_i}  e_j^{(i)}  C_j^{(i)}$ for all $i$.
Then
$$1-c_j^{(i)} = e_j^{(i)} (1-\ol{c}_i), \quad \forall i,j.$$
\end{lemma}

\begin{proof}
This follows from Proposition \ref{thm:MultAlphas_proportional}(1), since $1-\ol{c}_i$ is the number $\ol{\nu}_i$ corresponding to $\ol{C}_i$, and similarly for $ C_j^{(i)}$.
\end{proof}

\medskip

\begin{cor}\label{corollary coefficients}\mbox{}
\begin{enumerate}
\item For a fixed $i\in I$, we have that $\ol{c}_i=1$ if and only if some $c_j^{(i)}=1$ if and only if   $c_j^{(i)}=1$ for all $j$.
\item We have that $\ol{c}_i=1-\frac{1}{e_j^{(i)}}$ if and only if $c_j^{(i)}=0$.
\end{enumerate}
\end{cor}

The following is an immediate consequence of~\cite[Sections~2.41 and 2.42]{Kollar13}.

\begin{lemma}\label{relation boundaries}
Consider a commutative diagram below, where $\wt{\rho}\colon X\to \ol{X}$ is also a finite morphism between normal surfaces, and $\pi\colon X\to S$ and  $\ol{\pi}\colon \ol{X}\to \ol{S}$ are proper and birational morphisms.
	  \begin{center}
		\begin{tikzcd}
			S \arrow[d, twoheadrightarrow, "\rho"]
			& \arrow[l, "\pi" above] X \arrow[d, twoheadrightarrow, "\wt{\rho}"]\\
			\overline{S} & \arrow[l, "\ol{\pi}" above] \ol{X}
		\end{tikzcd}
	  \end{center}
Let $\sum_{i\in I}  \ol{E}_i$ (resp.~$\sum_{i\in I} \sum_{j=1}^{r_i} E_j^{(i)}$)    be the reduced exceptional divisor of $\ol{\pi}$ (resp.~of $\pi$),
where all $\ol{E}_i$ and $E_j^{(i)}$ are prime divisors, and $\rho^{-1}\ol{E}_i = \cup_{j=1}^{r_i}  E_j^{(i)}$ for all $i$.

Assume that $\rho^*(K_{\ol{S}}+ \ol{C}) = K_S+C$.
Then 
$$\wt{\rho}^*\left(K_{\ol{X}}+ \wh{\ol{C}}+\sum_{i\in I}  \ol{E}_i \right)= K_S+\wh{C}+\sum_{i \in I}\sum_{j=1}^{r_i} E_j^{(i)}.$$
\end{lemma}

\begin{prop}\label{relation lc models}
In the setting of Lemma \ref{relation boundaries}, we have that 
$$\left(X, \wh{C}+\sum_{i\in I} \sum_{j=1}^{r_i} E_j^{(i)}\right) \text{ is the log canonical model of } (S,C)$$
if and only if
$$\left(\ol{X}, \wh{\ol{C}}+\sum_{i \in I} \ol{E}_i\right) \text{ is the log canonical model of } (\ol{S}, \ol{C}).$$
\end{prop}

\begin{proof}
From Lemma \ref{relation boundaries}, we derive that $K_X+ \wh{C}+\sum_{i\in I} \sum_{j=1}^{r_i} E_j^{(i)}$ is $\pi$-ample if and only if $K_{\ol{X}}+ \wh{\ol{C}}+\sum_{i \in I} \ol{E}_i$ is $\ol{\pi}$-ample.

And by~\cite[Corollary 2.43]{Kollar13}, we have that $(X, \wh{C}+\sum_{i\in I} \sum_{j=1}^{r_i} E_j^{(i)})$ is log canonical if and only if $(\ol{X}, \wh{\ol{C}}+\sum_{i \in I} \ol{E}_i )$ is log canonical.
\end{proof}

\medskip
We now put ourselves in the context of Setting~\ref{setting2} and Theorem \ref{Thm:Comparison_Ztop}, where we have that $W=0$ and  $\ol{W}=-B_\rho=(\frac1{e_1} -1)\ol{L}_1+(\frac1{e_2} -1)\ol{L}_2$. We will apply Proposition \ref{relation lc models}  to obtain the following 
comparison between the log canonical model of $(S,D_{\red})$ and the one of $\big(\ol{S}, \ol{D}_{\red}+\sum_{L_k\not\subset |D_{\red}|} (1-\frac1{e_k})\ol{L}_k\big)$, where 
$ |D_{\red}|$ stands for the support of $D_{\red}$. Note thus that, when $L_k (k=1,2)$ is not in the support of  $D_{\red}$ and $e_k>1$, the curve $\ol{L}_k$ is needed in the divisor on the target.

\begin{theorem}\label{lc comparison}
We have that
$$\left(X, \wh{D}_{\red}+\sum_{i \in I}\sum_{j=1}^{r_i} E_j^{(i)}\right) \text{ is the log canonical model of } (S,D_{\red})$$
if and only if
$$\left(\ol{X}, \wh{\ol{D}}_{\red}+\sum_{L_k\not\subset |D_{\red}|} \left(1-\frac1{e_k}\right)\wh{\ol{L}}_k +\sum_{i \in I} \ol{E}_i\right)$$
is the log canonical model of 
$$\left(\ol{S}, \ol{D}_{\red}+\sum_{L_k\not\subset |D_{\red}|} \left(1-\frac1{e_k}\right)\ol{L}_k\right).$$
\end{theorem}

\begin{proof}
We apply Proposition \ref{relation lc models} with $C=D_{\red}$.  We claim that then, in the setting of Lemma \ref{relation boundaries}, we have $\ol{C}= \ol{D}_{\red}+\sum_{L_k\not\subset |D_{\red}|} (1-\frac1{e_k})\ol{L}_k$.

Indeed, when $L_k\not\subset |D_{\red}|$, the coefficient of $\ol{L}_k$ in $\ol{C}$ is $1-\frac1{e_k}$, by Corollary \ref{corollary coefficients}(2).  When $L_k \subset |D_{\red}|$,
we have that $\ol{L}_k$ is a component of $\ol{D}$, and its coefficient  in $\ol{C}$ is $1$, by Corollary \ref{corollary coefficients}(1).
\end{proof}

\begin{remark}
The poles of $\Ztopp{(\CC^2,0)}(D,0;s)$ and  $\Zmott{(\CC^2,0)}(D,0;s)$ are precisely all candidate poles that appear in the log canonical model of $ (S,D_{\red})$, i.e., all $-1/N_i$ induced by components of $D$ and all $-\nu_i/N_i$ induced by the exceptional components \cite[Section 3.4]{Veys97}.  
So by Theorems \ref{Thm:Comparison_Zmot} and \ref{Thm:Comparison_Ztop} and the above theorem, we have that the analogous statement also holds for 	
$\Ztopp{(\CC^2/G,[0])}(\ol{D},-B_\rho)$ and $\Zmott{(\CC^2/G,[0])}(\ol{D},-B_\rho)$.
\end{remark}
	
\begin{remark}	
We consider the context of Section \ref{subsec:orbit-pathology}, taking for simplicity $N=1$. So $D=C_1+C_2$ is a normal crossing divisor, where the smooth irreducible $C_1$ and $C_2$ are permuted by the $\mu_d$-action. The pair $(\CC^2, D)$ is its own log canonical model and then, by Theorem \ref{lc comparison}, the same must be true for the pair $(\CC^2/\mu_d, \ol{D}+\sum_{k=1}^2 (1-\frac1{e_k})\ol{L}_k)$.

More precisely, when $d\equiv 0 \mod 4$ (resp.~$d\equiv 2 \mod 4$, with say $b$ even), these log canonical  pairs are $(\CC^2/\mu_d, \ol{D})$ (resp.~$(\CC^2/\mu_d, \ol{D}+\frac12\ol{L}_1$)). These pairs are not in $\QQ$-normal crossing; their minimal embedded $\QQ$-resolution and minimal embedded resolution was presented if Figures~\ref{fig:JorgeEx2}, \ref{fig:JorgeEx3A}, \ref{fig:JorgeExCase2} and~\ref{fig:JorgeEx4},  respectively. 
\end{remark}

	% =========================================================
	% An extended Rodrigues-Veys Residue Theorem
	% =========================================================
	\bigskip
	\section{An extended Residue Theorem}\label{sec:RodVeysThm}

    In this section, we present a proof of Theorem~\ref{Thm:RodVeysQres}.
	The geometric determination of the poles of $\Zmott{(S,o)}(D,W; s)$ attempts to characterize the components $E_i$, contributing effectively to a given candidate pole; there are several results for the case $\Zmott{(S,o)}(D,0;s)$, e.g.~\cite{Veys95,Veys99,RodriguesVeys03,Rod04,Rod05}. When $W\neq 0$, some extra care is required~\cite{Veys07,NemethiVeys12}. We will follow the approach of~\cite{RodriguesVeys03} to characterize geometrically the poles of $\Zmott{(S,o)}(D,W; s)$ for arbitrary $W$.
	
	We prove this geometrical characterization in terms of an embedded $\QQ$-resolution in two steps. First, we prove the following generalization of~\cite[Theorem 3.4]{RodriguesVeys03} when $W\neq0$ for usual embedded resolutions. In a later section, we extend it for $\QQ$-resolutions.

\begin{theorem}\label{Thm:RodVeysExt}
	Let $(S,o)$ and $(D,W)$ be as in Setting~\ref{setting1}. Fix an embedded resolution $\pi \colon X\to (S,o)$ of $M=\supp(D)\cup \supp(W)$ and let $s_0\in \QQ$. With the notation introduced before, we have that $s_0$ is a pole of $\Zmott{(S,o)}(D,W; s)$ if and only if
	\begin{enumerate}[label=(\roman*)]
		\item\label{RVPart1} $s_0=-\nu_i/N_i$ for some irreducible component $E_i$ in the set $\wh{D}\cap\wh{W}$ or in the set $\wh{D}\setminus\wh{W}$ (in which case $\nu_i=1$), or
		\item\label{RVPart2} $s_0=-\nu_i/N_i$ for some non-rational exceptional curve $E_i$, or
		\item\label{RVPart3} $s_0=-\nu_i/N_i$ for a cycle of rational exceptional curves $E_i$, or
		\item\label{RVPart4} $s_0=-\nu_i/N_i$ for some rational exceptional curve $E_i$, intersecting at least three times other components $E_j$ with associated value $\alpha_j=\nu_j-(\nu_i/N_i)N_j\neq 1$.
	\end{enumerate}
\end{theorem}

\subsection{Proof of Theorem~\ref{Thm:RodVeysExt}}

Recall first that the numbers $\alpha_j$ appearing in (\ref{Eq:ResExc}) satisfy~\eqref{Eq:AlphaArithm}. In particular, when $E_i$ is rational and $k=2$ (resp.~$k=1$), we have that $[E_i^\circ]=\LL-1$ and $\alpha_1+\alpha_2=0$ (resp.~$[E_i^\circ]=\LL$ and $\alpha_1=-1$), and then clearly $\mathcal{R}^e =0$.
	Hence, if $s_0\in\QQ$ is a pole, it must come from one of the situations described in~\ref{RVPart1}, \ref{RVPart2}, \ref{RVPart3} or \ref{RVPart4}.

Now assume on the other hand that $s_0\in\QQ$ is a candidate pole induced by one of these four situations.
Since, by Proposition \ref{RVmot}, a candidate pole of order $2$ is automatically a pole of order $2$, we can assume that $s_0$ is a candidate pole of order 1. Furthermore, note that the situation described in~\ref{RVPart3} induces a candidate pole of order 2, thus we only need to deal with the other cases.

Due, for instance, to the occurrence of zero divisors in $\KVarC$ and $\mathcal{A}$, it is often not easy to show that an element in these rings is nonzero. A standard method is to use a specialization morphism to a ring without zero divisors. As in \cite{RodriguesVeys03}, we will use Hodge polynomials to proceed this way.

\bigskip

The {\em Hodge(-Deligne) polynomial} of a quasi-projective variety $V$ is 
	\[H(V):=\sum_{p,q}e^{p,q}(V)u^pv^q\in \ZZ[u,v],\]
	where $e^{p,q}(V)=\sum_{i\geq 0}(-1)^ih^{p,q}(H_c^i(V,\CC))$ and $h^{p,q}(H_c^i(V,\CC))$ denotes the rank of the
	$(p,q)$-Hodge component of the $i$th cohomology group with compact support of $V$. The following explicit formulas for the Hodge polynomial are well known, see e.g.~\cite{RodriguesVeys03}. First, $H(E_i\cap E_j)=\chi(E_i\cap E_j)=\Card(E_i\cap E_j)$.  Second,
	\begin{equation}\label{Eq:HodgeOpenStratum}
		H(E_i^\circ)=
			uv-g_iu-g_iv+1-\Card\left(E_i\cap \left(\cup_{j\neq i}E_j\right) \right), % 
			\end{equation}
where $g_i$ is the genus of $E_i$.

The assignment $V\to H(V)$ above induces a ring homomorphism $H:\KVarC\to \ZZ[u,v]$, which further induces a ring homomorphism $H$ from $\mathcal{A}_T$ to the field of rational functions in (rational powers of) $u$ and $v$ over $\QQ$. 
Using \eqref{Eq:HodgeOpenStratum}, the images of the expressions (\ref{Eq:ResExc}) and (\ref{Eq:ResStri}) are
	\begin{equation}\label{Eq:ResExcHodge}
		R^e := H(\mathcal{R}^e)= \frac{(1-uv)}{N_i(uv)^{-\nu_i/N_i}}\left(uv-g_i u-g_i v+(1-k)+\sum_{j=1}^k \frac{uv-1}{(uv)^{\alpha_j}-1}\right).
	\end{equation}
and
	\begin{equation}\label{Eq:ResStriHodge}
		R^s:=H(\mathcal{R}^s)= \frac{(1-uv)}{N_i(uv)^{-\nu_i/N_i}}\cdot\frac{uv-1}{(uv)^{\alpha_j}-1},
	\end{equation}
respectively.

\bigskip	
Suppose now that $s_0\in\QQ$ is a candidate pole of order 1, and denote by $I(s_0)$ the subset of indices $i\in I$, corresponding to all curves $E_i$ in $\pi^{-1}(M)$ with numerical data $(N_i,\nu_i)$ such that $s_0=-\nu_i/N_i$. We use the partition $I(s_0)=I_e(s_0)\cup I_s(s_0)$ to denote the components of the exceptional locus and the strict transform, respectively. 
		Assume that a curve $E_i$, for $i\in I(s_0)$, intersects exactly $k_i$ times other components, say $E_1,\ldots,E_{k_i}$, with numerical data $\set{(N_j,\nu_j)}_{j=1}^{k_i}$. Since $s_0$ is not a double pole, $\nu_j/N_j\neq \nu_i/N_i$ for each $j\in\{1,\ldots,k_i\}$.

We will show in each case that the image by $H$ of the total residue of $s_0$ is nonzero, which then implies that this total residue itself is nonzero in $\mathcal{A}$.

 The contribution $H(\mathcal{R}^e)$ or $H(\mathcal{R}^s)$ of $E_i$ to the residue of $s_0$ is, according to \eqref{Eq:ResExcHodge} and \eqref{Eq:ResStriHodge}, given by
		\begin{align*}
			R^e_i &= \frac{(1-uv)}{N_i(uv)^{s_0}}\left(uv-g_i(u+v)+(1-k_i)+\sum_{j=1}^{k_i} \frac{uv-1}{(uv)^{\alpha_j}-1}\right),\quad\text{if $i\in I_e(s_0)$}, \\
			 R_i^s&= \frac{(1-uv)}{N_i(uv)^{s_0}}\left(\sum_{j=1}^{k_i} \frac{uv-1}{(uv)^{\alpha_j}-1}\right),\quad\text{if $i\in I_s(s_0)$}.
		\end{align*}
		Note that $k_i=1$ for $R_i^s$, but we will keep this notation in order to get more compact expressions in what follows. 
		Now, the whole residue of $s_0$ is given by $$R = \sum_{i\in I_e(s_0)}R^e_i+\sum_{i\in I_s(s_0)}R^s_i.$$
We will analyze that residue by studying the expression
		\[\T=\frac{(uv)^{s_0}}{1-uv} \left(\sum_{i\in I_e(s_0)}R^e_i+\sum_{i\in I_s(s_0)}R^s_i\right).\]
In order to deal with the different contributions, we split the previous expressions as follows. Put  $k_i=k_i^++k_i^-$, where $k_i^+$ denotes the number of intersections with components $E_j$ having $\alpha_j>0$, and $k_i^-$  the number of intersections with components $E_j$ having $\alpha_j<0$. Furthermore, assume that $\alpha_j>0$ precisely for $j\in\{1,\ldots,k_i^+\}$. We also use the convention $\beta_j=-\alpha_j$ when $\alpha_j<0$. For instance, using these notations, we may rewrite $R^e_i $ %
 as follows:
		\[
		\frac{(1-uv)}{N_i(uv)^{s_0}}\left(uv-g_i(u+v)+(1-k_i^-)+\sum_{j=1}^{k_i^+} \frac{uv-(uv)^{\alpha_j}}{(uv)^{\alpha_j}-1}+\sum_{j=k_i^++1}^{k_i} \frac{(uv)^{\beta_j}(uv-1)}{1-(uv)^{\beta_j}}\right).
		\]
		A similar formula is obtained for $R^s_i$ by removing the expression $uv-g_i(u+v)$. Recall that $(k_i^+,k_i^-)$ is either $(1,0)$ or $(0,1)$ for strict transforms.
		
		We first deal with case~\ref{RVPart2}. Suppose that at least one of the exceptional curves is non-rational, say the genus of $E_{i_0}$ is  $g_{i_0}>0$.
		Note then that
		\begin{equation*}
			\lim\limits_{u\to 0} \T =
			\sum_{i\in I_e(s_0)} \frac{1-k_i^--g_iv}{N_i}+\sum_{i\in I_s(s_0)} \frac{k_i^+}{N_i}.
		\end{equation*}
		Since $g_{i_0}>0$ and $g_i\geq 0$ for $i\in I_e(s_0)$ with $i\neq i_0$, the latter is a  linear polynomial with a negative coefficient in front of $v$.
		So for this case, the residue of $s_0$ is not zero and  $s_0$ is thus a pole of order one.
		
		We now treat  the remaining cases. %  \ref{RVPart4}. %
		If $\lim_{u\to 0}\T\neq 0$ we are done, so we assume that $\lim_{u\to 0}\T= 0$. In particular this implies that
		\begin{equation}\label{Eq:limTzero0}
			0=\sum_{i\in I_e(s_0)} \frac{1-k_i^-}{N_i}+\sum_{i\in I_s(s_0)} \frac{k_i^+}{N_i}.
		\end{equation}
		Finally, computing the following limit together with~\eqref{Eq:limTzero0}, we obtain
		\begin{equation*}
			\lim\limits_{uv\to\infty}\frac{1}{uv}\T=\sum_{i\in I_e(s_0)} \frac{1-k_i^-}{N_i}-\sum_{i\in I_s(s_0)} \frac{k_i^-}{N_i} = -\sum_{i\in I_s(s_0)} \frac{k_i}{N_i}.
		\end{equation*}
		Note that case~\ref{RVPart1} is equivalent to $\lim\limits_{uv\to\infty}\frac{1}{uv}\T\neq 0$.
		Otherwise, we are left with the case where there are at least three exceptional rational curves contributing to $s_0$ as candidate pole (i.e., case~\ref{RVPart4}), but no strict transform is contributing.
		The proof in this case is already provided in \cite[Theorem 3.4]{RodriguesVeys03}.
	
\begin{remark} Let $(S,o)$ and $(D,W)$ be as in Setting~\ref{setting1}. As an \lq intermediate\rq\ between the motivic and topological zeta function, the \emph{Hodge zeta function} associated to the pair $(D,W)$ on $(S,o)$ is given by
	\begin{multline}\label{Def:HodZeta}
		\ZHod_{,(S,o)}(D,W;s):= \sum_{i\in I_e} H(E_i^\circ) \frac{uv-1}{(uv)^{\nu_{i}+N_{i}s}-1}\\
		+\sum_{\{i,j\}\subseteq I} H(E_i\cap E_j)\frac{(uv-1)^2}{((uv)^{\nu_{i}+N_{i}s}-1)((uv)^{\nu_{j}+N_{j}s}-1)}.
	\end{multline}
The arguments above yield the complete analogue of Theorem~\ref{Thm:RodVeysExt}, generalizing also \cite[Theorem 3.4]{RodriguesVeys03} to arbitrary $W$.
\end{remark}
	
\subsection{Proof of Theorem~\ref{Thm:RodVeysQres}}\label{Sec:ProofProp}

It suffices to prove that the geometric situation described in case \ref{ResPart4} of Theorem~\ref{Thm:RodVeysQres} is equivalent to case \ref{RVPart4} of Theorem~\ref{Thm:RodVeysExt}, once we transform an embedded $\QQ$-resolution into a usual embedded resolution by resolving the remaining quotient singularities. %

Assume that $\pi \colon X\to (S,o)$ is an embedded $\QQ$-resolution of $M$. We can produce an actual embedded resolution $\pi'\colon X'\to (S,o)$ of $M$ by resolving the finite set of quotient singular points by a finite sequence of blow-ups $h\colon X'\to\cdots\to X$. For each $P\in\Sing X$, $h^{-1}(P)$ is a chain of finitely many rational curves $F_1,\ldots,F_r$, where $r>0$ depends on  $P$.  
Thus, cases (i), (ii) and (iii) of Theorem~\ref{Thm:RodVeysQres} for $\pi$ lift automatically into cases \ref{RVPart1}, \ref{RVPart2} and \ref{RVPart3} of Theorem~\ref{Thm:RodVeysExt} for $\pi'$, respectively.

Now let $E_i$ be a rational exceptional component with numerical data $(N_i,\nu_i)$, and fix a point $P_j\in\P_\pi\cap E_i$. 
Denote by $\wt{E}_i$ the strict transform of $E_i$ by $h$. 
Recall that $P_j$ is either a (possibly singular) intersection point with another irreducible component, say $E_j$, or it belongs to $\Sing X\cap \left(E_i\setminus\bigcup_{\ell\in I, \ell\neq i} E_\ell\right)$. 
Assume that $P_j\in E_i \cap \Sing X$ with $\alpha$-value $\alpha_j= \frac{\nu_j-(\nu_i/N_i)N_j}{m_j}$ with respect to $E_i$, where $(N_j,\nu_j)$ are the numerical data of $E_j$ if $P_j$ is the intersection point between $E_i$ and $E_j$, and $(N_j,\nu_j)=(0,1)$ otherwise. Let $F_1,\ldots,F_r$ be the chain of rational exceptional curves in $X'$, associated to $P_j$, such that $F_r$ intersects $\wt{E}_i$ transversely. 
The numerical data $(N_r,\nu_r)$ of $F_r$ can be computed using~\cite[Lemma~2.4]{Veys97}, namely 
\[
  \nu_r + s N_r = \frac{1}{\Delta_{1,r}}\big((\nu_j+sN_j) + \Delta_{1,r-1}(\nu_{i} + sN_{i})\big),
\]
where $\Delta_{k,\ell}$ is the absolute value of the determinant of the intersection matrix $\left(-F_t\cdot F_s\right)_{t,s=k}^{\ell}$. %
It is well known known that $\Delta_{1,r}$ equals the order $m_j$ of $P_j$ in $X$. 
Computing the $\alpha$-value of $F_r$ with respect to $\wt{E}_i$, we  thus have that
\[
  \alpha_r = \nu_r - \frac{\nu_i}{N_i}N_r = \frac{\nu_j-(\nu_i/N_i)N_j}{m_j}=\alpha_j.
\]
Then the condition $\alpha_j\neq1$ for $P_j\in \P_\pi\cap E_i$ in $X$ is equivalent to $\alpha_r\neq 1$ for the intersection point $\wt{E}_i\cap F_r$ in $X'$. 
Since points $P_j\in\P_\pi\cap E_i$ are in bijection with points on $\wt{E}_i$ that intersect other components, this proves the equivalence between Theorem~\ref{Thm:RodVeysQres}\ref{ResPart4} for $\pi$ and Theorem~\ref{Thm:RodVeysExt}\ref{RVPart4} for $\pi'$.

Finally, since the exceptional sets of $\pi$ and $\pi'$ only differ by finite disjoint unions of chains of rational components, then, by Theorem~\ref{Thm:RodVeysExt} applied to $\pi'$, we have covered an exhaustive list of cases in Theorem~\ref{Thm:RodVeysQres} where a true pole of $\Zmott{(S,o)}(D,W; s)$ is provided.
	
% =========================================================
% References
% =========================================================
\providecommand{\bysame}{\leavevmode\hbox to3em{\hrulefill}\thinspace}
\providecommand{\MR}{\relax\ifhmode\unskip\space\fi MR }
\providecommand{\MRhref}[2]{%
	\href{http://www.ams.org/mathscinet-getitem?mr=#1}{#2}
}
\providecommand{\href}[2]{#2}


\begin{thebibliography}{10}
	
	\bibitem{ACLM05}
	E.~{Artal Bartolo}, P.~{Cassou-Nogu\`es}, I.~Luengo, and A.~{Melle
		Hern\'andez}, \emph{Quasi-ordinary power series and their zeta functions},
	Mem. Amer. Math. Soc. \textbf{178} (2005), no.~841, vi+85.
	
	\bibitem{AMO14A}
	E.~{Artal Bartolo}, J.~{Mart\'{\i}n-Morales}, and J.~{Ortigas-Galindo},
	\emph{Cartier and {W}eil divisors on varieties with quotient singularities},
	Internat. J. Math. \textbf{25} (2014), no.~11, 1450100, 20. 
	
	\bibitem{AMO14B}
	\bysame, \emph{Intersection theory on abelian-quotient {$V$}-surfaces and {$\bf
			Q$}-resolutions}, J. Singul. \textbf{8} (2014), 11--30. 
		
	\bibitem{CompactComplexSurfaces:book}
	W.~P. Barth, K.~Hulek, C.~A.~M. Peters, and A~{Van de Ven}, \emph{Compact
		complex surfaces}, second ed., Ergebnisse der Mathematik und ihrer
	Grenzgebiete. 3. Folge. A Series of Modern Surveys in Mathematics [Results in
	Mathematics and Related Areas. 3rd Series. A Series of Modern Surveys in
	Mathematics], vol.~4, Springer-Verlag, Berlin, 2004. 
	
	\bibitem{Bor18}
	L.~A. Borisov, \emph{The class of the affine line is a zero divisor in the
		{G}rothendieck ring}, J. Algebraic Geom. \textbf{27} (2018), no.~2, 203--209.
	
	
	\bibitem{CNS18}
	A.~Chambert-Loir, J.~Nicaise and J.~Sebag, {\em Motivic integration}, Progress in Mathematics, 325, Birkh\"auser/Springer, New York, 2018.
	
	
	\bibitem{DL98}
	J.~Denef and F.~Loeser, \emph{Motivic {I}gusa zeta functions}, J. Algebraic
	Geom. \textbf{7} (1998), no.~3, 505--537.
	
	\bibitem{DL99}
	\bysame, \emph{Germs of arcs on singular algebraic varieties and motivic
		integration}, Invent. Math. \textbf{135} (1999), no.~1, 201--232.
	
	\bibitem{Ful93}
	W.~Fulton, \emph{Introduction to toric varieties}, Annals of Mathematics
	Studies, vol. 131, Princeton University Press, Princeton, NJ, 1993, The
	William H. Roever Lectures in Geometry. 
	
	\bibitem{Ful98}
	\bysame, \emph{Intersection theory}, second ed., Ergebnisse der Mathematik und
	ihrer Grenzgebiete. 3. Folge. A Series of Modern Surveys in Mathematics
	[Results in Mathematics and Related Areas. 3rd Series. A Series of Modern
	Surveys in Mathematics], vol.~2, Springer-Verlag, Berlin, 1998. 
	
	\bibitem{Har77}
	R.~Hartshorne, \emph{Algebraic geometry}, Graduate Texts in Mathematics, vol.
	No. 52, Springer-Verlag, New York-Heidelberg, 1977.
	
	\bibitem{HNK71}
	F.~Hirzebruch, W.~D. Neumann, and S.~S. Koh, \emph{Differentiable manifolds and
		quadratic forms}, Lecture Notes in Pure and Applied Mathematics, vol. Vol. 4,
	Marcel Dekker, Inc., New York, 1971, Appendix II by W. Scharlau. 
	
	\bibitem{Igu74}
	J.~Igusa, \emph{Complex powers and asymptotic expansions. {I}. {F}unctions of
		certain types}, J. Reine Angew. Math. \textbf{268/269} (1974), 110--130,
	Collection of articles dedicated to Helmut Hasse on his seventy-fifth
	birthday, II.
	
	\bibitem{IshiiKollar03}
	S.~Ishii and J.~Kollár, \emph{The {N}ash problem on arc families of
		singularities}, Duke Math. J. \textbf{120} (2003), no.~3, 601–620.
	
	
	\bibitem{Ju08}
	H.~W.~E. Jung, \emph{Darstellung der {F}unktionen eines algebraischen
		{K}\"orpers zweier unabh\"angigen {V}er\"anderlichen {$x,y$} in der
		{U}mgebung einer {S}telle {$x=a,\ y=b$}}, J. Reine Angew. Math. \textbf{133}
	(1908), 289--314. 
	
	\bibitem{Kollar13}
	J.~Koll\'{a}r, \emph{Singularities of the minimal model program}, Cambridge
	Tracts in Mathematics, vol. 200, Cambridge University Press, Cambridge, 2013,
	With a collaboration of S\'{a}ndor Kov\'{a}cs. 
	
	\bibitem{Kol92}
	J.~{Koll\'{a}r (ed.)}, \emph{Flips and abundance for algebraic threefolds},
	Soci\'et\'e{} Math\'ematique de France, Paris, 1992, Papers from the Second
	Summer Seminar on Algebraic Geometry held at the University of Utah, Salt
	Lake City, Utah, August 1991, Ast\'erisque No. 211 (1992). 
	
	\bibitem{LCMMVVS:formulas}
	E.~{Le\'{o}n-Cardenal}, J.~{Mart\'{\i}n-Morales}, W.~Veys, and J.~{Viu-Sos},
	\emph{Motivic zeta functions on {$\mathbb{Q}$}-{G}orenstein varieties}, Adv.
	Math. \textbf{370} (2020), 107192, 34. 
	
	\bibitem{Loeser88}
	F.~Loeser, \emph{Fonctions d'{I}gusa {$p$}-adiques et polyn\^omes de
		{B}ernstein}, Amer. J. Math. \textbf{110} (1988), no.~1, 1--21.
	
	\bibitem{NemethiVeys12}
	A.~N\'emethi and W.~Veys, \emph{Generalized monodromy conjecture in dimension
		two}, Geom. Topol. \textbf{16} (2012), no.~1, 155--217.
	
	\bibitem{Nemethi22:book}
	A.~{Némethi}, \emph{Normal surface singularities}, Ergebnisse der Mathematik
	und ihrer Grenzgebiete. 3. Folge. A Series of Modern Surveys in Mathematics
	[Results in Mathematics and Related Areas. 3rd Series. A Series of Modern
	Surveys in Mathematics], vol.~74, Springer, Cham, [2022] \copyright 2022.
	
	
	\bibitem{Po02}
	B.~Poonen, \emph{The {G}rothendieck ring of varieties is not a domain}, Math.
	Res. Lett. \textbf{9} (2002), no.~4, 493--497.
	
	\bibitem{Pop11}
	P.~Popescu-Pampu, \emph{Introduction to {J}ung's method of resolution of
		singularities}, Topology of algebraic varieties and singularities, Contemp.
	Math., vol. 538, Amer. Math. Soc., Providence, RI, 2011, pp.~401--432.
	
	
	\bibitem{Rod04}
	B.~Rodrigues, \emph{Geometric determination of the poles of highest and second
		highest order of {H}odge and motivic zeta functions}, Nagoya Math. J.
	\textbf{176} (2004), 1--18. 
	
	\bibitem{Rod05}
	\bysame, \emph{On the geometric determination of the poles of {H}odge and
		motivic zeta functions}, J. Reine Angew. Math. \textbf{578} (2005), 129--146.
	
	
	\bibitem{RodriguesVeys03}
	B.~Rodrigues and W.~Veys, \emph{Poles of zeta functions on normal surfaces},
	Proc. London Math. Soc. (3) \textbf{87} (2003), no.~1, 164--196.
	
	\bibitem{Satake}
	I.~Satake, \emph{On a generalization of the notion of manifold}, Proc. Nat.
	Acad. Sci. U.S.A. \textbf{42} (1956), 359--363. 
	
	\bibitem{Steenbrink77}
	J.~H.~M. Steenbrink, \emph{Mixed {H}odge structure on the vanishing
		cohomology}, Real and complex singularities ({P}roc. {N}inth {N}ordic
	{S}ummer {S}chool/{NAVF} {S}ympos. {M}ath., {O}slo, 1976), Sijthoff and
	Noordhoff, Alphen aan den Rijn, 1977, pp.~525--563.
	
	\bibitem{Veys95}
	W.~Veys, \emph{Determination of the poles of the topological zeta function for
		curves}, Manuscripta Math. \textbf{87} (1995), no.~4, 435--448.
	
	\bibitem{Veys97}
	\bysame, \emph{Zeta functions for curves and log canonical models}, Proc.
	London Math. Soc. (3) \textbf{74} (1997), no.~2, 360--378.
	
	\bibitem{Veys99}
	\bysame, \emph{The topological zeta function associated to a function on a
		normal surface germ}, Topology \textbf{38} (1999), no.~2, 439--456.

	
	\bibitem{Veys07}
	\bysame, \emph{Monodromy eigenvalues and zeta functions with differential
		forms}, Adv. Math. \textbf{213} (2007), no.~1, 341--357. 
	
\end{thebibliography}
\end{document}